\documentclass[a4paper,10pt,bibtotoc,tocindent,fleqn]{article}

\usepackage[english]{babel}
\usepackage{amsmath}
\usepackage{amsthm}
\usepackage{amssymb}
\usepackage{amsfonts}
\usepackage{amsxtra}
\usepackage{color}
\usepackage[breaklinks]{hyperref}
\usepackage[all]{xy}
\usepackage{float}
\usepackage{rotfloat}
\usepackage{paralist}

\newcommand{\td}{\,\mathrm{d}}

\newcommand{\const}{\operatorname{const}}

\newcommand{\Ad}{\operatorname{Ad}}
\newcommand{\ad}{\operatorname{ad}}
\renewcommand\Re{\operatorname{Re}}
\renewcommand\Im{\operatorname{Im}}

\newcommand{\id}{\operatorname{id}}
\DeclareMathOperator{\tr}{tr}
\DeclareMathOperator{\rk}{rk}
\DeclareMathOperator{\Tr}{Tr}
\newcommand{\Sym}{\operatorname{Sym}}

\newcommand{\Herm}{\operatorname{Herm}}
\newcommand{\Str}{\operatorname{Str}}
\newcommand{\str}{\mathfrak{str}}
\newcommand{\Aut}{\operatorname{Aut}}

\newcommand{\der}{\mathfrak{der}}
\newcommand{\Co}{\operatorname{Co}}
\newcommand{\co}{\mathfrak{co}}
\newcommand{\GL}{\operatorname{GL}}
\newcommand{\gl}{\mathfrak{gl}}

\renewcommand{\sl}{\mathfrak{sl}}

\renewcommand{\sp}{\mathfrak{sp}}

\newcommand{\so}{\mathfrak{so}}

\newcommand{\su}{\mathfrak{su}}

\newcommand{\End}{\operatorname{End}}
\newcommand{\RR}{\mathbb{R}}
\newcommand{\KK}{\mathbb{K}}
\newcommand{\CC}{\mathbb{C}}
\newcommand{\ZZ}{\mathbb{Z}}
\newcommand{\NN}{\mathbb{N}}
\newcommand{\QQ}{\mathbb{Q}}

\newcommand{\HH}{\mathbb{H}}
\newcommand{\OO}{\mathbb{O}}
\newcommand{\DD}{\mathbb{D}}
\newcommand{\XX}{\mathbb{X}}
\newcommand{\BB}{\mathbb{B}}

\newcommand{\TT}{\mathbb{T}}

\renewcommand{\1}{\mathbf{1}}

\newcommand{\calA}{\mathcal{A}}

\newcommand{\calF}{\mathcal{F}}
\newcommand{\calO}{\mathcal{O}}
\newcommand{\calS}{\mathcal{S}}
\newcommand{\calW}{\mathcal{W}}
\newcommand{\calB}{\mathcal{B}}
\newcommand{\calH}{\mathcal{H}}
\newcommand{\calP}{\mathcal{P}}
\newcommand{\calU}{\mathcal{U}}
\newcommand{\calL}{\mathcal{L}}
\newcommand{\calX}{\mathcal{X}}

\newcommand{\calD}{\mathcal{D}}
\newcommand{\calT}{\mathcal{T}}
\newcommand{\calV}{\mathcal{V}}

\newcommand{\calR}{\mathcal{R}}
\newcommand{\calI}{\mathcal{I}}
\newcommand{\calJ}{\mathcal{J}}
\newcommand{\calK}{\mathcal{K}}

\newcommand{\frakg}{\mathfrak{g}}
\newcommand{\fraku}{\mathfrak{u}}
\newcommand{\frakk}{\mathfrak{k}}
\newcommand{\frakp}{\mathfrak{p}}

\newcommand{\frakn}{\mathfrak{n}}

\newcommand{\frakl}{\mathfrak{l}}

\newcommand{\frakh}{\mathfrak{h}}

\DeclareMathOperator{\Det}{Det}

\newcommand{\Int}{\operatorname{Int}}

\newcommand{\Ann}{\operatorname{Ann}}
\newcommand{\Ass}{\operatorname{Ass}}
\newcommand{\gr}{\operatorname{gr}}

\DeclareMathOperator*{\bigoplushat}{\widehat{\bigoplus}}
\newcommand{\disc}{\textup{disc}}
\newcommand{\cont}{\textup{cont}}

\theoremstyle{plain}
\newtheorem{theorem}{Theorem}[section]
\newtheorem{proposition}[theorem]{Proposition}
\newtheorem{lemma}[theorem]{Lemma}
\newtheorem{corollary}[theorem]{Corollary}
\newtheorem{conjecture}[theorem]{Conjecture}

\newtheorem{thmalph}{Theorem}

\theoremstyle{definition}
\newtheorem{definition}[theorem]{Definition}
\newtheorem{example}[theorem]{Example}

\newtheorem{remark}[theorem]{Remark}

\theoremstyle{remark}
\renewenvironment{pf}{\begin{proof}[\sc Proof]}{\end{proof}}

\numberwithin{equation}{section}

\begin{document}

\title{A geometric quantization of the Kostant--Sekiguchi correspondence for scalar type unitary highest weight representations}
\author{Jan M\"ollers}
\date{June 18, 2013}
\maketitle
\begin{abstract}

For any Hermitian Lie group $G$ of tube type we give a geometric quantization procedure of certain $K_\CC$-orbits in $\frakp_\CC^*$ to obtain all scalar type highest weight representations. Here $K_\CC$ is the complexification of a maximal compact subgroup $K\subseteq G$ with corresponding Cartan decomposition $\frakg=\frakk+\frakp$ of the Lie algebra of $G$. We explicitly realize every such representation $\pi$ on a Fock space consisting of square integrable holomorphic functions on its associated variety $\Ass(\pi)\subseteq\frakp_\CC^*$.

The associated variety $\Ass(\pi)$ is the closure of a single nilpotent $K_\CC$-orbit $\calO^{K_\CC}\subseteq\frakp_\CC^*$ which corresponds by the Kostant--Sekiguchi correspondence to a nilpotent coadjoint $G$-orbit $\calO^G\subseteq\frakg^*$. The known Schr\"odinger model of $\pi$ is a realization on $L^2(\calO)$, where $\calO\subseteq\calO^G$ is a Lagrangian submanifold. We construct an intertwining operator from the Schr\"odinger model to the new Fock model, the generalized Segal--Bargmann transform, which gives a geometric quantization of the Kostant--Sekiguchi correspondence (a notion invented by Hilgert, Kobayashi, \O rsted and the author).

The main tool in our construction are multivariable $I$- and $K$-Bessel functions on Jordan algebras which appear in the measure of $\calO^{K_\CC}$, as reproducing kernel of the Fock space and as integral kernel of the Segal--Bargmann transform. As a corollary to our construction we also obtain the integral kernel of the unitary inversion operator in the Schr\"odinger model in terms of a multivariable $J$-Bessel function as well as explicit Whittaker vectors.\\

\\


\end{abstract}

\newpage

\tableofcontents

\newpage

\addcontentsline{toc}{section}{Introduction}
\section*{Introduction}

The unitary highest weight representations of a semisimple Lie group $G$ were classified by Enright--Howe--Wallach \cite{EHW83} and Jakobsen \cite{Jak83} independently (see also \cite{EJ90}). Among them are the scalar type representations which were first completely described by Berezin \cite{Ber75} for the classical groups and by Wallach \cite{Wal79} in the general case and are parameterized by the so-called Wallach set (sometimes also referred to as the Berezin--Wallach set). There are various different realizations of these representations, e.g. on spaces of holomorphic functions on a bounded symmetric domain. Rossi--Vergne \cite{RV76} gave a realization which can by the work of Hilgert--Neeb--\O rsted \cite{HNO94,HNO96a,HNO96b} be understood as the geometric quantization of a nilpotent coadjoint $G$-orbit. More precisely, their model lives on the space $L^2(\calO)$ of square integrable functions on a Lagrangian submanifold $\calO\subseteq\calO^G$ of a nilpotent coadjoint $G$-orbit $\calO^G\subseteq\frakg^*$ in the dual of the Lie algebra $\frakg$ of $G$. We refer to this realization as the Schr\"odinger model.

The Kostant--Sekiguchi correspondence asserts a bijection between the set of nilpotent coadjoint $G$-orbits in $\frakg^*$ and the set of nilpotent $K_\CC$-orbits in $\frakp_\CC^*$, where $K_\CC$ is the complexification of a maximal compact subgroup $K\subseteq G$ of $G$ with corresponding Cartan decomposition $\frakg=\frakk+\frakp$. Under this bijection $\calO^G$ corresponds to a nilpotent $K_\CC$-orbit $\calO^{K_\CC}$ in $\frakp_\CC^*$. Harris--Jakobsen \cite{HJ82} obtained a realization of each scalar type unitary highest weight representation of $\frakg$ on regular functions on the corresponding associated variety in $\frakp_\CC^*$ (see also \cite{Jos92}). However, their construction is not explicit and of purely algebraic nature. They neither give a geometric construction of the invariant inner product nor do they analytically describe the representation space for the group action.

However, for the special case of the even part of the metaplectic representation, which is a highest weight representation of scalar type of $\frakg=\sp(k,\RR)$, a geometric quantization of the corresponding $K_\CC$-orbit is well-known. It gives the even part $\calF_{\textup{even}}(\CC^k)$ of the classical Fock space, consisting of even holomorphic functions on $\CC^k$ which are square integrable with respect to the Gaussian measure $e^{-|z|^2}\td z$. In this case also the intertwining operator between the Schr\"odinger model on $L^2_{\textup{even}}(\RR^k)$, the space of even $L^2$-functions on $\RR^k$, and the Fock model $\calF_{\textup{even}}(\CC^k)$ is explicit. It is given by the classical Segal--Bargmann transform which is the unitary isomorphism $\BB:L^2_{\textup{even}}(\RR^k) \stackrel{\sim}{\to}\calF_{\textup{even}}(\CC^k)$ given by
\begin{align}
 \BB\psi(z) &= e^{-\frac{1}{2}z^2}\int_{\RR^k}{e^{2z\cdot x}e^{-x^2}\psi(x)\td x}.\label{eq:IntroClassicalBargmann}
\end{align}

For Hermitian groups of tube type, we construct in a completely analytic and geometric way analogues of the Fock space and the Segal--Bargmann transform for any unitary highest weight representation of scalar type. The generalized Fock space consists of holomorphic functions on $\calO^{K_\CC}$ which are square integrable with respect to a certain measure. This establishes a geometric quantization of the nilpotent $K_\CC$-orbit $\calO^{K_\CC}$. We further find explicitly the intertwining operator between this new Fock model and the Schr\"odinger model, the generalized Segal--Bargmann transform. This integral transform can be understood as a \textit{geometric quantization of the Kostant--Sekiguchi correspondence}, a notion invented in \cite{HKMO12}. We remark that the special case of the minimal nilpotent $K_\CC$-orbit was treated in \cite{HKMO12} which was in fact the starting point of our investigations.

Although some of our proofs are obvious generalizations of the proofs in \cite{HKMO12} for the rank $1$ case (the minimal scalar type highest weight representation), the general case requires new techniques. Whereas in \cite{HKMO12} it was possible to work with classical one-variable Bessel functions, we have to use multivariable $J$-, $I$- and $K$-Bessel functions on Jordan algebras in the general case. These were studied before in \cite{Cle88,Dib90,DG93,DGKR00,FT87}, partly in a different context. We systematically investigate them further, proving asymptotic expansions, growth estimates, invariance properties, differential equations, integral formulas and their restrictions to the $K_\CC$-orbits $\calO^{K_\CC}$. We then use the $K$-Bessel functions in the construction of the measures on the orbits $\calO^{K_\CC}$ and the $I$-Bessel functions as reproducing kernels of the Fock spaces and as integral kernels of the Segal--Bargmann transforms. Additional results not studied in \cite{HKMO12} involve a detailed analysis of corresponding $K_\CC$-orbits and rings of differential operators on them (Section \ref{sec:RingsOfDiffOps}), the intertwining operator to the bounded symmetric domain model (Section \ref{sec:IntertwinerBdSymDomainModel}), Whittaker vectors in the Schr\"odinger model (Section \ref{sec:WhittakerVectors}) and applications to branching laws (Section \ref{sec:ApplicationBranching}). We further believe that our more general construction extends the known realizations of unitary highest weight representations in a natural and organic way.\\

We now explain our results in more detail. Let $G$ be a simple connected Hermitian Lie group of tube type with trivial center. Then $G$ occurs as the identity component of the conformal group of a simple Euclidean Jordan algebra $V$. Its Lie algebra $\frakg$ acts on $V$ by vector fields up to order $2$. More precisely, we have a decomposition
\begin{align}
 \frakg &= \frakn+\frakl+\overline{\frakn},\label{eq:IntroGelfandNaimark}
\end{align}
where $\frakn\cong V$ acts on $V$ by translations, $\frakl=\str(V)$ is the structure algebra of $V$ acting by linear transformations and $\overline{\frakn}$ acts by quadratic vector fields. There is a natural Cartan involution $\vartheta$ on $G$ given in terms of the inversion in the Jordan algebra. Denote by $K=G^\vartheta$ the corresponding maximal compact subgroup of $G$ and by $\frakg=\frakk+\frakp$ the associated Cartan decomposition. Then $\frakk$ has a one-dimensional center $Z(\frakk)=\RR Z_0$. Let $K_\CC$ be a complexification of $K$, acting on $\frakp_\CC$. Then $\frakp_\CC$ decomposes into two irreducible $K_\CC$-modules $\frakp^+$ and $\frakp^-$ and $Z_0\in Z(\frakk)$ can be chosen such that $\frakp^\pm$ is the $\pm i$ eigenspace of $\ad(Z_0)$ on $\frakp_\CC$. We have the well-known decomposition
\begin{align}
 \frakg_\CC &= \frakp^-+\frakk_\CC+\frakp^+.\label{eq:IntroCartanDecomp}
\end{align}
There is a Cayley type transform $C\in\Int(\frakg_\CC)$ (see \eqref{eq:DefCayleyTypeTransform} for the precise definition) which interchanges the decompositions \eqref{eq:IntroGelfandNaimark} and \eqref{eq:IntroCartanDecomp} in the sense that
\begin{align*}
 C(\frakp^-) &= \frakn_\CC, & C(\frakk_\CC) &= \frakl_\CC, & C(\frakp^+) &= \overline{\frakn}_\CC.
\end{align*}

\subsection*{Unitary highest weight representations}

Let $(\pi,\calH)$ be an irreducible unitary representation of the universal cover $\widetilde{G}$ of $G$ and denote by $(\td\pi,\calH_{\widetilde{K}})$ its underlying $(\frakg,\widetilde{K})$-module, $\widetilde{K}\subseteq\widetilde{G}$ being the universal cover of $K$. We say that $(\pi,\calH)$ or $(\td\pi,\calH_{\widetilde{K}})$ is a \textit{highest weight representation} if
\begin{align*}
 \calH_{\widetilde{K}}^{\frakp^+} &:= \{v\in\calH_{\overline{K}}:\td\pi(\frakp^+)v=0\} \neq 0.
\end{align*}
The space $\calH_{\widetilde{K}}^{\frakp^+}$ is an irreducible $\frakk$-representation. If further $\dim\calH_{\widetilde{K}}^{\frakp^+}=1$ we say that $(\pi,\calH)$ is \textit{of scalar type}. In this case only the center $Z(\frakk)=\RR Z_0$ of $\frakk$ can act nontrivially on $\calH_{\widetilde{K}}^{\frakp^+}$ and the scalar type highest weight representations of $\widetilde{G}$ are parameterized by this scalar. More precisely, we define the Wallach set
\begin{align*}
 \calW &= \left\{0,\frac{d}{2},\ldots,(r-1)\frac{d}{2}\right\}\cup\left((r-1)\frac{d}{2},\infty\right),
\end{align*}
where $d$ is a certain root multiplicity of $\frakg$ and $r$ denotes the rank of the Hermitian symmetric space $G/K$. Then for each $\lambda\in\calW$ there is exactly one scalar type highest weight representation $\calH$ of $\widetilde{G}$ in which $Z_0$ acts on $\calH_{\widetilde{K}}^{\frakp^+}$ by the scalar $-i\frac{r\lambda}{2}$.

\subsection*{Nilpotent orbits}

We identify $\frakg^*$ with $\frakg$ and $\frakp_\CC^*$ with $\frakp_\CC$ via the Killing form of $\frakg_\CC$. By the Kostant--Sekiguchi correspondence the nilpotent adjoint $G$-orbits in $\frakg$ are in bijection with the nilpotent $K_\CC$-orbits in $\frakp_\CC$. For a certain class of orbits we make this correspondence more precise. The analytic subgroup $L$ of $G$ with Lie algebra $\frakl$ acts on $V\cong\frakn$ via the adjoint representation. It has finitely many orbits, among them an open orbit $\Omega\subseteq V$ (the orbit of the identity element in the Jordan algebra) which is a symmetric cone. Its closure has a stratification into lower dimensional orbits $\calO_0,\ldots,\calO_{r-1}$ with $\calO_k\subseteq\overline{\calO_{k+1}}$. Put $\calO_r:=\Omega$. These orbits in $\frakn$ generate nilpotent adjoint $G$-orbits
\begin{align*}
 \calO_k^G &:= G\cdot\calO_k\subseteq\frakg, & k=0,\ldots,r.
\end{align*}
Via the Kostant--Sekiguchi correspondence these $G$-orbits correspond to nilpotent $K_\CC$-orbits $\calO_0^{K_\CC},\ldots,\calO_r^{K_\CC}$ in $\frakp^+\subseteq\frakp_\CC$. Under the Cayley type transform $C\in\Int(\frakg_\CC)$ these $K_\CC$-orbits map to
\begin{align*}
 \calX_k &:= C(\calO_k^{K_\CC})\subseteq\overline{\frakn}_\CC\cong V_\CC, & k=0,\ldots,r.
\end{align*}
Since $C(\frakk_\CC)=\frakl_\CC$ each $\calX_k$ is an $L_\CC$ orbit in $V_\CC$, where $L_\CC$ is a complexification of $L$, acting linearly on $V_\CC$. Under the identifications described above $\calO_k\subseteq\calX_k$ is a totally real submanifold.

\subsection*{The Schr\"odinger model}

Rossi--Vergne \cite{RV76} showed that the scalar type unitary highest weight representation corresponding to the Wallach point $\lambda\in\calW$ can be realized on a space of $L^2$-functions on the orbit $\calO_k$ for $\lambda=k\frac{d}{2}$, $k=0,\ldots,r-1$, and on the orbit $\calO_r=\Omega$ for $\lambda>(r-1)\frac{d}{2}$. For convenience we put $\calO_\lambda:=\calO_k$ where either $\lambda=k\frac{d}{2}$, $0\leq k\leq r-1$, or $k=r$ for $\lambda>(r-1)\frac{d}{2}$. For each $\lambda\in\calW$ there is a certain $L$-equivariant measure $\td\mu_\lambda$ on $\calO_\lambda$ (see Section \ref{sec:StructureGroup} for details). Then the natural irreducible unitary representation of $L\ltimes\exp(\frakn)$ on $L^2(\calO_\lambda,\td\mu_\lambda)$ extends to an irreducible unitary representation $\pi_\lambda$ of $\widetilde{G}$ which is the scalar type highest weight representation belonging to the Wallach point $\lambda\in\calW$. The differential representation $\td\pi_\lambda$ of $\frakg$ in this realization was computed in \cite{HKM12} and is given by differential operators on $\calO_\lambda$ up to order $2$ with polynomial coefficients. A special role is played by the second order vector-valued Bessel operator $\calB_\lambda$ which gives the differential action of $\overline{\frakn}$ (see Section \ref{sec:DiffOperators} for its definition). This operator was first introduced by Dib \cite{Dib90} and studied further in \cite{HKM12}.

\subsection*{Bessel functions on Jordan algebras}

For each $\lambda\in\calW$ we construct Bessel functions $\calJ_\lambda(z,w)$ and $\calI_\lambda(z,w)$ for $z,w\in\calX_\lambda:=L_\CC\cdot\calO_\lambda$ (see Sections \ref{sec:JBessel} and \ref{sec:IBessel}). Both $\calJ_\lambda(z,w)$ and $\calI_\lambda(z,w)$ can be extended to functions on $V_\CC\times\overline{\calX_\lambda}$ which are holomorphic in $z$ and antiholomorphic in $w$. They solve the differential equations
\begin{align*}
 (\calB_\lambda)_z\calJ_\lambda(z,w) &= -\overline{w}\calJ_\lambda(z,w), & (\calB_\lambda)_z\calI_\lambda(z,w) &= \overline{w}\calI_\lambda(z,w).
\end{align*}
For $w=e$ the identity element and $\lambda>(r-1)\frac{d}{2}$ the corresponding one-variable functions $\calJ_\lambda(x)=\calJ_\lambda(x,e)$ and $\calI_\lambda(x)=\calI_\lambda(x,e)$ on $\Omega$ were first studied by Herz~\cite{Her55} for real symmetric matrices and later for the more general case in \cite{Dib90,DG93,DGKR00,FK94,FT87}. We further define a $K$-Bessel function $\calK_\lambda(x)$ on $\calO_\lambda$ for every $\lambda\in\calW$ which solves the same differential equation as $\calI_\lambda(x)$, but has a different asymptotic behaviour (see Section \ref{sec:KBessel}). These functions were introduced by Dib \cite{Dib90} and studied further in \cite{Cle88,FK94}.

\subsection*{The Fock space}

The closure of $\calX_k\cong\calO_k^{K_\CC}$ is an affine algebraic variety and hence the space $\CC[\overline{\calX_k}]$ of regular functions on $\overline{\calX_k}$ equals the space $\calP(\calX_k)$ of restrictions of holomorphic polynomials on $V_\CC$ to $\calX_k$. We construct measures on $\calX_k$ in two steps. Fix $\lambda\in\calW$. First the $L$-equivariant measure $\td\mu_\lambda$ on $\calO_\lambda\subseteq\calX_\lambda$ gives rise to an $L_\CC$-equivariant measure $\td\nu_\lambda$ on $\calX_\lambda$ for every $\lambda\in\calW$. In the second step we extend the function
\begin{align*}
 \omega_\lambda(x) &:= \calK_\lambda\left(\left(\frac{x}{2}\right)^2\right), & x\in\calO_\lambda,
\end{align*}
uniquely to a positive density on $\calX_\lambda=L_\CC\cdot\calO_\lambda$ which, under the identification of $\calX_\lambda$ with a $K_\CC$-orbit, is invariant under the action of $K$. Then every polynomial $F\in\calP(\calX_\lambda)$ is square integrable on $\calX_\lambda$ with respect to the measure $\omega_\lambda\td\nu_\lambda$ and we let $\calF_\lambda$ be the closure of $\calP(\calX_\lambda)$ with respect to the inner product
\begin{align*}
 \langle F,G\rangle_{\calF_\lambda} &:= \const\cdot\int_{\calX_\lambda}{F(z)\overline{G(z)}\omega_\lambda(z)\td\nu_\lambda(z)}, & F,G\in\calP(\calX_\lambda).
\end{align*}

\begin{thmalph}[see Proposition \ref{prop:FockContPointEvaluations} \&\ Theorems \ref{thm:FockRepKernel} \&\ \ref{thm:FockAsRestriction}]\label{thm:IntroFockSpace}
For each $\lambda\in\calW$ the space $\calF_\lambda$ is a Hilbert space of holomorphic functions on $\calX_\lambda$ with reproducing kernel
\begin{align*}
 \KK_\lambda(z,w) &= \calI_\lambda\left(\frac{z}{2},\frac{w}{2}\right), & z,w\in\calX_\lambda.
\end{align*}
We further have the following intrinsic description:
\begin{multline*}
 \calF_\lambda = \left\{F|_{\calX_\lambda}:F:V_\CC\to\CC\mbox{ holomorphic and}\right.\\
 \left.\int_{\calX_\lambda}{|F(z)|^2\omega_\lambda(z)\td\nu_\lambda(z)}<\infty\right\}.
\end{multline*}
\end{thmalph}

In Section \ref{sec:BesselFischer} we prove that the inner product on $\calF_\lambda$ can also be expressed as a Fischer type inner product involving the Bessel operator $\calB_\lambda$.

\subsection*{Unitary action on the Fock space}

Complexifying the differential operators $\td\pi_\lambda(X)$, $X\in\frakg$, in the differential representation of the Schr\"odinger model one obtains a Lie algebra representation $\td\pi_\lambda^\CC$ of $\frakg$ on holomorphic functions on $\calX_\lambda$ by polynomial differential operators (see Section \ref{sec:ComplexificationSchrödinger} for details). We define a $\frakg$-module structure on $\calP(\calX_\lambda)$ by composing this action with the Cayley type transform $C\in\Int(\frakg_\CC)$:
\begin{align*}
 \td\rho_\lambda(X) &:= \td\pi_\lambda^\CC(C(X)), & X\in\frakg.
\end{align*}

\begin{thmalph}[see Propositions \ref{prop:FockIrreducible}, \ref{prop:FockInfUnitary} \&\ \ref{prop:AssVar} \&\ Theorem \ref{thm:UnitaryRepOnFock}]
For each $\lambda\in\calW$ the $\frakg$-module $(\td\rho_\lambda,\calP(\calX_\lambda))$ is an irreducible $(\frakg,\widetilde{K})$-module which is infinitesimally unitary with respect to the inner product on $\calF_\lambda$. It integrates to an irreducible unitary representation $\rho_\lambda$ of $\widetilde{G}$ on $\calF_\lambda$ with associated variety
\begin{align*}
 \Ass(\rho_\lambda) &= \overline{\calO_k^{K_\CC}},
\end{align*}
where $k\in\{0,\ldots,r\}$ is such that $\calX_\lambda=\calX_k$.
\end{thmalph}

The group action of $\widetilde{K}$ in the Fock model is induced by the geometric action of $K_\CC$ on the orbits $\calX_k\cong\calO_k^{K_\CC}$ up to multiplication with a character (see Proposition \ref{prop:FockModelKAction}). Therefore the $\widetilde{K}$-type decomposition of $\rho_\lambda$ equals the decomposition of the space of polynomials on $\calO_k^{K_\CC}\subseteq\frakp^+$ with respect to the natural $K_\CC$-action which is essentially the Hua--Kostant--Schmid decomposition (see Proposition \ref{prop:FockKAction}).

In Section \ref{sec:RingsOfDiffOps} we use the Fock model to give explicit generators for the rings $\DD(\overline{\calX_k})$ of differential operators on the singular affine algebraic varieties $\overline{\calX_k}\subseteq V_\CC$, $k=0,\ldots,r-1$, in terms of the Bessel operator $\calB_\lambda$.

Although there are already known explicit realizations for highest weight representations on spaces of holomorphic functions our construction has a certain advantage. In the realization on holomorphic functions on the corresponding bounded symmetric domain $G/K$ the group action is quite explicit, but the representation space as well as the inner product are defined in a rather abstract way, at least for the singular highest weight representations. In our Fock model the representation space is by Theorem \ref{thm:IntroFockSpace} explicitly given by holomorphic functions on $V_\CC$ which are square integrable with respect to a certain measure on the orbit $\calX_\lambda$. Although our group action is less explicit we still have a useful expression of the Lie algebra action in terms of the Bessel operators which allows e.g. the study of discrete branching laws as done in Theorem \ref{thm:IntroBranching}.

\subsection*{The Segal--Bargmann transform}

The unitary representations $(\pi_\lambda,L^2(\calO_\lambda,\td\mu_\lambda))$ and $(\rho_\lambda,\calF_\lambda)$ are isomorphic and we explicitly find the integral kernel of the intertwining operator which resembles the classical Segal--Bargmann transform \eqref{eq:IntroClassicalBargmann}. For this let $\tr\in V_\CC^*$ denote the Jordan trace of the complex Jordan algebra $V_\CC$ and recall the $I$-Bessel function $\calI_\lambda(z,w)$.

\begin{thmalph}[see Proposition \ref{prop:BargmannMappingProperty} \&\ Theorem \ref{thm:BargmannIsomorphism}]
Let $\lambda\in\calW$, then for $\psi\in L^2(\calO_\lambda,\td\mu_\lambda)$ the formula
\begin{align*}
 \BB_\lambda\psi(z) &= e^{-\frac{1}{2}\tr(z)}\int_{\calO_\lambda}{\calI_\lambda(z,x)e^{-\tr(x)}\psi(x)\td\mu_\lambda(x)}, & z\in\calX_\lambda,
\end{align*}
defines a function $\BB_\lambda\psi\in\calF_\lambda$. This gives a unitary isomorphism $\BB_\lambda:L^2(\calO_\lambda,\td\mu_\lambda)\to\calF_\lambda$ which intertwines the representations $\pi_\lambda$ and $\rho_\lambda$.
\end{thmalph}

The $\frakk$-finite vectors in the Fock model and in the bounded symmetric domain model have the same realization as polynomials on $V_\CC$ and therefore also the intertwiner between these two models is worth studying. In Theorem \ref{thm:IntertwinerFockBdSymDomain} we find an explicit formula for it in terms of its integral kernel.

We remark that Brylinski--Kostant \cite{BK94} construct a Fock space realization for minimal representations of non-Hermitian Lie groups as geometric quantization of the minimal nilpotent $K_\CC$-orbit in $\frakp_\CC^*$ (see also \cite{AF11}). Apart from the fact that their cases are disjoint to ours they do not construct an intertwining operator between the known Schr\"odinger model and their Fock model such as our generalized Segal--Bargmann transform. Further their measure on the $K_\CC$-orbit is not positive whereas our measure is explicitly given in terms of the $K$-Bessel function and hence positive. Moreover, the Lie algebra action in their picture is given by pseudodifferential operators while in our model the Lie algebra acts by differential operators up to second order.

Recently Achab~\cite{Ach12} also constructed a Fock space realization for the minimal representation of Hermitian groups. Her construction looks different from ours and it would be interesting to find a connection between her model and our model.

\subsection*{The unitary inversion operator}

The explicit integral formula for the Segal--Bargmann transform can be used to find the integral kernel of the unitary inversion operator $\calU_\lambda$ in the Schr\"odinger model. This operator $\calU_\lambda$ is a unitary automorphism on $L^2(\calO_\lambda,\td\mu_\lambda)$ of order $2$ (see Section \ref{sec:UnitaryInversionOperator} for its precise definition). The action of $\calU_\lambda$ together with the relatively simple action of a maximal parabolic subgroup generates the whole group action in the Schr\"odinger model. To describe its integral kernel recall the $J$-Bessel function $\calJ_\lambda(z,w)$.

\begin{thmalph}[see Theorem \ref{thm:IntegralKernelInversionOperator}]\label{thm:IntroUnitaryInversion}
For each $\lambda\in\calW$ the unitary inversion operator $\calU_\lambda$ is given by
\begin{align*}
 \calU_\lambda\psi(x) &= 2^{-r\lambda}\int_{\calO_\lambda}{\calJ_\lambda(x,y)\psi(y)\td\mu_\lambda(y)}, & \psi\in C_c^\infty(\calO_\lambda).
\end{align*}
\end{thmalph}

In various special cases Theorem \ref{thm:IntroUnitaryInversion} was proved earlier. Ding--Gross \cite[Theorem 4.10 (iii)]{DG93} and Faraut--Koranyi \cite[Theorem XV.4.1]{FK94} established Theorem \ref{thm:IntroUnitaryInversion} for the relative holomorphic discrete series, i.e. $\lambda$ is contained in the continuous part of the Wallach set and is \lq large enough\rq. (In fact Ding--Gross proved an analogous version of Theorem \ref{thm:IntroUnitaryInversion} for all holomorphic discrete series representations.) For $\frakg=\su(k,k)$ and $\lambda$ in the discrete part of the Wallach set Theorem \ref{thm:IntroUnitaryInversion} was proved by Ding--Gross--Kunze--Richards \cite[Theorem 5.7]{DGKR00}. For $\frakg=\so(2,n)$ and $\lambda$ the minimal non-zero discrete Wallach point Kobayashi--Mano gave two different proofs for this result (see \cite[Theorem 6.1.1]{KM07a} and \cite[Theorem 5.1.1]{KMa11}). The case of the minimal non-zero discrete Wallach point was systematically treated in \cite{HKMO12}. The proof we present for the general case is quite simple, once the Bargmann transform is established, and works in the same way for all scalar type unitary highest weight representations.

In Section \ref{sec:WhittakerVectors} we use the integral kernel of $\calU_\lambda$ to construct explicit algebraic and smooth Whittaker vectors in the Schr\"odinger model.

\subsection*{Application to branching laws}

We demonstrate the use of our Fock model in one specific example, the branching $\so(2,n)\searrow\so(2,m)\oplus\so(n-m)$ of the minimal unitary highest weight representation of $\so(2,n)$. Note that the subalgebra $\frakh:=\so(2,m)\oplus\so(n-m)$ of $\frakg=\so(2,n)$ is symmetric and of holomorphic type, i.e. the corresponding involution of $\frakg$ acts as the identity on the center of $\frakk$. Kobayashi \cite{Kob08} proved that any highest weight representation of scalar type if restricted to a holomorphic type subalgebra is discretely decomposable and the decomposition is multiplicity free. He further determined the explicit decomposition for representations in the relative holomorphic discrete series. For arbitrary highest weight representation he gives a general algorithm to find the explicit branching law. We find the explicit decomposition for the non-zero discrete Wallach point by an easy computation in the Fock model. For $j=m,n$ we denote by $\td\pi_\lambda^{\so(2,j)}$ the unitary highest weight representation of $\so(2,j)$ of scalar type with Wallach parameter $\lambda$ as constructed above. Further we let $\calH^k(\RR^{n-m})$ be the irreducible representation of $\so(n-m)$ on the space of spherical harmonics on $\RR^{n-m}$ of degree $k$.

\begin{thmalph}[see Theorem \ref{thm:BranchingSO(n,2)}]\label{thm:IntroBranching}
The unitary highest weight representation of $\frakg=\so(2,n)$ belonging to the minimal non-zero discrete Wallach point $\lambda=\frac{n-2}{2}$ decomposes under restriction to the subalgebra $\frakh=\so(2,m)\oplus\so(n-m)$ as follows:
\begin{align*}
 \td\pi_\lambda^{\so(2,n)} &= \bigoplus_{k=0}^\infty{\td\pi_{\lambda+k}^{\so(2,m)}\boxtimes\calH^k(\RR^{n-m})}.
\end{align*}
\end{thmalph}

We remark that for odd $n$ the representation $\td\pi_\lambda^{\so(2,n)}$ cannot be obtained via the theta correspondence and hence also the branching law cannot be obtained in this way.

For even $n$ this branching law was proved by Kobayashi--{\O}rsted in \cite[Theorem A]{KO03b} and, as pointed out to the author by B. {\O}rsted, their proof actually carries over also to the case of odd $n$, although it is not explicitly stated in their paper. However, the techniques they use are more involved than the computation we do in Section \ref{sec:ApplicationBranching} to prove Theorem \ref{thm:IntroBranching}.

In the general case of a Hermitian Lie algebra $\frakg$ Kobayashi--Oshima \cite{KO12} recently classified all symmetric subalgebras of $\frakg$ such that the restriction of the minimal scalar type unitary highest weight representation of $\frakg$ is discretely decomposable. In \cite{MO12} the explicit decompositions in these cases are determined.

\subsection*{Organization of the paper}

This paper is organized as follows: In Section \ref{sec:Preliminaries} we recall the necessary background of Euclidean Jordan algebras, their corresponding groups and nilpotent orbits. We also study polynomials and differential operators such as the Bessel operators on Jordan algebras. Section \ref{sec:RealizationsUHWR} describes three known models for unitary highest weight representations of scalar type, the bounded symmetric domain model, the tube domain model and the Schr\"odinger model. We further investigate a natural complexification of the latter one. Section \ref{sec:BesselFunctions} deals with Bessel functions on Jordan algebras. We introduce $J$-, $I$- and $K$-Bessel functions and discuss their properties. In Section \ref{sec:FockModel} the Fock model for unitary highest weight representations of scalar type is constructed. We calculate the reproducing kernel of the Fock space and investigate rings of differential operators on the associated varieties. The Segal--Bargmann transform intertwining the Fock model and the Schr\"odinger model is introduced in Section \ref{sec:BargmannTransform}. Here we also give a formula for the intertwiner between Fock model and bounded symmetric domain model. In Section \ref{sec:UnitaryInversionOperator} we use the Bargmann transform to obtain the integral kernel of the unitary inversion operator. We then describe Whittaker vectors in the Schr\"odinger model in terms of this integral kernel. Finally we use the Fock model in Section \ref{sec:ApplicationBranching} to obtain in a very simple way the branching law for the smallest non-zero highest weight representation of $\so(2,n)$ restricted to $\so(2,m)\oplus\so(n-m)$.\\

\textbf{Acknowledgements:} We thank T. Kobayashi and B. \O rsted for helpful discussions on the topic of this paper.\\

\textbf{Notation:} $\NN=\{1,2,3,\ldots\}$, $\NN_0=\NN\cup\{0\}$, $\RR^\times=\RR\setminus\{0\}$, $\RR_+=\{x\in\RR:x>0\}$, $\CC^\times=\CC\setminus\{0\}$, $\TT=\{z\in\CC:|z|=1\}$,\\
$V^*$: algebraic dual of a vector space $V$, $V'$: topological dual of a topological vector space $V$.

\newpage
\section{Preliminaries}\label{sec:Preliminaries}

In this section we set up the notation and recall the construction of three known models for unitary highest weight representations of Hermitian Lie groups of tube type. For this we use the theory of Euclidean Jordan algebras which formalizes various aspects in a simple fashion. We mostly follow the notation of \cite{FK94}.

\subsection{Simple Euclidean Jordan algebras}

Let $V$ be a simple Euclidean Jordan algebra of dimension $n=\dim V$ with unit element $e$. For $x\in V$ we denote by $L(x)\in\End(V)$ the multiplication by $x$ and by
\begin{align*}
 P(x) ={}& 2L(x)^2-L(x^2)
\intertext{the \textit{quadratic representation}. Its polarized version is given by}
 P(x,y) ={}& L(x)L(y)+L(y)L(x)-L(xy), & x,y\in V.
\intertext{For $x,y\in V$ we also define the \textit{box operator} $x\Box y\in\End(V)$ by}
 x\Box y :={}& L(xy)+[L(x),L(y)]
\end{align*}
and note the identity $(x\Box y)z=P(x,z)y$ for $x,y,z\in V$.

The generic minimal polynomial of $V$ provides the \textit{Jordan trace} $\tr\in V^*$ and the \textit{Jordan determinant} $\Delta$, a homogeneous polynomial of degree $r=\rk V$, the rank of $V$. Both can be expressed in terms of $L(x)$ and $P(x)$ (see \cite[Proposition III.4.2]{FK94}):
\begin{align*}
 \Tr L(x) &= \frac{n}{r}\tr(x), & \Det P(x) &= \Delta(x)^{\frac{2n}{r}},
\end{align*}
for $x\in V$ ($\frac{n}{r}$ is always an integer). Here we use $\Tr$ and $\Det$ for the trace and determinant of a linear operator. Using the Jordan trace we define the \textit{trace form}
\begin{align*}
 (x|y) &:= \tr(xy), & x,y\in V,
\end{align*}
which is positive definite since $V$ is Euclidean. Both $L(x)$ and $P(x)$ are symmetric operators with respect to the trace form (see \cite[Proposition II.4.3]{FK94}). For any orthonormal basis $(e_\alpha)_\alpha\subseteq V$ with respect to the trace form we have the identity (see \cite[Exercise III.6]{FK94})
\begin{align}
 \sum_\alpha{e_\alpha^2} &= \frac{n}{r}e.\label{eq:SumOfSquares}
\end{align}
We use the trace form to normalize the Lebesgue measure $\td x$ on $V$. The trace form also defines a norm on $V$ by
\begin{align*}
 |x| &:= \sqrt{(x|x)}, & x\in V.
\end{align*}

Let $c\in V$ be an \textit{idempotent}, i.e. $c^2=c$. The only possible eigenvalues of the symmetric operator $L(c)$ are $1$, $\frac{1}{2}$ and $0$. Denote by $V(c,1)$, $V(c,\frac{1}{2})$ and $V(c,0)$ the corresponding eigenspaces. The decomposition
\begin{align*}
 V &= V(c,1)\oplus V(c,\tfrac{1}{2})\oplus V(c,0)
\end{align*}
is called \textit{Peirce decomposition}. The subspaces $V(c,1)$ and $V(c,0)$ are themselves simple Euclidean Jordan algebras with identity elements $c$ and $e-c$, respectively. An idempotent $c\in V$ is called \textit{primitive} if it is non-zero and cannot be written as a nontrivial sum of two idempotents. Further, two idempotents $c_1,c_2\in V$ are called \textit{orthogonal} if $c_1c_2=0$. (This implies $(c_1|c_2)=0$.) A set of primitive orthogonal idempotents which add up to the identity $e$ is called a \textit{Jordan frame}. The cardinality of every Jordan frame is equal to the \textit{rank} $r$ of $V$. Further we say that an idempotent $c\in V$ has \textit{rank} $k$ if it can be written as the sum of $k$ primitive orthogonal idempotents. From now on we fix a Jordan frame $c_1,\ldots,c_r$. The symmetric operators $L(c_1),\ldots,L(c_r)$ commute pairwise and therefore we obtain the mutual Peirce decomposition
\begin{align*}
 V &= \bigoplus_{1\leq i\leq j\leq r}{V_{ij}},\\
\intertext{where}
 V_{ii} &= V(c_i,1), && 1\leq i\leq r,\\
 V_{ij} &= V(c_i,\tfrac{1}{2})\cap V(c_j,\tfrac{1}{2}), && 1\leq i<j\leq r.
\end{align*}
The spaces $V_{ii}$ are one-dimensional and spanned by $c_i$ whereas the spaces $V_{ij}$, $1\leq i<j\leq r$ have a common dimension $d$. In particular, this implies the dimension formula
\begin{align*}
 \frac{n}{r} &= 1+(r-1)\frac{d}{2}.
\end{align*}

\subsection{The structure group and its orbits}\label{sec:StructureGroup}

The \textit{structure group} $\Str(V)$ of $V$ is defined to be the group of $g\in\GL(V)$ such that
\begin{align*}
 P(gx) &= gP(x)g^\#, & x\in V,
\end{align*}
where $g^\#$ denotes the adjoint with respect to the trace form. This is by \cite[Exercise VIII.5]{FK94} equivalent to the existence of a scalar $\chi(g)\in\RR^\times$ such that
\begin{align*}
 \Delta(gx) &= \chi(g)\Delta(x), & x\in V.
\end{align*}
The map $\chi:\Str(V)\to\RR^\times$ defines a character of $\Str(V)$ which is on the identity component $L:=\Str(V)_0$ given by (see \cite[Proposition III.4.3]{FK94})
\begin{align*}
 \chi(g) &= \Det(g)^{\frac{r}{n}}, & g\in L.
\end{align*}
Let $O(V)$ denote the orthogonal group of $V$ with respect to the trace form. Then
\begin{equation*}
 K^L := L\cap O(V)
\end{equation*}
is a maximal compact subgroup of $L$, the corresponding Cartan involution being $\vartheta(g)=(g^{-1})^\#$. The group $K^L$ coincides with the identity component of the group of Jordan algebra \textit{automorphisms}, i.e.
\begin{align*}
 K^L &= \{g\in L:g(x\cdot y)=gx\cdot gy\,\forall x,y\in V\}.
\intertext{At the same time $K^L$ is the stabilizer subgroup of the unit element $e$, i.e.}
 K^L &= \{g\in L:ge=e\}.
\end{align*}
The Lie algebra $\frakl=\str(V)$ of $L$ has the Cartan decomposition (see \cite[Proposition VIII.2.6]{FK94})
\begin{align*}
 \frakl &= \frakk^\frakl\oplus\frakp^\frakl,
\intertext{where $\frakk^\frakl$ is the Lie algebra of $K^L$ consisting of all \textit{derivations} of $V$ and $\frakp^\frakl$ is the space of multiplication operators:}
 \frakk^\frakl &= \{D\in\gl(V):D(x\cdot y)=Dx\cdot y+x\cdot Dy,\,\forall x,y\in V\},\\
 \frakp^\frakl &= L(V) = \{L(x):x\in V\}.
\end{align*}

The group $L$ acts on $V$ with finitely many orbits. The orbit $\Omega:=L\cdot e$ is a symmetric cone and is isomorphic to the Riemannian symmetric space $L/K^L$. The Lebesgue measure $\td x$ on $V$ clearly restricts to $\Omega$. The following integral formula due to J.-L. Clerc \cite[Proposition 2.7]{Cle88} is used in Section \ref{sec:KBessel} to calculate $K$-Bessel functions on boundary orbits:

\begin{lemma}\label{lem:IntFormulaOmega}
Let $c\in V$ be an idempotent of rank $k$ and let $\Omega_1$ and $\Omega_0$ denote the symmetric cones in the Euclidean subalgebras $V(c,1)$ and $V(c,0)$, respectively. Further we let $\Delta^0(x)$ be the determinant function of the Euclidean Jordan algebra $V(c,0)$. Then the following integral formula holds:
\begin{multline*}
 \int_\Omega{f(x)\td x} = 2^{-k(r-k)d}\int_{\Omega_1}{\int_{V(c,\frac{1}{2})}{\int_{\Omega_0}{f\left(\exp(c\Box x_{\frac{1}{2}})(x_1+x_0)\right)}}}\\
 {{{\Delta^0(x_0)^{kd}\td x_0}\td x_{\frac{1}{2}}}\td x_1}.
\end{multline*}
\end{lemma}

\begin{proof}
Let $\Phi:\Omega_1\times V(c,\frac{1}{2})\times\Omega_0\to V$ be defined by $\Phi(x_1,x_{\frac{1}{2}},x_0)=\exp(c\Box x_{\frac{1}{2}})(x_1+x_0)$. By \cite[Proposition 3]{Las87} this map is a diffeomorphism onto $\Omega$ and is explicitly given by
\begin{align*}
 \Phi(x_1,x_{\frac{1}{2}},x_0) &= \underbrace{x_1+\frac{1}{2}c(x_{\frac{1}{2}}(x_{\frac{1}{2}}x_0))}_{\in V(c,1)}+\underbrace{x_{\frac{1}{2}}x_0}_{\in V(c,\frac{1}{2})}+\underbrace{x_0}_{\in V(c,0)}.
\end{align*}
Hence, in the Peirce decomposition $V=V(c,1)\oplus V(c,\frac{1}{2})\oplus V(c,0)$ the Jacobian of $\Phi$ takes the form
\begin{align*}
D\Phi(x_1,x_{\frac{1}{2}},x_0) &= \left(\begin{array}{ccc}\1 & * & *\\0 & L(x_0)|_{V(c,\frac{1}{2})} & *\\0 & 0 & \1\end{array}\right).
\end{align*}
By \cite[Propositions IV.4.1 \& IV.4.2]{FK94} we have
\begin{align*}
 \Det(L(x_0)|_{V(c,\frac{1}{2})}) &= 2^{-k(r-k)d}\Delta^0(x_0)^{kd}
\end{align*}
and hence
\begin{align*}
 \left|\Det(D\Phi(x_1,x_{\frac{1}{2}},x_0))\right| &= 2^{-k(r-k)d}\Delta^0(x_0)^{kd}.
\end{align*}
Now the claimed integral formula follows from the transformation formula.
\end{proof}

The boundary $\partial\Omega$ of the symmetric cone $\Omega$ has a stratification into lower dimensional orbits. In fact
\begin{align}
 \overline{\Omega} &= \calO_0\cup\calO_1\cup\ldots\cup\calO_r,\label{eq:StratificationOmega}
\end{align}
where
\begin{align*}
 \calO_j &:= L\cdot e_j & \mbox{with} && e_j &= c_1+\cdots+c_j.
\end{align*}
(We use the convention $e_0=0$ here.) Note that $\calO_r=\Omega$. Each $x\in\calO_j$ has a polar decomposition $x=ka$ with $k\in K^L$ and $a=\sum_{i=1}^j{a_ic_i}$, $a_1,\ldots,a_j>0$ (see \cite[Chapter VI.2]{FK94}). Assuming $a_1\geq\ldots\geq a_j>0$ the element $a$ is unique. We define complex powers of $x=ka\in\calO_k$ by $x^s:=ka^s$, $a^s=\sum_{i=1}^j{a_i^sc_i}$ for $s\in\CC$. Note that for $s\in\NN$ this definition agrees with the usual definition of powers by the Jordan algebra multiplication. Since every $k\in K^L$ is a Jordan algebra automorphism the identity $x^s\cdot x^t=x^{s+t}$ holds for $s,t\in\CC$.

The \textit{Gamma function} of $\Omega$ is for $\Re\lambda>(r-1)\frac{d}{2}$ defined by the absolutely converging integral
\begin{align*}
 \Gamma_\Omega(\lambda) :={}& \int_\Omega{e^{-\tr(x)}\Delta(x)^{\lambda-\frac{n}{r}}\td x}
\intertext{and is extended meromorphically to the whole complex plane by the identity (see \cite[Theorem VII.1.1]{FK94})}
 \Gamma_\Omega(\lambda) ={}& (2\pi)^{\frac{n-r}{2}}\prod_{j=1}^r{\Gamma\left(\lambda-(j-1)\frac{d}{2}\right)}.
\end{align*}
Using the Gamma function we define the \textit{Riesz distributions} $R_\lambda\in\calS'(V)$ by
\begin{align*}
 \langle R_\lambda,\varphi\rangle &:= \frac{2^{r\lambda}}{\Gamma_\Omega(\lambda)}\int_\Omega{\varphi(x)\Delta(x)^{\lambda-\frac{n}{r}}\td x}, & \varphi\in\calS(V).
\end{align*}
This integral converges for $\lambda>(r-1)\frac{d}{2}$ and defines a tempered distribution $R_\lambda\in\calS'(V)$ which has an analytic continuation to all $\lambda\in\CC$ (see \cite[Proposition VII.2.1 \&\ Theorem VII.2.2]{FK94}). The distribution $R_\lambda$ is positive if and only if $\lambda$ belongs to the \textit{Wallach set} (sometimes referred to as the \textit{Berezin--Wallach set}, see \cite[Theorem VII.3.1]{FK94})
\begin{align*}
 \calW &= \left\{0,\frac{d}{2},\ldots,(r-1)\frac{d}{2}\right\}\cup\left((r-1)\frac{d}{2},\infty\right).
\end{align*}
Note that the Wallach set consists of a discrete part $\calW_{\disc}$ and a continuous part $\calW_{\cont}$ given by
\begin{align*}
 \calW_{\disc} &= \left\{0,\frac{d}{2},\ldots,(r-1)\frac{d}{2}\right\}, & \calW_{\cont} &= \left((r-1)\frac{d}{2},\infty\right).
\end{align*}
For $\lambda\in\calW_{\cont}$ in the continuous part the distribution $R_\lambda$ is an $L$-equivariant measure $\td\mu_\lambda$ supported on $\overline{\Omega}$ which is absolutely continuous with respect to the Lebesgue measure $\td x$. Hence the boundary $\partial\Omega=\overline{\calO_{r-1}}$ is a set of measure zero. For $\lambda=k\frac{d}{2}\in\calW_{\disc}$ in the discrete part the distribution $R_\lambda$ is an $L$-equivariant measure $\td\mu_\lambda$ supported on $\overline{\calO_k}$ for which $\partial\calO_k=\overline{\calO_{k-1}}$ is a set of measure zero (see \cite[Proposition VII.2.3]{FK94}). If we put
\begin{equation*}
 \calO_\lambda := \begin{cases}\calO_k & \mbox{for $\lambda=k\frac{d}{2}$, $0\leq k\leq r-1$,}\\\Omega & \mbox{for $\lambda>(r-1)\frac{d}{2}$,}\end{cases}
\end{equation*}
then we obtain measure spaces $(\calO_\lambda,\td\mu_\lambda)$ for $\lambda\in\calW$. The $L$-equivariant measures $\td\mu_\lambda$ transform by
\begin{align*}
 \td\mu_\lambda(gx) &= \chi(g)^\lambda\td\mu(x), & g\in L.
\end{align*}
Note that with this normalization of the measures $\td\mu_\lambda$ the function $\psi_0(x)=e^{-\tr(x)}$ always has norm $1$ in $L^2(\calO_\lambda,\td\mu_\lambda)$. In fact, for $\lambda>(r-1)\frac{d}{2}$ we find
\begin{align*}
 \int_{\calO_\lambda}{|\psi_0(x)|^2\td\mu_\lambda(x)} &= \frac{2^{r\lambda}}{\Gamma_\Omega(\lambda)}\int_\Omega{e^{-2\tr(x)}\Delta(x)^{\lambda-\frac{n}{r}}\td x}\\
 &= \frac{1}{\Gamma_\Omega(\lambda)}\int_\Omega{e^{-\tr(x)}\Delta(x)^{\lambda-\frac{n}{r}}\td x} = 1,
\end{align*}
and the general case follows by analytic continuation.

In terms of the polar decomposition of each orbit $\calO_k$ we have the following integral formula (cf. \cite[Proposition 1.8]{HKM12})
\begin{align}
 \int_{\calO_\lambda}{f(x)\td\mu_\lambda(x)} &= \const\cdot\int_{K^L}{\int_{t_1>\ldots>t_k}{f(ke^{\bf t})J_\lambda({\bf t})\td{\bf t}}\td k},\label{eq:IntFormulaOlambda}
\end{align}
where we use the notation ${\bf t}=(t_1,\ldots,t_k)$,
\begin{align}
 e^{\bf t} &= \sum_{i=1}^k{e^{t_i}c_i},\label{eq:DefExpt}\\
 J_\lambda({\bf t}) &= e^{\frac{r\lambda}{k}\sum_{j=1}^k{t_j}}\prod_{1\leq i<j\leq k}{\sinh^d\left(\frac{t_i-t_j}{2}\right)},\label{eq:Jlambda}
\end{align}
and $k\in\{0,\ldots,r\}$ such that $\calO_\lambda=\calO_k$.

\subsection{Conformal group and KKT algebra}\label{sec:ConformalGroup}

For $a\in V$ let $n_a$ be the translation by $a$, i.e.
\begin{align*}
 n_a(x) &= x+a, & x\in V.
\end{align*}
Denote by $N:=\{n_a:a\in V\}$ the group of all translations. Further define the \textit{conformal inversion} $j$ by
\begin{align*}
 j(x) &:= -x^{-1}, & \mbox{for $x\in V$ invertible.}
\end{align*}
The map $j$ defines a rational transformation of $V$ of order $2$. The \textit{conformal group} $\Co(V)$ of $V$ is defined as the subgroup of all rational transformations of $V$ generated by $N$, $\Str(V)$ and $j$:
\begin{equation*}
 \Co(V) := \langle N,\Str(V),j\rangle_{\textup{grp}}.
\end{equation*}
We let $G:=\Co(V)_0$ be its identity component. $G$ is a connected simple Hermitian Lie group of tube type with trivial center and every such group occurs in this fashion (see Table \ref{tb:Groups} for a classification).

Conjugation with $j$ defines a Cartan involution $\vartheta$ of $G$ by
\begin{align*}
 \vartheta(g) &= j\circ g\circ j, & g\in G.
\end{align*}
Its restriction to $L$ agrees with the previously introduced Cartan involution $\vartheta(g)=(g^{-1})^\#$, $g\in L$. Let $K:=G^\vartheta$ be the corresponding maximal compact subgroup of $K$.

The Lie algebra $\frakg$ of $G$ is called \textit{Kantor--Koecher--Tits algebra}. It acts on $V$ by quadratic vector fields of the form
\begin{align*}
 X(z) &= u+Tz-P(z)v, & z\in V,
\end{align*}
with $u,v\in V$ and $T\in\frakl$. For convenience we write $X=(u,T,v)\in\frakg$ for short. In this notation the Lie bracket of $X_j=(u_j,T_j,v_j)$, $j=1,2$, is given by (see \cite[Proposition X.5.8]{FK94})
\begin{equation*}
 [X_1,X_2] = (T_1u_2-T_2u_1,[T_1,T_2]+2(u_1\Box v_2)-2(u_2\Box v_1),-T_1^\# v_2+T_2^\# v_1).
\end{equation*}
This yields the grading
\begin{align*}
 \frakg &= \frakn\oplus\frakl\oplus\overline{\frakn},
\end{align*}
where $\frakn\cong V$ is the Lie algebra of $N$ and $\overline{\frakn}=\vartheta\frakn$. Thus,
\begin{align*}
 \frakn &= \{(u,0,0):u\in V\}
\intertext{acts via constant vector fields, $\frakl$ via linear vector fields and}
 \overline{\frakn} &= \{(0,0,v):v\in V\}
\end{align*}
by quadratic vector fields. Note that the abelian subalgebras $\frakn$ and $\overline{\frakn}$ together generate $\frakg$ as a Lie algebra.

On $\frakg$ the Cartan involution $\vartheta$ acts via 
\begin{align*}
 \vartheta(u,D+L(a),v) &= (-v,D-L(a),-u)
\end{align*}
and hence gives the Cartan decomposition $\frakg=\frakk\oplus\frakp$ with
\begin{align*}
 \frakk &= \{(u,D,-u):u\in V,D\in\frakk^\frakl\},\\
 \frakp &= \{(u,L(a),u):u,a\in V\}.
\end{align*}
The Lie algebra $\frakk$ of $K$ has a one-dimensional center (see \cite[Lemma 1.6.2]{Moe10})
\begin{align*}
 Z(\frakk) &= \RR Z_0, & Z_0 &= -\tfrac{1}{2}(e,0,-e),
\end{align*}
and hence the universal cover $\widetilde{K}$ of $K$ is the direct product of $\RR$ with a compact connected simply-connected semisimple group. Therefore the deck transformation group of the universal cover $\widetilde{G}$ of $G$ is a finite extension of $\ZZ$.

Under the action of $K$ the space $\frakp_\CC$ decomposes into two irreducible components
\begin{align*}
 \frakp_\CC &= \frakp^+\oplus\frakp^-.
\end{align*}
In fact, $\frakp^\pm$ is the $\pm i$ eigenspace of $\ad(Z_0)$ on $\frakp_\CC$ and is explicitly given by
\begin{align*}
 \frakp^\pm &= \{(u,\pm2iL(u),u):u\in V_\CC\}.
\end{align*}

We set 
\begin{align*}
 E &:= (e,0,0), & H &:= (0,2\,\id,0), & F &:= (0,0,e).
\end{align*}
Then $(E,F,H)$ forms an $\sl_2$-triple in $\frakg$. We define a \textit{Cayley type transform} $C\in\operatorname{Int}(\frakg_\CC)$ by the formula
\begin{align}
 C &:= \exp(-\tfrac{1}{2}i\ad(E))\exp(-i\ad(F)).\label{eq:DefCayleyTypeTransform}
\end{align}
It is then routine to check the following formulas:
\begin{align}
 C(a,0,0) &= \textstyle(\frac{a}{4},iL(a),a),\label{eq:CayleyTypeTransform1}\\
 C(0,L(a)+D,0) &= \textstyle(i\frac{a}{4},D,-ia),\label{eq:CayleyTypeTransform2}\\
 C(0,0,a) &= \textstyle(\frac{a}{4},-iL(a),a).\label{eq:CayleyTypeTransform3}
\end{align}
From these formulas it is easy to see that the transform $C$ induces isomorphisms
\begin{align*}
 \frakk_\CC&\stackrel{\sim}{\to}\frakl_\CC,\,(u,D,-u)\mapsto D+2iL(u),\\
 \frakp^+&\stackrel{\sim}{\to}\overline{\frakn}_\CC,\,(u,2iL(u),u)\mapsto(0,0,4u),\\
 \frakp^-&\stackrel{\sim}{\to}\frakn_\CC,\,(u,-2iL(u),u)\mapsto(u,0,0).
\end{align*}
Since $\frakn$ and $\overline{\frakn}$ together generate $\frakg$ as a Lie algebra it is now immediate that $\frakp^++\frakp^-$ generates $\frakg_\CC$ as a Lie algebra. Hence the real form $\{(a,L(b),a):a,b\in V\}$ of $\frakp^++\frakp^-$ generates $\frakg$.

\begin{table}[H]
\begin{center}
\begin{tabular}{|c|c|c|c|c|c|}
  \cline{1-6}
  $V$ & $\co(V)$ & $\str(V)$ & $n$ & $r$ & $d$\\
  \hline\hline
  $\RR$ & $\sl(2,\RR)$ & $\RR$ & $1$ & $1$ & $0$\\
  $\Sym(k,\RR)$ ($k\geq2$) & $\sp(k,\RR)$ & $\sl(k,\RR)\oplus\RR$ & $\frac{1}{2}k(k+1)$ & $k$ & $1$\\
  $\Herm(k,\CC)$ ($k\geq2$) & $\su(k,k)$ & $\sl(k,\CC)\oplus\RR$ & $k^2$ & $k$ & $2$\\
  $\Herm(k,\HH)$ ($k\geq2$) & $\so^*(4k)$ & $\su^*(2k)\oplus\RR$ & $k(2k-1)$ & $k$ & $4$\\
  $\RR^{1,k-1}$ ($k\geq3$) & $\so(2,k)$ & $\so(1,k-1)\oplus\RR$ & $k$ & $2$ & $k-2$\\
  $\Herm(3,\OO)$ & $\mathfrak{e}_{7(-25)}$ & $\mathfrak{e}_{6(-26)}\oplus\RR$ & $27$ & $3$ & $8$\\
  \hline
\end{tabular}
\caption{Simple Euclidean Jordan algebras, corresponding groups and structure constants\label{tb:Groups}}
\end{center}
\end{table}

\subsection{Complexifications}\label{sec:Complexifications}

The complexification $V_\CC$ of $V$ is a complex Jordan algebra. The Jordan trace of $V_\CC$ is given by the $\CC$-linear extension of the Jordan trace $\tr$ of $V$. The same holds for the trace form which is the $\CC$-bilinear extension of the trace form of $V$. By abuse of notation we use the same notation $\tr$ and $(-|-)$ for the complex objects. A norm on $V_\CC$ is given by
\begin{align*}
 |z| &:= \sqrt{(z|\overline{z})}, & z\in V_\CC.
\end{align*}

$V_\CC$ can also be considered as a real Jordan algebra and we write $W=V_\CC$ for it. As real Jordan algebra $W$ is simple since $V$ is simple Euclidean. The trace form $(-|-)_W$ of $W$ is of signature $(n,n)$ and explicitly given by (see \cite[Lemma 1.2.3~(1)]{Moe10})
\begin{align*}
 (z|w)_W &= 2\left((\Re(z)|\Re(w))-(\Im(z)|\Im(w))\right) = 2\Re(z|w), & z,w\in W.
\end{align*}
Note that $(-|-)_W$ is not $\CC$-linear.

The identity component $L_\CC:=\Str(V_\CC)_0$ of the structure group of the complex Jordan algebra $V_\CC$ is a natural complexification of $L$ (see \cite[Proposition VIII.2.6]{FK94}). As a real Lie group it is the same as $\Str(W)_0$, the identity component of the structure group of the real Jordan algebra $W$. Let $U\subseteq L_\CC$ denote the analytic subgroup of $L_\CC$ with Lie algebra
\begin{equation*}
 \fraku:=\frakk^\frakl+i\frakp^\frakl,
\end{equation*}
then $U$ is a maximal compact subgroup of $L_\CC$. The isomorphism $C:\frakk_\CC\to\frakl_\CC$ introduced in Section \ref{sec:ConformalGroup} restricts to an isomorphism $\frakk\stackrel{\sim}{\to}\fraku$ which integrates to a covering map $\eta:\widetilde{K}\to U\subseteq L_\CC$ with differential
\begin{equation*}
 \td\eta(u,D,-u)=D+2iL(u).
\end{equation*}
Since $\fraku_\CC=\frakl_\CC$ the subgroup $U$ is totally real in $L_\CC$.

For $g\in L_\CC$ we denote by $g^\#$ the adjoint of $g$ with respect to the $\CC$-bilinear trace form $(-|-)$ of $V_\CC$. Since $(-|-)_W=2\Re(-|-)$ this is also the adjoint with respect to the trace form $(-|-)_W$ of $W$. Further put
\begin{align*}
 g^* &:= \overline{g}^\#, & g\in L_\CC.
\end{align*}
Then $g\mapsto(g^{-1})^*$ defines a Cartan involution of $L_\CC$ with corresponding maximal compact subgroup $U$.

The determinant function $\Delta_W$ of the real Jordan algebra $W$ is a polynomial of degree $2r$. As in the Euclidean case an element $g\in\Str(W)$ is characterized by the property that there exists a scalar $\chi_W(g)\in\RR^\times$ such that
\begin{align*}
 \Delta_W(gz) &= \chi_W(g)\Delta_W(z), & z\in W.
\end{align*}
This defines a real character $\chi_W:L_\CC\to\RR_+$.

The group $L_\CC$ has only one open orbit $L_\CC\cdot e$ (see \cite[Theorem 4.2]{Kan98}). Its boundary again has a stratification by lower-dimensional orbits. More precisely we have
\begin{align*}
 \overline{L_\CC\cdot e} &= \calX_0\cup\calX_1\cup\ldots\cup\calX_r,
\end{align*}
where $\calX_k=L_\CC\cdot e_k$ (see \cite[Theorem 4.2]{Kan98}). Note that $\calO_k\subseteq\calX_k$ is totally real and $\calX_k$ is closed under conjugation. The closure $\overline{\calX_k}$ of each $L_\CC$-orbit is an affine algebraic subvariety of $V_\CC$ (see \cite[Theorem 2.9]{GK98}). From \cite[Section 1.5.2]{Moe10} it is easy to deduce the following formula for the dimension of $\calX_k$:
\begin{align}
 \dim_\CC\calX_k &= \dim_\RR\calO_k = k+k(2r-k-1)\frac{d}{2}.\label{eq:OrbitDimension}
\end{align}
As in the real case every $z\in\calX_k$ has a polar decomposition $z=ua$ with $u\in U$ and $a=\sum_{i=1}^k{a_ic_i}$, $a_1\geq\ldots\geq a_k>0$, the element $a$ being unique with this property (see \cite[Proposition X.3.2]{FK94}).

For $\lambda\in\calW$ and $k\in\{0,\ldots,r\}$ such that $\calO_\lambda=\calO_k$ we similarly put $\calX_\lambda:=\calX_k$.

\begin{proposition}
For every $\lambda\in\calW$ the measure $\td\nu_\lambda$ on $\calX_\lambda$ defined by
\begin{align}
 \int_{\calX_\lambda}{f(z)\td\nu_\lambda(z)} &= \int_U{\int_{\calO_\lambda}{f(ux^{\frac{1}{2}})\td\mu_\lambda(x)}\td u}\label{eq:IntFormulaOnXlambda}
\end{align}
is the unique (up to scalar multiples) $L_\CC$-equivariant measure on $\calX_\lambda$ which transforms by $\chi_W^\lambda$.
\end{proposition}

\begin{proof}
Except for the integral formula \eqref{eq:IntFormulaOnXlambda} this is essentially \cite[Proposition 1.8]{HKM12}. By \cite[Proposition 1.8]{HKM12} a $\chi_W^\lambda$-equivariant measure on $\calX_\lambda$ is given by
\begin{align*}
 f\mapsto\int_U{\int_{t_1>\ldots>t_k}{f(ue^{\bf t})J_\lambda^W({\bf t})\td{\bf t}}\td u},
\end{align*}
where $e^{\bf t}$ is as in \eqref{eq:DefExpt},
\begin{align*}
 J_\lambda^W({\bf t}) &= e^{\frac{2r\lambda}{k}\sum_{j=1}^k{t_j}}\prod_{1\leq i<j\leq k}{\sinh^d\left(\frac{t_i-t_j}{2}\right)\cosh^d\left(\frac{t_i-t_j}{2}\right)},
\end{align*}
and $k\in\{0,\ldots,r\}$ such that $\calX_k=\calX_\lambda$. Using $\sinh x\cdot\cosh x=\frac{1}{2}\sinh 2x$ we find
\begin{align*}
 J_\lambda^W({\bf t}) &= 2^{-k(k-1)\frac{d}{2}}e^{\frac{r\lambda}{k}\sum_{j=1}^k{2t_j}}\prod_{1\leq i<j\leq k}{\sinh^d\left(t_i-t_j\right)} = 2^{-k(k-1)\frac{d}{2}}J_\lambda(2{\bf t})
\end{align*}
with $J_\lambda({\bf s})$ as in \eqref{eq:Jlambda}. Now note that $K^L\subseteq U$ and $(ke^{\bf t})^{\frac{1}{2}}=ke^{\frac{1}{2}{\bf t}}$ for $k\in K^L$ and hence
\begin{align*}
& \int_U{\int_{t_1>\ldots>t_k}{f(ue^{\bf t})J_\lambda^W({\bf t})\td{\bf t}}\td u}\\
 ={}& \const\cdot\int_U{\int_{K^L}{\int_{t_1>\ldots>t_k}{f(uke^{\bf t})J_\lambda(2{\bf t})\td{\bf t}}\td k}\td u}\\
 ={}& \const\cdot\int_U{\int_{K^L}{\int_{s_1>\ldots>s_k}{f(uke^{\frac{1}{2}{\bf s}})J_\lambda({\bf s})\td{\bf s}}\td k}\td u}\\
 ={}& \const\cdot\int_U{\int_{\calO_\lambda}{f(ux^{\frac{1}{2}})\td\mu_\lambda(x)}\td u},
\end{align*}
where we have used the integral formula \eqref{eq:IntFormulaOlambda} for the last equality.
\end{proof}

\subsection{Nilpotent orbits and the Kostant--Sekiguchi correspondence}\label{sec:NilpotentOrbits}

\subsubsection*{Nilpotent coadjoint $G$-orbits}

We identify $\frakg^*$ with $\frakg$ by means of the Killing form and view coadjoint $G$-orbits on $\frakg^*$ as adjoint $G$-orbits on $\frakg$. Further we use the embedding $V\hookrightarrow\frakg,\,u\mapsto(u,0,0)$ to identify the $L$-orbits $\calO_k\subseteq V$, $k=0,\ldots,r$, with $L$-orbits in $\frakn\cong V$. Since $\frakn$ is nilpotent, the $G$-orbits
\begin{equation*}
 \calO_k^G := G\cdot(e_k,0,0)
\end{equation*}
are nilpotent adjoint orbits in $\frakg$. Since further $\calO_k=L\cdot(e_k,0,0)$ we clearly have $\calO_k\subseteq\calO_k^G$. As usual we endow $\calO_k^G$ with the Kirillov--Kostant--Souriau symplectic form. The following result was proved in \cite[Theorem 2.9 (4)]{HKM12} for the minimal orbit $\calO_1$. For the general case the author could not find the statement in the existing literature.

\begin{proposition}
$\calO_k\subseteq\calO_k^G$ is a Lagrangian subvariety.
\end{proposition}

\begin{proof}
Since $\calO_k\subseteq\frakn$ and $\frakn$ is abelian the symplectic form vanishes on $\calO$ and it remains to show that $\dim\calO_k^G=2\dim\calO_k$. By \eqref{eq:OrbitDimension} we have
\begin{align*}
 \dim\calO_k &= k+k(2r-k-1)\frac{d}{2}.
\end{align*}
To determine $\dim\calO_k^G$ we note that
\begin{align*}
 0 &= [(u,T,v),(e_k,0,0)] = (Te_k,-2e_k\Box v,0)
\end{align*}
if and only if $Te_k=0$ and $v\in V(e_k,0)$. Hence
\begin{align*}
 \dim\calO_k^G &= \dim\calO_k+\dim(V/V(e_k,0))\\
 &= \dim\calO_k+\dim V(e_k,1)+\dim V(e_k,\tfrac{1}{2})\\
 &= \dim\calO_k+(k+k(k-1)\tfrac{d}{2})+k(r-k)d\\
 &= 2\dim\calO_k.\qedhere
\end{align*}
\end{proof}

\subsubsection*{Nilpotent $K_\CC$-orbits}

Via the Kostant--Sekiguchi correspondence the nilpotent adjoint $G$-orbits $\calO_k^G\subseteq\frakg^*$ correspond to nilpotent $K_\CC$-orbits $\calO_k^{K_\CC}\subseteq\frakp_\CC^*$. Identifying $\frakp^*$ with $\frakp$ by means of the Killing form we view $\calO_k^{K_\CC}$ as $K_\CC$-orbits in $\frakp_\CC=\frakp^++\frakp^-$. Following \cite{Sek87} we let $(E_k,F_k,H_k)$ be the strictly normal $\sl_2$-triple in $\frakg$ given by
\begin{align*}
 E_k &:= (e_k,0,0), & H_k &:= (0,2L(e_k),0), & F_k &:= (0,0,e_k).
\end{align*}
We form a new $\sl_2$-triple $(E_k^d,F_k^d,H_k^d)$ by putting
\begin{align*}
 E_k^d &:= \frac{1}{2}(E_k+F_k+iH_k) = \frac{1}{2}(e_k,2iL(e_k),e_k),\\
 H_k^d &:= i(E_k-F_k) = i(e_k,0,-e_k),\\
 F_k^d &:= \frac{1}{2}(E_k+F_k-iH_k) = \frac{1}{2}(e_k,-2iL(e_k),e_k).
\end{align*}
Then $E_k^d\in\frakp^+$ and the $K_\CC$ orbit $\calO_k^{K_\CC}$ corresponding to $\calO_k^G$ is given by
\begin{equation*}
 \calO_k^{K_\CC} = K_\CC\cdot E_k^d.
\end{equation*}
Since $\frakp^+$ is $K_\CC$-stable we have $\calO_k^{K_\CC}\subseteq\frakp^+$.

If we use the embedding $V_\CC\hookrightarrow\frakg_\CC,\,u\mapsto(0,0,u)$ to identify the $L_\CC$-orbits $\calX_k\subseteq V_\CC$, $k=0,\ldots,r$, with $L_\CC$-orbits in $\overline{\frakn}_\CC\cong V_\CC$ we obtain the following result:

\begin{proposition}
The Cayley type transform $C\in\Int(\frakg_\CC)$ is a bijection from the $K_\CC$-orbit $\calO_k^{K_\CC}\subseteq\frakp^+$ onto the $L_\CC$-orbit $\calX_k\subseteq\overline{\frakn}_\CC$ for every $k\in\{0,\ldots,r\}$.
\end{proposition}

\begin{proof}
Observe that $C(E_k^d)=(0,0,4e_k)$. Since the Lie algebra homomorphism $C$ is a bijection from $\frakk_\CC$ onto $\frakl_\CC$ it gives a bijection between the orbits $\calO_k^{K_\CC}=K_\CC\cdot E_k^d$ and $\calX_k=L_\CC\cdot(0,0,4e_k)$ which shows the claim.
\end{proof}

\subsection{Differential operators}\label{sec:DiffOperators}

\subsubsection*{General notation}

For a map $f:X\to X$ on a (real or complex) vector space $X$ we write $Df(x):X\to X$ for its (real or complex) Jacobian at the point $x\in X$. For a complex-valued function $g:X\to\CC$ we denote the directional derivative in the direction of $a\in X$ by $\partial_ag$:
\begin{align*}
 \partial_ag(x) &= \left.\frac{\td}{\td t}\right|_{t=0}g(x+ta).
\end{align*}

\subsubsection*{Bessel operators - the real case}

Let $J$ be a real semisimple Jordan algebra with trace form $(-|-)_J$. The \textit{gradient} with respect to $(-|-)_J$ will be denoted by $\frac{\partial}{\partial x}$. If $(e_\alpha)_\alpha\subseteq J$ is a basis of $J$ with dual basis $(\widehat{e}_\alpha)_\alpha\subseteq J$ with respect to the trace form then the gradient is in coordinates $x=\sum_\alpha{x_\alpha e_\alpha}\in J$ expressed as
\begin{align*}
 \frac{\partial}{\partial x} &= \sum_\alpha{\frac{\partial}{\partial x_\alpha}\widehat{e}_\alpha}.
\end{align*}
In terms of the directional derivative the gradient is characterized by the identity $\left(a\left|\frac{\partial}{\partial x}\right.\right)_J=\partial_a$ for $a\in J$. The \textit{Bessel operator} $\calB_\lambda^J$ with parameter $\lambda\in\CC$ is a vector-valued second order differential operator on $J$ given by
\begin{align*}
 \calB_\lambda^J &:= P\left(\frac{\partial}{\partial x}\right)x+\lambda\frac{\partial}{\partial x}.
\end{align*}
In coordinates it is given by
\begin{align*}
 \calB_\lambda^J &= \sum_{\alpha,\beta}{\frac{\partial^2}{\partial x_\alpha\partial x_\beta}P(\widehat{e}_\alpha,\widehat{e}_\beta)x}+\lambda\sum_\alpha{\frac{\partial}{\partial x_\alpha}\widehat{e}_\alpha}.
\end{align*}
Denote by $\ell$ the left-action of $\Str(J)$ on functions on $J$ given by $\ell(g)f(x)=f(g^{-1}x)$. Then the Bessel operator satisfies the following equivariance property:
\begin{align}
 \ell(g)\calB_\lambda^J\ell(g^{-1}) &= g^\#\calB_\lambda^J, & g\in\Str(J),\label{eq:BesselEquivariance}
\end{align}
where $g^\#$ denotes the adjoint with respect to the trace form of $J$. We further have the following product rule:
\begin{align}
 \calB_\lambda^J\left[f(x)g(x)\right] &= \calB_\lambda^Jf(x)\cdot g(x)+2P\left(\frac{\partial f}{\partial x},\frac{\partial g}{\partial x}\right)x+f(x)\cdot\calB_\lambda^Jg(x).\label{eq:BesselProdRule}
\end{align}

For either $J=V$ or $J=W=V_\CC$ the operator $\calB_\lambda^J$ is tangential to the orbit $\Str(J)_0\cdot e_k$ if $\lambda=k\frac{d}{2}$, $0\leq k\leq r-1$, and hence defines a differential operator on this orbit (see \cite[Theorem 1.7.5]{Moe10}). On $K^L$-invariant (resp. $U$-invariant) functions on $V$ (resp. $W$) we have the following formulas:

\begin{proposition}\label{prop:BesselOnInvariants}
\begin{enumerate}
\item \textup{(\cite[Theorem XV.2.7]{FK94})} Let $f\in C^\infty(V)$ be $K^L$-invariant. Then with $F(a_1,\ldots,a_r)=f(a)$, $a=\sum_{i=1}^r{a_ic_i}$, we have
\begin{align*}
 \calB^V_\lambda f(a) &= \sum_{i=1}^r{(\calB^V_\lambda)^iF(a_1,\ldots,a_r)c_i},
\end{align*}
where
\begin{multline*}
 (\calB^V_\lambda)^i = a_i\frac{\partial^2}{\partial a_i^2}+\left(\lambda-(r-1)\frac{d}{2}\right)\frac{\partial}{\partial a_i}\\
 +\frac{d}{2}\sum_{j\neq i}{\frac{1}{a_i-a_j}\left(a_i\frac{\partial}{\partial a_i}-a_j\frac{\partial}{\partial a_j}\right)}.
\end{multline*}
\item \textup{(\cite{MS12c})} Let $f\in C^\infty(W)$ be $U$-invariant. Then with $F(a_1,\ldots,a_r)=f(a)$, $a=\sum_{i=1}^r{a_ic_i}$, we have
\begin{align*}
 \calB^W_\lambda f(a) &= \sum_{i=1}^r{(\calB^W_\lambda)^iF(a_1,\ldots,a_r)c_i},
\end{align*}
where
\begin{multline*}
 (\calB^W_\lambda)^i = \frac{1}{4}\Bigg(a_i\frac{\partial^2}{\partial a_i^2}+\left(2\lambda-1-(r-1)d\right)\frac{\partial}{\partial a_i}\\
 +\frac{d}{2}\sum_{j\neq i}{\left(\frac{1}{a_i-a_j}+\frac{1}{a_i+a_j}\right)\left(a_i\frac{\partial}{\partial a_i}-a_j\frac{\partial}{\partial a_j}\right)}\Bigg).
\end{multline*}
\end{enumerate}
\end{proposition}

For $J=V$ we will write $\calB_\lambda:=\calB^V_\lambda$ for short.

\subsubsection*{Bessel operators - the complex case}

We also need the Bessel operator of the complex Jordan algebra $V_\CC$. It is a holomorphic vector-valued differential operator on $V_\CC$ which is defined in the same way as the Bessel operator of $V$. More precisely, if $(e_\alpha)_\alpha\subseteq V_\CC$ is a $\CC$-basis of $V_\CC$ with dual basis $(\widehat{e}_\alpha)_\alpha\subseteq V_\CC$ with respect to the trace form $(-|-)$ (which is $\CC$-bilinear) we define the gradient of $V_\CC$ by
\begin{align*}
 \frac{\partial}{\partial z} &= \sum_\alpha{\frac{\partial}{\partial z_\alpha}\widehat{e}_\alpha},\\
\intertext{where}
 \frac{\partial}{\partial z_\alpha} &= \frac{1}{2}\left(\frac{\partial}{\partial x_\alpha}-i\frac{\partial}{\partial y_\alpha}\right)
\end{align*}
is the Wirtinger derivative and $z=\sum_\alpha{z_\alpha e_\alpha}$ with $z_\alpha=x_\alpha+iy_\alpha$. The Bessel operator $\calB_\lambda^{V_\CC}$ of $V_\CC$ is then defined by
\begin{align*}
 \calB_\lambda^{V_\CC} &= P\left(\frac{\partial}{\partial z}\right)z+\lambda\frac{\partial}{\partial z}\\
 &= \sum_{\alpha,\beta}{\frac{\partial^2}{\partial z_\alpha\partial z_\beta}P(\widehat{e}_\alpha,\widehat{e}_\beta)z}+\lambda\sum_\alpha{\frac{\partial}{\partial z_\alpha}\widehat{e}_\alpha}.
\end{align*}
For a holomorphic function $f$ on $V_\CC$ we have
\begin{align*}
 \left.\left(\calB_\lambda^{V_\CC}f\right)\right|_V &= \calB_\lambda^V(f|_V)
\end{align*}
and hence $\calB_\lambda^{V_\CC}$ is a natural complexification of the Bessel operator $\calB_\lambda$ of $V$. Therefore, by abuse of notation, we also abbreviate $\calB_\lambda=\calB_\lambda^{V_\CC}$.

\subsection{Bounded symmetric domains of tube type}

Let $T_\Omega:=V+i\Omega\subseteq V_\CC$ be the tube domain associated with the symmetric cone $\Omega$. Each rational transformation $g\in G$ extends to a holomorphic automorphism of $T_\Omega$. This establishes an isomorphism between $G$ and the identity component $\Aut(T_\Omega)_0$ of the group of all holomorphic automorphisms of $T_\Omega$ (see \cite[Theorem X.5.6]{FK94}). Under this isomorphism the maximal compact subgroup $K\subseteq G$ corresponds to the stabilizer subgroup of the element $ie\in T_\Omega$ so that $T_\Omega\cong G/K$. We identify an element $g\in G$ with its holomorphic extension to $T_\Omega$.

Let
\begin{align*}
 p:T_\Omega\to V_\CC,\,p(z):=(z-ie)(z+ie)^{-1},
\end{align*}
and define $\calD:=p(T_\Omega)$. Then $p$ restricts to a holomorphic isomorphism $p:T_\Omega\stackrel{\sim}{\to}\calD$ with inverse the \textit{Cayley transform} (see \cite[Theorem X.4.3]{FK94})
\begin{align}
 c:\calD\stackrel{\sim}{\to}T_\Omega,\,c(w)=i(e+w)(e-w)^{-1}.\label{eq:DefCayleyTransform}
\end{align}
The open set $\calD\subseteq V_\CC$ is a bounded symmetric domain of tube type and every such domain occurs in this fashion. Clearly $G$ is also isomorphic to the identity component $\Aut(\calD)_0$ of the group of all holomorphic automorphisms of $\calD$ via the conjugation map
\begin{align}
 \alpha: G\to\Aut(\calD)_0,\,g\mapsto p\circ g\circ p^{-1}=c^{-1}\circ g\circ c.\label{eq:DefinitionAlpha}
\end{align}
The stabilizer subgroup of the origin $0\in\calD$ corresponds to the maximal compact subgroup $K\subseteq G$ and also $\calD\cong G/K$.

Denote by $\Sigma\subseteq\partial\calD$ the Shilov boundary of $\calD$. The group $U$ acts transitively on $\Sigma$ and the stabilizer of $e\in\Sigma$ is equal to $K^L\subseteq U$ (see \cite[Proposition X.3.1 \&\ Theorem X.4.6]{FK94}). Hence $\Sigma\cong U/K^L$ is a compact symmetric space. Denote by $\td\sigma$ the normalized $U$-invariant measure on $\Sigma\cong U/K^L$. We write $L^2(\Sigma)$ for $L^2(\Sigma,\td\sigma)$ and denote the corresponding $L^2$-inner product by $\langle-,-\rangle_\Sigma$.

\subsection{Polynomials}\label{sec:Polynomials}

In this Section we recall known properties for the space of polynomials on $V_\CC$.

\subsubsection*{Principal minors and decompositions}

Let $\calP(V_\CC)$ denote the space of holomorphic polynomials on $V_\CC$. The space $\calP(V_\CC)$ carries a natural representation $\ell$ of $L_\CC$ given by
\begin{align*}
 \ell(g)p(z) &= p(g^{-1}z), & g\in L_\CC,p\in\calP(V_\CC).
\end{align*}
Since $L$ and $U$ are both real forms of the complex group $L_\CC$ the decompositions of $\calP(V_\CC)$ into irreducible $L$- and $U$-representations are the same. This decomposition can be described as follows:

For $j\in\{1,\ldots,r\}$ we denote by $P_j$ the orthogonal projection $V\to V(e_j,1)$. Let $\Delta_{V(e_j,1)}$ be the Jordan determinant of the Euclidean Jordan algebra $V(e_j,1)$. We define a polynomial $\Delta_j$ on $V$ by the formula
\begin{align*}
 \Delta_j(x) &:= \Delta_{V(e_j,1)}(P_jx), & x\in V,
\end{align*}
and extend it to a holomorphic polynomial on $V_\CC$. The polynomials $\Delta_1,\ldots,\Delta_r$ are called \textit{principal minors} of $V$. Note that $\Delta_r=\Delta$ is the Jordan determinant of $V$. For ${\bf m}\in\NN_0^r$ we write ${\bf m}\geq0$ if $m_1\geq\ldots\geq m_r\geq0$. If ${\bf m}\geq0$ we define a polynomial $\Delta_{\bf m}$ on $V_\CC$ by
\begin{align*}
 \Delta_{\bf m}(z) &:= \Delta_1(z)^{m_1-m_2}\cdots\Delta_{r-1}(z)^{m_{r-1}-m_r}\Delta_r(z)^{m_r}, & z\in V_\CC.
\end{align*}
The polynomials $\Delta_{\bf m}(z)$ are called \textit{generalized power functions}. Define $\calP_{\bf m}(V_\CC)$ to be the subspace of $\calP(V_\CC)$ spanned by $\ell(g)\Delta_{\bf m}$ for $g\in L$. (Equivalently one can let $g\in U$ or $g\in L_\CC$.) We write $d_{\bf m}:=\dim\calP_{\bf m}(V_\CC)$ for its dimension.

\begin{theorem}[Hua--Kostant--Schmid, see e.g. {\cite[Theorem XI.2.4]{FK94}}]\label{thm:HuaKostantSchmid}
For each ${\bf m}\geq0$ the space $\calP_{\bf m}(V_\CC)$ is an irreducible $L$-module and the space $\calP(V_\CC)$ decomposes into irreducible $L$-modules as
\begin{align*}
 \calP(V_\CC) &= \bigoplus_{\bf m\geq0}{\calP_{\bf m}(V_\CC)}.
\end{align*}
\end{theorem}

The same spaces $\calP_{\bf m}(V_\CC)$ occur in the decomposition of $L^2(\Sigma)$ into irreducible $U$-representations. To make this precise let $\ZZ_+^r$ be the set of all ${\bf m}\in\ZZ^r$ with $m_1\geq\ldots\geq m_r$. For ${\bf m}\in\ZZ_+^r$ we define another tuple ${\bf m}':=(m_1-m_r,\ldots,m_{r-1}-m_r,0)$. Then ${\bf m}'\geq0$ and we define
\begin{align*}
 \calP_{\bf m}(\Sigma) &:= \{\Delta^{m_r}p|_\Sigma:p\in\calP_{{\bf m}'}(V_\CC)\}.
\end{align*}
If ${\bf m}\geq0$ then $\calP_{\bf m}(\Sigma)$ coincides with the space of restrictions of polynomials in $\calP_{\bf m}(V_\CC)$ to $\Sigma$.

\begin{theorem}[Cartan--Helgason, see e.g. {\cite[Theorem XII.2.2]{FK94}}]\label{thm:CartanHelgason}
For each ${\bf m}\in\ZZ^r_+$ the space $\calP_{\bf m}(\Sigma)$ is an irreducible unitary $U$-representation and the space $L^2(\Sigma)$ decomposes into the direct Hilbert space sum
\begin{align*}
 L^2(\Sigma) &= \bigoplushat_{{\bf m}\in\ZZ^r_+}{\calP_{\bf m}(\Sigma)}.
\end{align*}
\end{theorem}

We now study the restriction of polynomials to the orbits $\calX_k\subseteq V_\CC$ and $\calO_k\subseteq V$. Note that for ${\bf m}\geq0$ and $k\in\{0,\ldots,r-1\}$ the condition $m_{k+1}=0$ means $m_{k+1}=\ldots=m_r=0$. For convenience we also use this notation for $k=r$. In this case $m_{k+1}=0$ should mean no restriction on ${\bf m}$.

\begin{proposition}[{\cite[Proposition I.7~(vi)~(a)]{HN01}}]
For $x\in\calX_k$ we have $\Delta_{k+1}(x)=\ldots=\Delta_r(x)=0$. In particular, $\Delta_{\bf m}$ vanishes on $\calX_k$ iff $m_{k+1}\neq0$, ${\bf m}\geq0$.
\end{proposition}

For $0\leq k\leq r$ let $\calP(\calX_k)$ and $\calP_{\bf m}(\calX_k)$ be the images of $\calP(V_\CC)$ and $\calP_{\bf m}(V_\CC)$ under the restriction map $C^\infty(V_\CC)\to C^\infty(\calX_k)$. Similar we define $\calP(\calO_k)$ and $\calP_{\bf m}(\calO_k)$ to be the images of $\calP(V_\CC)$ and $\calP_{\bf m}(V_\CC)$ under the restriction map $C^\infty(V_\CC)\to C^\infty(\calO_k)$.

\begin{corollary}\label{cor:HuaKostantSchmid}
Let $k\in\{0,\ldots,r\}$ and ${\bf m}\geq0$. Then $\calP_{\bf m}(\calX_k)$ (resp. $\calP_{\bf m}(\calO_k)$) is non-trivial if and only if $m_{k+1}=0$. In this case the restriction from $V_\CC$ to $\calX_k$ (resp. $\calO_k$) induces an isomorphism $\calP_{\bf m}(V_\CC)\cong\calP_{\bf m}(\calX_k)$ (resp. $\calP_{\bf m}(V_\CC)\cong\calP_{\bf m}(\calO_k)$) of $L$-modules (resp. $L_\CC$-modules). In particular we have the following decompositions:
\begin{align*}
 \calP(\calX_k) &= \bigoplus_{{\bf m}\geq0,\,m_{k+1}=0}{\calP_{\bf m}(\calX_k)}, &
 \calP(\calO_k) &= \bigoplus_{{\bf m}\geq0,\,m_{k+1}=0}{\calP_{\bf m}(\calO_k)}.
\end{align*}
\end{corollary}

For $\lambda\in\CC$ and ${\bf m}\geq0$ we define the \textit{Pochhammer symbol} $(\lambda)_{\bf m}$ by
\begin{align*}
 (\lambda)_{\bf m} &:= \prod_{i=1}^r{\left(\lambda-(i-1)\frac{d}{2}\right)_{m_i}},
\end{align*}
where $(a)_n=a(a+1)\cdots(a+n-1)$ denotes the classical Pochhammer symbol for $a\in\CC$, $n\in\NN_0$.

\begin{lemma}[{\cite[Lemma XI.2.3]{FK94}}]\label{lem:LaplaceTransformPolynomials}
For $p\in\calP_{\bf m}(V_\CC)$ and $\lambda\in\calW$ we have
\begin{align*}
 \int_{\calO_\lambda}{e^{-(x|y)}p(x)\td\mu_\lambda(x)} &= 2^{r\lambda}(\lambda)_{\bf m}\Delta(y)^{-\lambda}p(y^{-1}), & y\in\Omega.
\end{align*}
\end{lemma}

\subsubsection*{Spherical polynomials}\label{sec:SphericalPolynomials}

The representations $\calP_{\bf m}(V_\CC)$ of $L$ (resp. $U$) are $K^L$-spherical. The $K^L$-spherical vectors in $\calP_{\bf m}(V_\CC)$ are spanned by (see \cite[Proposition XI.3.1]{FK94})
\begin{align*}
 \Phi_{\bf m}(z) &= \int_{K^L}{\Delta_{\bf m}(kz)\td k}, & z\in V_\CC.
\end{align*}
The $L^2(\Sigma)$-norm of these functions are given by (see \cite[Proposition XI.4.1~(i)]{FK94})
\begin{align}
 \|\Phi_{\bf m}\|_\Sigma^2 &= \frac{1}{d_{\bf m}}, & {\bf m}\geq0.\label{eq:L2SigmaNormOfPhim}
\end{align}

By \cite[Corollary XI.3.4]{FK94} there exists a unique polynomial $\Phi_{\bf m}(z,w)$ on $V_\CC\times V_\CC$ holomorphic in $z$ and antiholomorphic in $w$ such that
\begin{align*}
 \Phi_{\bf m}(gz,w) &= \Phi_{\bf m}(z,g^*w), && \forall\,g\in L_\CC,\\
 \Phi_{\bf m}(x,x) &= \Phi_{\bf m}(x^2), && \forall\,x\in V.
\end{align*}

\begin{lemma}\label{lem:PropertiesPhimDoubleVariable}
\begin{enumerate}
\item For all $z\in V_\CC$ we have
\begin{align*}
 \Phi_{\bf m}(z,e) = \Phi_{\bf m}(z).
\end{align*}
\item For $x\in V$ and $y\in\overline{\Omega}$ we have
\begin{align*}
 \Phi_{\bf m}(x,y) &= \Phi_{\bf m}(P(y^{\frac{1}{2}})x).
\end{align*}
\item For $z,w\in V_\CC$ we have
\begin{align*}
 \overline{\Phi_{\bf m}(z,w)} &= \Phi_{\bf m}(w,z).
\end{align*}
\end{enumerate}
\end{lemma}

\begin{proof}
\begin{enumerate}
\item First let $x\in\Omega$. Recall the complex powers $x^s$, $s\in\CC$, defined in Section \ref{sec:StructureGroup}. Then
\begin{align*}
 \Phi_{\bf m}(x,e) &= \Phi_{\bf m}(P(x^{\frac{1}{4}})x^{\frac{1}{2}},e) = \Phi_{\bf m}(x^{\frac{1}{2}},P(x^{\frac{1}{4}})e)\\
 &= \Phi_{\bf m}(x^{\frac{1}{2}},x^{\frac{1}{2}}) = \Phi_{\bf m}(x).
\end{align*}
Since both sides are holomorphic in $x$ and $\Omega\subseteq V_\CC$ is totally real, this also holds for $x\in V_\CC$.
\item Let $x\in V$ and $y\in\Omega$. Then by (1) we obtain
\begin{equation*}
 \Phi_{\bf m}(x,y) = \Phi_{\bf m}(x,P(y^{\frac{1}{2}})e) = \Phi_{\bf m}(P(y^{\frac{1}{2}})x,e) = \Phi_{\bf m}(P(y^{\frac{1}{2}})x).
\end{equation*}
Now both sides are continuous in $y\in\overline{\Omega}$ and the claim follows.
\item First let $x,y\in\Omega$. Then by \cite[Lemma XIV.1.2]{FK94} there exists $k\in K^L$ such that $P(y^{\frac{1}{2}})x=kP(x^{\frac{1}{2}})y$. Using (2) and the $K^L$-invariance of $\Phi_{\bf m}$ we find that
\begin{equation*}
 \Phi_{\bf m}(x,y) = \Phi_{\bf m}(P(y^{\frac{1}{2}})x) = \Phi_{\bf m}(P(x^{\frac{1}{2}})y) = \Phi_{\bf m}(y,x).
\end{equation*}
Since $\Phi_{\bf m}(x,y)\in\RR$ for $x,y\in V$, we obtain
\begin{align*}
 \overline{\Phi_{\bf m}(z,w)} &= \Phi_{\bf m}(w,z) & \forall\,z,w\in\Omega.
\end{align*}
Both sides are holomorphic in $z$ and antiholomorphic in $w$ and $\Omega\subseteq V_\CC$ is totally real. Hence, the formula also holds for $z,w\in V_\CC$.\qedhere
\end{enumerate}
\end{proof}

\begin{lemma}\label{lem:PhimVanishing}
If ${\bf m}\geq0$ with $m_{k+1}\neq0$, then $\Phi_{\bf m}(z,w)=0$ for all $z\in V_\CC$ and $w\in\overline{\calX_k}$.
\end{lemma}

\begin{proof}
Write $w=he_k$ with $g\in L_\CC$. Then
\begin{align*}
 \Phi_{\bf m}(z,w) &= \Phi_{\bf m}(z,ge_k) = \Phi_{\bf m}(g^*z,e_k).
\end{align*}
Therefore, it suffices to show that $\Phi_{\bf m}(-,e_k)=0$ as a polynomial on $V_\CC$. Since $\Phi_{\bf m}(z,w)$ is holomorphic in the first variable it suffices to show that $\Phi_{\bf m}(-,e_k)=0$ as a polynomial on $V$ and since $\Omega\subseteq V$ is open it is enough to prove $\Phi_{\bf m}(x,e_k)=0$ for $x\in\Omega$. But for $x\in\Omega$ we have by Lemma \ref{lem:PropertiesPhimDoubleVariable}~(2) and (3)
\begin{align*}
 \Phi_{\bf m}(x,e_k) &= \Phi_{\bf m}(P(x^{\frac{1}{2}})e_k).
\end{align*}
Now write $x^{\frac{1}{2}}=ge$ with $g\in L$, then
\begin{align*}
 P(x^{\frac{1}{2}})e_k &= P(ge)e_k = gP(e)g^\# e_k = gg^\# e_k
\end{align*}
and hence $P(x^{\frac{1}{2}})e_k\in\calO_k=L\cdot e_k$. But by Corollary \ref{cor:HuaKostantSchmid} we have $\Phi_{\bf m}|_{\calO_k}=0$ since $m_{k+1}\neq0$ and the proof is complete.
\end{proof}

\begin{example}\label{ex:SphericalPolynomialsRank1}
On the rank $1$ orbit $\calX_1$ the spherical polynomials $\Phi_{\bf m}(z,w)$ are non-zero if and only if ${\bf m}=(m,0,\ldots,0)$, $m\in\NN_0$, and in this case
\begin{align*}
 \Phi_{\bf m}(z,w) &= \frac{(\frac{n}{r})_m}{d_{\bf m}m!}(z|\overline{w})^m, & z,w\in\calX_1.
\end{align*}
(This can e.g. be derived from the expansion of $\tr(x)^k$ into the polynomials $\Phi_{\bf m}(x)$, see \cite[Chapter XI.5]{FK94}.)
\end{example}

\subsubsection*{The Fischer inner product}

We equip $\calP(V_\CC)$ with the \textit{Fischer inner product}
\begin{align*}
 [p,q] &:= \left.p\left(\frac{\partial}{\partial z}\right)\overline{q}(z)\right|_{z=0}, & p,q\in\calP(V_\CC),
\end{align*}
where $\overline{q}(z):=\overline{q(\overline{z})}$, $z\in V_\CC$. The action of $U$ on $\calP(V_\CC)$ is unitary with respect to this inner product and hence the irreducible constituents $\calP_{\bf m}(V_\CC)$ are pairwise orthogonal. Since $U$ also acts unitarily on $\calP(V_\CC)$ with respect to the inner product on $L^2(\Sigma)$ the two inner products are proportional on each irreducible constituent $\calP_{\bf m}(V_\CC)$ (see \cite[Corollary XI.4.2]{FK94}):
\begin{align}
 [p,q] &= \left(\frac{n}{r}\right)_{\bf m}\langle p,q\rangle_\Sigma, & p,q\in\calP_{\bf m}(V_\CC).\label{eq:RelFischerSigma}
\end{align}
Denote by $K^{\bf m}(z,w)$ the reproducing kernel of $\calP_{\bf m}(V_\CC)$ with respect to the Fischer inner product. By \cite[Propositions XI.3.3 \&\ XI.4.1]{FK94} we have
\begin{align*}
 K^{\bf m}(z,w) &= \frac{d_{\bf m}}{(\frac{n}{r})_{\bf m}}\Phi_{\bf m}(z,w), & z,w\in V_\CC.
\end{align*}
Since by \cite[Proposition XI.1.1]{FK94} the completion of $\calP(V_\CC)$ with respect to the Fischer inner product is a Hilbert space with reproducing kernel $K(z,w)=e^{(z|\overline{w})}$ we obtain the following expansion:
\begin{align}
 e^{(z|\overline{w})} &= \sum_{{\bf m}\geq0}{\frac{d_{\bf m}}{(\frac{n}{r})_{\bf m}}\Phi_{\bf m}(z,w)}, & z,w\in V_\CC.\label{eq:ExponentialTraceExpansion}
\end{align}

\begin{lemma}\label{lem:SchurOrthogonalityLemma}
For ${\bf m},{\bf n}\geq0$ and $p\in\calP_{\bf m}(V_\CC)$, $q\in\calP_{\bf n}(V_\CC)$:
\begin{align*}
 \int_U{p(uz)\overline{q(uw)}\td u} &= \Phi_{\bf m}(z,w)\langle p,q\rangle_\Sigma & \forall\,z,w\in V_\CC.
\end{align*}
\end{lemma}

\begin{proof}
Using the reproducing kernel property and the Schur orthogonality relations we obtain
\begin{align*}
 \int_U{p(uz)\overline{q(uw)}\td u} &= \int_U{[p,K^{\bf m}(-,uz)]\overline{[q,K^{\bf n}(-,uw)]}\td u}\\
 &= \int_U{[\ell(u^{-1})p,K^{\bf m}(-,z)]\overline{[\ell(u^{-1})q,K^{\bf n}(-,w)]}\td u}\\
 &= \frac{1}{d_{\bf m}}[p,q]\overline{[K^{\bf m}(-,z),K^{\bf n}(-,w)]}.
\end{align*}
With \eqref{eq:RelFischerSigma} and
\begin{align*}
 \overline{[K^{\bf m}(-,z),K^{\bf n}(-,w)]} &= \delta_{\bf mn}\overline{K^{\bf m}(w,z)} = \delta_{\bf mn}\frac{d_{\bf m}}{(\frac{n}{r})_{\bf m}}\Phi_{\bf m}(z,w)
\end{align*}
the claimed identity follows.
\end{proof}

\subsubsection*{Laguerre functions}

For ${\bf m}\geq0$ the polynomial $\Phi_{\bf m}(e+x)$ is $K^L$-invariant of degree $|{\bf m}|$ and can hence be written as a linear combination of the $K^L$-invariant polynomials $\Phi_{\bf n}(x)$ for $|{\bf n}|\leq|{\bf m}|$. Following \cite[Chapter XV.4]{FK94} we define the \textit{generalized binomial coefficients} ${\bf m\choose n}$ by the formula
\begin{align*}
 \Phi_{\bf m}(e+x) &= \sum_{|{\bf n}|\leq|{\bf m}|}{{\bf m\choose n}\Phi_{\bf n}(x)}, & x\in V.
\end{align*}
Using the generalized binomial coefficients we define the \textit{generalized Laguerre polynomials} $L_{\bf m}^\lambda(x)$ by
\begin{align}
 L_{\bf m}^\lambda(x) &:= (\lambda)_{\bf m}\sum_{|{\bf n}|\leq|{\bf m}|}{{\bf m\choose n}\frac{1}{(\lambda)_{\bf n}}\Phi_{\bf n}(-x)}, & x\in V,\label{eq:DefLagPoly}\\
\intertext{and the \textit{generalized Laguerre functions} $\ell_{\bf m}^\lambda(x)$ by}
 \ell_{\bf m}^\lambda(x) &:= e^{-\tr(x)}L_{\bf m}^\lambda(2x), & x\in V.\label{eq:DefLagFunc}
\end{align}
Both $L_{\bf m}^\lambda(x)$ and $\ell_{\bf m}^\lambda(x)$ are $K^L$-invariant. Note that for $\lambda>(r-1)\frac{d}{2}$ we have $(\lambda)_{\bf n}\neq0$ for all ${\bf n}\geq0$ and hence $L_{\bf m}^\lambda(x)$ and $\ell_{\bf m}^\lambda(x)$ are defined on $V$. For $x\in\overline{\calO_k}$, $k=0,\ldots,r-1$, we further have $\Phi_{\bf n}(x)=0$ if $n_{k+1}\neq0$. Therefore the sum in \eqref{eq:DefLagPoly} reduces to a sum over ${\bf n}$ with $n_{k+1}=\ldots=n_r=0$. For such ${\bf n}$ we have $(\lambda)_{\bf n}\neq0$ for $\lambda=k\frac{d}{2}$ and hence the expression \eqref{eq:DefLagPoly} is defined. This gives Laguerre polynomials $L_{\bf m}^\lambda(x)$ and Laguerre functions $\ell_{\bf m}^\lambda(x)$ on $\overline{\calO_\lambda}$ for each $\lambda\in\calW$.

Properties of the generalized Laguerre functions have been studied in \cite{ADO06,DOZ03}.

\newpage
\section{Three different realizations of unitary highest weight representations of scalar type}\label{sec:RealizationsUHWR}

We describe three known realizations of highest weight representations of scalar type: the \textit{bounded symmetric domain model}, the \textit{tube domain model} and the \textit{Schr\"odinger model}. We further discuss a complexification of the Schr\"odinger model which we use in Section \ref{sec:ActionOnFockSpace} to construct yet another model, the Fock model.

\subsection{The Schr\"odinger model}\label{sec:SchrödingerModel}

The highest weight representation of the universal cover $\widetilde{G}$ belonging to the Wallach point $\lambda\in\calW$ can be realized on $L^2(\calO_\lambda,\td\mu_\lambda)$. We sketch the construction here (see e.g. \cite{Moe10} for details). First, there is a Lie algebra representation $\td\pi_\lambda$ of $\frakg$ on $C^\infty(\calO_\lambda)$ for every $\lambda\in\calW$ given by
\begin{align*}
 \td\pi_\lambda(u,0,0) &= i(u|x),\\
 \td\pi_\lambda(0,T,0) &= \partial_{T^\# x}+\frac{r\lambda}{2n}\Tr(T^\#),\\
 \td\pi_\lambda(0,0,v) &= i(v|\calB_\lambda).
\end{align*}
Note that for $\lambda=k\frac{d}{2}$, $k=0,\ldots,r-1$, the Bessel operator $\calB_\lambda$ is tangential to the orbit $\calO_k$ and hence defines a differential operator acting on $C^\infty(\calO_k)$ (see Section \ref{sec:DiffOperators}). The representation $\td\pi_\lambda$ is further infinitesimally unitary with respect to the $L^2$-inner product on $L^2(\calO_\lambda,\td\mu_\lambda)$.

The subrepresentation of $(\td\pi_\lambda,C^\infty(\calO_\lambda))$ generated by the function
\begin{align*}
 \psi_0(x) &= e^{-\tr(x)}, & x\in\calO_\lambda,
\end{align*}
defines a $(\frakg,\widetilde{K})$-module $(\td\pi_\lambda,W^\lambda)$ whose underlying vector space turns out to be (see e.g. \cite[Proposition XIII.3.2]{FK94})
\begin{align*}
 W^\lambda &= \calP(\calO_\lambda)e^{-\tr(x)}.
\end{align*}
The $(\frakg,\widetilde{K})$-module $(\td\pi_\lambda,W^\lambda)$ integrates to an irreducible unitary representation $(\pi_\lambda,L^2(\calO_\lambda,\td\mu_\lambda))$ of $\widetilde{G}$. The minimal $\widetilde{K}$-type is spanned by the function $\psi_0$ which is $\widetilde{K}$-equivariant:
\begin{align*}
 \pi_\lambda(k)\psi_0 &= \xi_\lambda(k)\psi_0, & k\in\widetilde{K},
\end{align*}
where $\xi_\lambda:\widetilde{K}\to\TT$ is the character of $\widetilde{K}$ with differential $\td\xi_\lambda(u,T,-u)=i\lambda\tr(u)$. This implies that the representation $\pi_\lambda$ descends to a finite cover of $G$ if and only if $\lambda\in\QQ$ (which holds in particular for $\lambda\in\calW_{\disc}$). For $\lambda>\frac{2n}{r}-1$ the representation $\pi_\lambda$ belongs to the relative holomorphic discrete series of $\widetilde{G}$.

\subsection{The tube domain model}\label{sec:TubeDomainModel}

For $\lambda\in\CC$ consider the function
\begin{align*}
 T_\Omega\times T_\Omega\to\CC,\,(z,w)\mapsto\Delta\left(\frac{z-\overline{w}}{2i}\right)^{-\lambda}.
\end{align*}
It is of positive type if and only if $\lambda\in\calW$ (see \cite[Proposition XIII.1.2 \&\ Theorem XIII.2.4]{FK94}). Denote by $\calH^2_\lambda(T_\Omega)$ the Hilbert space of holomorphic functions on $T_\Omega$ which has the function $\Delta(\frac{z-\overline{w}}{2i})^{-\lambda}$ as reproducing kernel. For $\lambda>1+(r-1)d$ this space coincides with the space of holomorphic functions $F$ on $T_\Omega$ such that
\begin{align*}
 \int_{T_\Omega}{|F(z)|^2\Delta(y)^{\lambda-\frac{2n}{r}}\td x\td y} < \infty,
\end{align*}
where $z=x+iy\in T_\Omega$ (see \cite[Chapter XIII.1]{FK94}). For every $\lambda\in\calW$ there is an irreducible unitary representation $\pi_\lambda^{T_\Omega}$ of $\widetilde{G}$ on $\calH_\lambda^2(T_\Omega)$. Note that $\widetilde{G}$ acts on $T_\Omega$ by composition of the action of $G=\Aut(T_\Omega)_0$ on $T_\Omega$ with the covering map $\widetilde{G}\to G$. Then the representation $\pi_\lambda^{T_\Omega}$ is given by
\begin{align*}
 \pi_\lambda^{T_\Omega}(g)F(z) &= \mu_\lambda^{T_\Omega}(g^{-1},z)F(g^{-1}z), & g\in\widetilde{G},z\in T_\Omega,
\end{align*}
where the cocycle $\mu_\lambda^{T_\Omega}(g,z)$ is given by
\begin{align*}
 \mu_\lambda^{T_\Omega}(g,z) &= \Det(Dg(z))^{\frac{r\lambda}{2n}}, & g\in\widetilde{G},z\in T_\Omega,
\end{align*}
the powers being well-defined on the universal cover $\widetilde{G}$.

The representations $\calH^2_\lambda(T_\Omega)$ and $L^2(\calO_\lambda,\td\mu_\lambda)$ are isomorphic, the intertwining operator being the \textit{Laplace transform} (see \cite[Theorems XIII.1.1 \&\ XIII.3.4]{FK94})
\begin{align}
 \calL_\lambda:L^2(\calO_\lambda,\td\mu_\lambda)\to\calH^2_\lambda(T_\Omega),\,\calL_\lambda\psi(z) &:= \int_{\calO_\lambda}{e^{i(z|x)}\psi(x)\td\mu_\lambda(x)}.\label{eq:DefLaplaceTrafo}
\end{align}
The Laplace transform $\calL_\lambda$ is a unitary isomorphism (up to a scalar) intertwining the group actions.

\subsection{The bounded symmetric domain model}\label{sec:BoundedSymmetricDomainModel}

The polynomial $h(x):=\Delta(e-x^2)$, $x\in V$, is $K^L$-invariant and therefore by \cite[Corollary XI.3.4]{FK94} there exists a unique polynomial $h(z,w)$ holomorphic in $z\in V_\CC$ and antiholomorphic in $w\in V_\CC$ such that
\begin{align*}
 h(gz,w) &= h(z,g^*w), & \forall\,g\in L_\CC,\\
 h(x,x) &= h(x), & \forall\,x\in V.
\end{align*}
We also write $h(z):=h(z,z)$ for $z\in V_\CC$. Consider the powers $h(z,w)^{-\lambda}$ for $\lambda\in\CC$ as functions on $\calD\times\calD$. Then $h(z,w)^{-\lambda}$ is positive definite on $\calD\times\calD$ if and only if $\lambda\in\calW$ (see \cite[Theorem XIII.2.7]{FK94}). The corresponding Hilbert space of holomorphic functions on $\calD$ with reproducing kernel $h(z,w)^{-\lambda}$ will be denoted by $\calH^2_\lambda(\calD)$. For $\lambda>1+(r-1)d$ this space coincides with the space of holomorphic functions $f$ on $\calD$ such that
\begin{align*}
 \int_\calD{|f(w)|^2h(w)^{\lambda-\frac{2n}{r}}\td u\td v} < \infty,
\end{align*}
where $w=u+iv\in\calD$ (see \cite[Proposition XIII.1.4]{FK94}). For each $\lambda\in\calW$ there is an irreducible unitary representation $\pi_\lambda^\calD$ of $\widetilde{G}$ on $\calH^2_\lambda(\calD)$. To give an explicit formula recall the isomorphism $\alpha:G\to\Aut(\calD)_0$ defined in \eqref{eq:DefinitionAlpha} and view it as a covering map $\alpha:\widetilde{G}\to\Aut(\calD)_0$. Then the representation $\pi_\lambda^\calD$ is given by
\begin{align*}
 \pi_\lambda^\calD(g)f(w) &= \mu_\lambda^\calD(g^{-1},w)f(\alpha(g)^{-1}w), & g\in\widetilde{G},w\in\calD,
\end{align*}
where the cocycle $\mu_\lambda^\calD(g,w)$ is given by
\begin{align*}
 \mu_\lambda^\calD(g,w) &= \Det(D(\alpha(g))(w))^{\frac{r\lambda}{2n}}, & g\in\widetilde{G},w\in\calD,
\end{align*}
the powers being well-defined on the universal cover $\widetilde{G}$.

The representations $\calH^2_\lambda(\calD)$ and $\calH^2_\lambda(T_\Omega)$ are isomorphic, the intertwining operator being given by
\begin{align*}
 \gamma_\lambda:\calH^2_\lambda(T_\Omega)\to\calH^2_\lambda(\calD),\,\gamma_\lambda F(w):=\Delta(e-w)^{-\lambda}F(c(w)),
\end{align*}
where $c(w)$ is the Cayley transform defined in \eqref{eq:DefCayleyTransform} (see \cite[Proposition XIII.1.3 \&\ Theorem XIII.3.4]{FK94}). The operator $\gamma_\lambda$ is unitary (up to a scalar) and intertwines the group actions.

\subsection{$\frakk$-type decompositions}\label{sec:KTypeDecompositions}

In the bounded symmetric domain model the $\widetilde{K}$-type decomposition is very explicit. Let $\lambda\in\calW$ and $0\leq k\leq r$ such that $\calO_\lambda=\calO_k$. Then $\calH^2_\lambda(\calD)$ decomposes into the direct Hilbert space sum
\begin{align*}
 \calH^2_\lambda(\calD) &= \bigoplushat_{{\bf m}\geq0,\,m_{k+1}=0}{\,\calP_{\bf m}(V_\CC)},
\end{align*}
each summand $\calP_{\bf m}(V_\CC)$ is irreducible under the action $\pi^\calD_\lambda(\widetilde{K})$ and on it the norm is given by (see \cite[Theorem XIII.2.7]{FK94})
\begin{align}
 \|p\|_{\calH^2_\lambda(\calD)}^2 &= \frac{\left(\frac{n}{r}\right)_{\bf m}}{(\lambda)_{\bf m}}\|p\|_\Sigma^2, & p\in\calP_{\bf m}(V_\CC).\label{eq:NormOnBdSymDomainModel}
\end{align}
Further, the $\frakk^\frakl$-spherical vector in each $\widetilde{K}$-type $\calP_{\bf m}(V_\CC)$ is the spherical polynomial $\Phi_{\bf m}(z)$ which has norm
\begin{align*}
 \|\Phi_{\bf m}\|^2_{\calH^2_\lambda(\calD)} &= \frac{(\frac{n}{r})_{\bf m}}{d_{\bf m}(\lambda)_{\bf m}}.
\end{align*}

Correspondingly the underlying $(\frakg,\widetilde{K})$-module $W^\lambda$ of the Schr\"odinger model $L^2(\calO_\lambda,\td\mu_\lambda)$ decomposes into $\widetilde{K}$-types
\begin{align*}
 W^\lambda &= \bigoplus_{{\bf m\geq0},\,m_{k+1}=0}{W^\lambda_{\bf m}},
\end{align*}
where $W^\lambda_{\bf m}=\calL_\lambda^{-1}\circ\gamma_\lambda^{-1}(\calP_{\bf m}(V_\CC))$. The $\frakk^\frakl$-spherical vector in $W^\lambda_{\bf m}$ is given by the Laguerre function $\ell_{\bf m}^\lambda(x)$ defined in Section \ref{sec:Polynomials} (see \cite[Proposition XV.4.2]{FK94}) which has norm (see \cite[Corollary XV.4.3~(i)]{FK94})
\begin{align}
 \|\ell_{\bf m}^\lambda\|_{L^2(\calO_\lambda,\td\mu_\lambda)}^2 &= \frac{(\frac{n}{r})_{\bf m}(\lambda)_{\bf m}}{d_{\bf m}}.\label{eq:NormLaguerreFunctions}
\end{align}

\subsection{Complexification of the Schr\"odinger model}\label{sec:ComplexificationSchrödinger}

The infinitesimal action $\td\pi_\lambda$ in the Schr\"odinger model is given by second order differential operators on $\calO_\lambda$ with polynomial coefficients. Hence the action can be extended to an action $\td\pi_\lambda^\CC$ of $\frakg$ on $C^\infty(\calX_\lambda)$ by holomorphic differential operators. More precisely, let $D$ be a differential operator on $V$ with polynomial coefficients. Choose some basis $e_1,\ldots,e_n$ of $V$ and write $x=\sum_{j=1}^n{x_je_j}\in V$. Then $D$ is in coordinates given by
\begin{align*}
 D = \sum_{\alpha\in\NN_0^n}{c_\alpha(x)\frac{\partial^{|\alpha|}}{\partial x_1^{\alpha_1}\cdots\partial x_n^{\alpha_n}}}
\end{align*}
with polynomials $c_\alpha(x)$ of which only finitely many are non-zero. The coefficients extend uniquely to holomorphic polynomials $c_\alpha(z)$ on $V_\CC$. We define the complexification $D^\CC$ of $D$ to a holomorphic differential operator on $V_\CC$ by
\begin{align*}
 D^\CC &:= \sum_{\alpha\in\NN_0^n}{c_\alpha(z)\frac{\partial^{|\alpha|}}{\partial z_1^{\alpha_1}\cdots\partial z_n^{\alpha_n}}},
\end{align*}
where
\begin{align*}
 \frac{\partial}{\partial z_j} &= \frac{1}{2}\left(\frac{\partial}{\partial x_j}-i\frac{\partial}{\partial y_j}\right)
\end{align*}
denotes the \textit{Wirtinger derivative} for $z_j=x_j+iy_j$ in coordinates $z=\sum_{j=1}^n{z_je_j}\in V_\CC$. The operator $D^\CC$ is a complexification of $D$ in the sense that for a holomorphic function $f$ on $V_\CC$ we have
\begin{align*}
 (D^\CC f)|_V &= D(f|_V).
\end{align*}
Applying this procedure to $\td\pi_\lambda(X)$ for $X\in\frakg$ we put
\begin{align*}
 \td\pi_\lambda^\CC(X) &:= \td\pi_\lambda(X)^\CC.
\end{align*}
It remains to show that these holomorphic differential operators are actually tangential to the orbit $\calX_\lambda$ and hence define an action $\td\pi_\lambda^\CC$ of $\frakg$ on $C^\infty(\calX_\lambda)$. For this we use the Schr\"odinger model of certain unipotent representations of the complex group $G_\CC=\Co(V_\CC)$ (the conformal group of the complex Jordan algebra $V_\CC$), viewed as a real Lie group.

In \cite[Proposition 2.14]{HKM12} the authors construct a representation $\td\tau_\lambda$ of $\frakg_\CC$, viewed as a real Lie algebra, on $C^\infty(\calX_\lambda)$ for each $\lambda\in\calW$. It is explicitly given by
\begin{align*}
 \td\tau_\lambda(u,0,0) &= i(u|z)_W,\\
 \td\tau_\lambda(0,T,0) &= \partial_{T^\# z}+\frac{r\lambda}{2n}\Tr_W(T^\#),\\
 \td\tau_\lambda(0,0,v) &= i(v|\calB_\lambda^W)_W.
\end{align*}
By $\Tr_W$ we mean the real trace of an operator on the real vector space $W$. Note that $\td\tau_\lambda$ does not act via holomorphic differential operators, but via real differential operators up to second order on $\calX_\lambda$. Further, its is shown in \cite[Theorem 1.12]{HKM12} that the representation $\td\tau_\lambda$ is infinitesimally unitary with respect to the inner product of $L^2(\calX_\lambda,\td\nu_\lambda)$.

We have the following result relating $\td\tau_\lambda$ to the complexification $\td\pi_\lambda^\CC$:

\begin{proposition}\label{prop:PiCC}
For $X\in\frakg$ we have
\begin{align}
 \td\pi_\lambda^\CC(X) &= \frac{1}{2}\left(\td\tau_\lambda(X)-i\td\tau_\lambda(iX)\right).\label{eq:ComplexifiedRep}
\end{align}
In particular, for every $X=(u,T,v)\in\frakg$ and all $F,G\in C^\infty(\calX_\lambda)$ we have
\begin{equation*}
 \int_{\calX_\lambda}{\td\pi_\lambda^\CC(u,T,v)F(z)\cdot G(z)\td\nu_\lambda(z)} = \int_{\calX_\lambda}{F(z)\cdot\td\pi_\lambda^\CC(u,-T,v)G(z)\td\nu_\lambda(z)}.
\end{equation*}
\end{proposition}

\begin{remark}
The formula \eqref{eq:ComplexifiedRep} can be understood as an analog of the Wirtinger derivative
\begin{align*}
 \frac{\partial}{\partial z} &= \frac{1}{2}\left(\frac{\partial}{\partial x}-i\frac{\partial}{\partial y}\right).
\end{align*}
\end{remark}

\begin{proof}[{Proof of Proposition \ref{prop:PiCC}}]
First note that since $\frakn$ and $\overline{\frakn}$ together generate $\frakg$ as a Lie algebra it suffices to show \eqref{eq:ComplexifiedRep} for $X\in\frakn$ and $X\in\overline{\frakn}$. Further note that for $z\in V_\CC$ we have
\begin{align}
 (a|z) &= \frac{1}{2}\left((a|z)_W-i(ia|z)_W\right), & a\in V.\label{eq:WirtingerForInnerProduct}
\end{align}
Then \eqref{eq:ComplexifiedRep} is immediate for $X\in\frakn$. Now let $X=(0,0,a)\in\overline{\frakn}$. Let $(e_\alpha)$ be any orthonormal basis of $V$ with respect to the trace form $(-|-)$. Write $x=\sum_\alpha{x_\alpha e_\alpha}$. We view $W=V_\CC$ as a real Jordan algebra. The vectors
\begin{align*}
 f_\alpha &:= \frac{1}{\sqrt{2}}e_\alpha & \mbox{and} && g_\alpha &:= \frac{1}{\sqrt{2}}ie_\alpha
\end{align*}
constitute an $\RR$-basis of $W$. Its dual basis with respect to the trace form $(-|-)_W$ is given by $(\widehat{f}_\alpha:=f_\alpha)_\alpha\cup(\widehat{g}_\alpha:=-g_\alpha)_\alpha$. We write $z=\sum_\alpha{z_\alpha e_\alpha}=\sum_\alpha{(a_\alpha f_\alpha+b_\alpha g_\alpha)}$ with $a_\alpha,b_\alpha\in\RR$ and $z_\alpha=\frac{1}{\sqrt{2}}(a_\alpha+ib_\alpha)$. Hence,
\begin{align*}
 \frac{\partial}{\partial a_\alpha} &= \frac{1}{\sqrt{2}}\frac{\partial}{\partial x_\alpha} & \mbox{and} && \frac{\partial}{\partial b_\alpha} &= \frac{1}{\sqrt{2}}\frac{\partial}{\partial y_\alpha}
\end{align*}
with $z_\alpha=x_\alpha+iy_\alpha$, $x_\alpha,y_\alpha\in\RR$. With \eqref{eq:WirtingerForInnerProduct} we have
\begin{align*}
 & \frac{1}{2}\left(\td\tau_\lambda(X)-i\td\tau_\lambda(iX)\right) = \frac{i}{2}\left((a|\calB_\lambda^W)-i(ia|\calB_\lambda^W)\right)\\
 ={}& i\sum_{\alpha,\beta}{\Bigg(\frac{\partial^2}{\partial a_\alpha\partial a_\beta}(a|P(\widehat{f}_\alpha,\widehat{f}_\beta)z)+2\frac{\partial^2}{\partial a_\alpha\partial b_\beta}(a|P(\widehat{f}_\alpha,\widehat{g}_\beta)z)}\\
 & \hspace{1.7cm}{+\frac{\partial^2}{\partial b_\alpha\partial b_\beta}(a|P(\widehat{g}_\alpha,\widehat{g}_\beta)z)\Bigg)}+i\lambda\sum_\alpha{\Bigg(\frac{\partial}{\partial a_\alpha}\widehat{f}_\alpha+\frac{\partial}{\partial b_\alpha}\widehat{g}_\alpha\Bigg)}\\
 ={}& \frac{i}{4}\sum_{\alpha,\beta}{\Bigg(\frac{\partial^2}{\partial x_\alpha\partial x_\beta}(a|P(e_\alpha,e_\beta)z)-2i\frac{\partial^2}{\partial x_\alpha\partial y_\beta}(a|P(e_\alpha,e_\beta)z)}\\
 & \hspace{1.5cm}{-\frac{\partial^2}{\partial y_\alpha\partial y_\beta}(a|P(e_\alpha,e_\beta)z)\Bigg)}+\frac{i\lambda}{2}\sum_\alpha{\Bigg(\frac{\partial}{\partial x_\alpha}e_\alpha-i\frac{\partial}{\partial y_\alpha}e_\alpha\Bigg)}\\
 ={}& i\sum_{\alpha,\beta}{\frac{\partial^2}{\partial z_\alpha\partial z_\beta}(a|P(e_\alpha,e_\beta)z)}+i\lambda\sum_\alpha{\frac{\partial}{\partial z_\alpha}e_\alpha} = \td\pi_\lambda^\CC(X).
\end{align*}
This shows \eqref{eq:ComplexifiedRep} for $X\in\overline{\frakn}$ and hence it follows for all $X\in\frakg$.\\
The stated integral formula now follows from \eqref{eq:ComplexifiedRep} using the fact that $\td\tau_\lambda(X)$ is given by skew-adjoint real differential operators operators on $L^2(\calX_\lambda,\td\nu_\lambda)$ with real coefficients if $X=(0,T,0)\in\frakg$ and purely imaginary coefficients if $X=(u,0,v)\in\frakg$.
\end{proof}

Since $\td\tau_\lambda$ restricts to an action on $C^\infty(\calX_\lambda)$ by differential operators, the same is true for $\td\pi_\lambda^\CC$ by the previous proposition. Therefore, $\td\pi_\lambda^\CC$ is a representation of $\frakg$ on $C^\infty(\calX_\lambda)$ by holomorphic polynomial differential operators of order at most $2$.

\newpage
\section{Bessel functions on Jordan algebras}\label{sec:BesselFunctions}

In this section we study \textit{$J$-, $I$- and $K$-Bessel functions} on symmetric cones and their boundary orbits. These functions play a fundamental role in the study of Schr\"odinger and Fock models and the intertwining operators between them.

\subsection{$J$-Bessel function}\label{sec:JBessel}

For $\lambda\in\CC$ with $(\lambda)_{\bf m}\neq0$ for all ${\bf m}\geq0$ and $z,w\in V_\CC$ we put
\begin{align*}
 \calJ_\lambda(z,w) &:= \sum_{{\bf m}\geq0}{(-1)^{|{\bf m}|}\frac{d_{\bf m}}{(\frac{n}{r})_{\bf m}(\lambda)_{\bf m}}\Phi_{\bf m}(z,w)}.
\end{align*}
(This notation agrees with the one used in \cite[Chapter XV.2]{FK94}.) One problem is that for a discrete Wallach point $\lambda=k\frac{d}{2}$, $k=0,\ldots,r-1$, we have $(\lambda)_{\bf m}=0$ for all ${\bf m}\geq0$ with $m_{k+1}\neq0$. However, by Lemma \ref{lem:PhimVanishing} we find that for $z\in\overline{\calX_k}$ or $w\in\overline{\calX_k}$:
\begin{align*}
 \calJ_\lambda(z,w) &= \sum_{{\bf m}\geq0, m_{k+1}=0}{(-1)^{|{\bf m}|}\frac{d_{\bf m}}{(\frac{n}{r})_{\bf m}(\lambda)_{\bf m}}\Phi_{\bf m}(z,w)}.
\end{align*}
In this expression the coefficients are non-singular at $\lambda=k\frac{d}{2}$ since $(\lambda)_{\bf m}\neq0$ for all ${\bf m}\geq0$ with $m_{k+1}=0$. Therefore we obtain for each $\lambda\in\calW$ a $J$-Bessel function $\calJ_\lambda(z,w)$ on $V_\CC\times\overline{\calX_\lambda}$. It only remains to show convergence of the defining series.

\begin{lemma}\label{lem:JBesselConvergence}
For $\lambda\in\calW$ the series for $\calJ_\lambda(z,w)$ converges absolutely for all $z\in V_\CC$ and $w\in\overline{\calX_\lambda}$ and the following estimate holds:
\begin{align*}
 |\calJ_\lambda(z,w)| &\leq C(1+|z|\cdot|w|)^{\frac{r(2n-1)}{4}}e^{2r\sqrt{|z|\cdot|w|}} & \forall\,z\in V_\CC,w\in\overline{\calX_\lambda}
\end{align*}
for some constant $C>0$ which depends only on the structure constants of $V$ and $\lambda$.
\end{lemma}

\begin{proof}
Let $z,w\in V_\CC$. Then by \cite[Proposition X.3.2]{FK94} there exist $u\in U$ and $a=\sum_{j=1}^r{a_jc_j}$ with $a_1\geq\ldots\geq a_r\geq0$ such that $w=ua=uP(a^{\frac{1}{2}})e$. With Lemma \ref{lem:PropertiesPhimDoubleVariable} (2) we find
\begin{equation*}
 \Phi_{\bf m}(z,w) = \Phi_{\bf m}(z,uP(a^{\frac{1}{2}})e) = \Phi_{\bf m}(P(a^{\frac{1}{2}})u^{-1}z,e) = \Phi_{\bf m}(P(a^{\frac{1}{2}})u^{-1}z).
\end{equation*}
Now suppose further that $w\in\overline{\calX_k}$, $0\leq k\leq r$, and $m_{k+1}=0$. Then $a_{k+1}=\ldots=a_r=0$ and hence $P(a^{\frac{1}{2}})$ projects onto $V(e_k,1)\subseteq\overline{\calX_k}$. Thus we find that $P(a^{\frac{1}{2}})u^{-1}z\in\overline{\calX_k}$ and again by \cite[Proposition X.3.2]{FK94} we can write
\begin{align*}
 P(a^{\frac{1}{2}})u^{-1}z &= u'b,
\end{align*}
where $u'\in U$ and $b=\sum_{j=1}^k{b_jc_j}$, $b_1\geq\ldots\geq b_k\geq0$. Now, by \cite[Theorem XII.1.1~(i)]{FK94} we obtain
\begin{align*}
 |\Phi_{\bf m}(z,w)| &= |\Phi_{\bf m}(u'b)| \leq b_1^{m_1}\cdots b_k^{m_k}.
\end{align*}
We further have the following obvious inequalities (assuming $m_{k+1}=0$ for $\lambda=k\tfrac{d}{2}$)
\begin{align*}
 d_{\bf m} &\leq \dim\calP_{|{\bf m}|}(V_\CC) = {n+|{\bf m}|-1\choose n-1} \leq C_1(1+|{\bf m}|)^{n-1}\\
 &\leq C_1(1+m_1)^{n-1}\cdots(1+m_k)^{n-1},\\
 (\tfrac{n}{r})_{\bf m} &= \prod_{j=1}^r{(\tfrac{n}{r}-(j-1)\tfrac{d}{2})_{m_j}} \geq \prod_{j=1}^r{(1)_{m_j}} = {\bf m}!,\\
 (\lambda)_{\bf m} &= \prod_{j=1}^k{(\lambda-(j-1)\tfrac{d}{2})_{m_j}} \geq \prod_{j=1}^r{(\lambda-(k-1)\tfrac{d}{2})_{m_j}}.
\end{align*}
We abbreviate $\lambda':=\lambda-(k-1)\frac{d}{2}>0$. Putting things together gives
\begin{align*}
 |\calJ_\lambda(z,w)| &\leq \sum_{{\bf m}\geq0, m_{k+1}=0}{\frac{d_{\bf m}}{(\frac{n}{r})_{\bf m}(\lambda)_{\bf m}}|\Phi_{\bf m}(z,w)|}\\
 &\leq C_1\sum_{{\bf m}\in\NN_0^k}{\frac{(1+m_1)^{n-1}\cdots(1+m_k)^{n-1}}{{\bf m}!(\lambda')_{m_1}\cdots(\lambda')_{m_k}}b_1^{m_1}\cdots b_k^{m_k}}\\
 &= C_1\prod_{j=1}^k\left(\sum_{m=0}^\infty{\frac{(1+m)^{n-1}}{m!(\lambda')_m}b_j^m}\right).
\end{align*}
Now note that $mb^m=(b\frac{\td}{\td b})b^m$, so
\begin{align*}
 \sum_{m=0}^\infty{\frac{(1+m)^{n-1}}{m!(\lambda')_m}b^m} &= \left(1+b\frac{\td}{\td b}\right)^{n-1}\sum_{m=0}^\infty{\frac{1}{m!(\lambda')_m}b^m}\\
 &= \left(1+b\frac{\td}{\td b}\right)^{n-1}{_0F_1}(\lambda';b),
\end{align*}
where ${_0F_1}(\beta;z)$ denotes the hypergeometric function. For ${_0F_1}(\beta;z)$ we have the obvious identity
\begin{align*}
 \frac{\td}{\td z}{_0F_1}(\beta;z) &= \frac{1}{\beta}{_0F_1}(\beta+1;z)
\end{align*}
and hence $\left(1+b\frac{\td}{\td b}\right)^{n-1}{_0F_1}(\lambda';b)$ is a linear combination of functions of the type
\begin{align*}
 &b^k{_0F_1}(\lambda'+k;b), & k=0,\ldots,n-1.
\end{align*}
Now by \cite[equations (4.5.2) \&\ (4.8.5)]{AAR99} the asymptotic behaviour of the hypergeometric function as $z\to\infty$ can be estimated by
\begin{align*}
 |{_0F_1}(\beta;z)| &\lesssim |z|^{\frac{1-2\beta}{4}}e^{2|z|^{\frac{1}{2}}}
\end{align*}
and hence we obtain
\begin{align*}
 \left|\left(1+b\frac{\td}{\td b}\right)^{n-1}{_0F_1}(\lambda';b)\right| &\leq C_2(1+|b|)^{\frac{2n-2\lambda'-1}{4}}e^{2\sqrt{b}} & \forall\,b\in[0,\infty),
\end{align*}
the constant $C_2>0$ only depending on $n$ and $\lambda'$. Inserting this into the estimate for $\calJ_\lambda(z,w)$ gives
\begin{align*}
 |\calJ_\lambda(z,w)| &\leq C_1C_2^k\prod_{j=1}^k\left((1+|b_j|)^{\frac{2n-2\lambda'-1}{4}}e^{2\sqrt{b_j}}\right).
\end{align*}
Now note that
\begin{equation*}
 |b| = |u'b| = |P(a^{\frac{1}{2}})u^{-1}z| \leq \|P(a^{\frac{1}{2}})\|\cdot |u^{-1}z| = \|P(a^{\frac{1}{2}})\|\cdot |z|,
\end{equation*}
where $\|P(a^{\frac{1}{2}})\|^2$ is the largest eigenvalue of $P(a^{\frac{1}{2}})^\#P(a^{\frac{1}{2}})=P(a^{\frac{1}{2}})^2$. Since $P(a^{\frac{1}{2}})^2$ acts on $V_{ij}$ by $a_ia_j$, its largest eigenvalue is $a_1^2$. Hence
\begin{align*}
 |b| &\leq a_1|z| \leq |a|\cdot|z| = |w|\cdot|z|.
\end{align*}
Altogether we finally obtain
\begin{align*}
  |\calJ_\lambda(z,w)| &\leq C_1C_2^k\left(\prod_{j=1}^k{(1+|b_j|)}\right)^{\frac{2n-2\lambda'-1}{4}}e^{2(\sqrt{b_1}+\cdots+\sqrt{b_r})}\\
  &\leq C_1C_2^k(1+b_1+\cdots+b_r)^{\frac{r(2n-2\lambda'-1)}{4}}e^{2\sqrt{r}\sqrt{b_1+\cdots+b_r}}\\
  &\leq C_1C_2^kr^{\frac{r(2n-2\lambda'-1)}{8}}(1+|b|)^{\frac{r(2n-2\lambda'-1)}{4}}e^{2r\sqrt{|b|}}\\
  &\leq C(1+|z|\cdot|w|)^{\frac{r(2n-1)}{4}}e^{2r\sqrt{|z|\cdot|w|}}
\end{align*}
with $C=C_1C_2^kr^{\frac{r(2n-1)}{8}}>0$ which gives the claim.
\end{proof}

The estimate obtained in Lemma \ref{lem:JBesselConvergence} is not sharp, but suffices for most of our purposes. Recently, Nakahama~\cite{Nak12} obtained a sharper estimate which we use in Section~\ref{sec:WhittakerVectors} to find explicit Whittaker vectors:

\begin{proposition}[{\cite[Corollary 1.2]{Nak12}}]\label{prop:JBesselSharpEstimate}
For $\lambda\in\calW$ and $k\in\NN_0$ with $\Re\lambda+k>\frac{2n}{r}-1$ there exists a constant $C_{\lambda,k}>0$ such that
\begin{align*}
 |\calJ_\lambda(z^2,e)| &\leq C_{\lambda,k}(1+|z|_1^k)e^{2|\Im z|_1}, & z\in\overline{\calX_\lambda},
\end{align*}
where $|z|_1=\sum_{j=1}^r{|a_j|}$, $z=u\sum_{j=1}^r{a_jc_j}$, $u\in U$, $a_j\in\RR$.
\end{proposition}

\begin{example}\label{ex:JBesselRank1}
On the rank $1$ orbit we have, using Lemma \ref{lem:PhimVanishing} and Example \ref{ex:SphericalPolynomialsRank1}
\begin{equation*}
 \calJ_\lambda(z,w) = \sum_{m=0}^\infty{\frac{(-1)^m}{m!(\lambda)_m}(z|\overline{w})^m} = \Gamma(\lambda)\widetilde{J}_{\lambda-1}(2\sqrt{(z|\overline{w})}), \qquad z,w\in\calX_1,
\end{equation*}
where $\widetilde{J}_\alpha(z)=(\frac{z}{2})^{-\alpha}J_\alpha(z)$ is the classical \textit{renormalized $J$-Bessel function} which is an even entire function on $\CC$.
\end{example}

The following proposition is clear with the results of Section \ref{sec:SphericalPolynomials}:

\begin{proposition}\label{prop:JBesselProperties}
The $J$-Bessel function $\calJ_\lambda(z,w)$ has the following properties:
\begin{enumerate}
\item $\calJ_\lambda(z,w)=\overline{\calJ_\lambda(w,z)}$ for $z\in V_\CC$, $w\in\overline{\calX_\lambda}$,
\item $\calJ_\lambda(gz,w)=\calJ_\lambda(z,g^*w)$ for $z\in V_\CC$, $w\in\overline{\calX_\lambda}$, $g\in L_\CC$.
\end{enumerate}
\end{proposition}

To prove the differential equation for $\calJ_\lambda(z,w)$ we use the same method as Faraut--Koranyi \cite[Theorem XV.2.6]{FK94} for the one-variable $J$-Bessel function $\calJ(x,e)$. The first step is to calculate the Laplace transform as defined in \eqref{eq:DefLaplaceTrafo}.

\begin{lemma}\label{lem:LaplaceTrafoJBessel}
Fix $w\in\overline{\calX_\lambda}$ and consider the function $\calJ_\lambda(-,w)$ on $V$ given by $x\mapsto\calJ_\lambda(x,w)$. The Laplace transform of $\calJ_\lambda(-,w)$ is given by
\begin{align*}
 \left(\calL_\lambda\calJ_\lambda(-,w)\right)(z) &= 2^{r\lambda}\Delta(-iz)^{-\lambda}e^{-i(z^{-1}|\overline{w})}, & z\in T_\Omega.
\end{align*}
\end{lemma}

\begin{proof}
Let $0\leq k\leq r$ be such that $\calO_\lambda=\calO_k$. The we find, using Lemma~\ref{lem:LaplaceTransformPolynomials} and \eqref{eq:ExponentialTraceExpansion}:
\begin{align*}
 & \left(\calL_\lambda\calJ_\lambda(-,w)\right)(z)\\
 ={}& \int_{\calO_\lambda}{e^{i(z|x)}\calJ_\lambda(x,w)\td\mu_\lambda(x)}\\
 ={}& \sum_{{\bf m}\geq0, m_{k+1}=0}{(-1)^{|{\bf m}|}\frac{d_{\bf m}}{(\frac{n}{r})_{\bf m}(\lambda)_{\bf m}}\int_{\calO_\lambda}{e^{i(z|x)}\Phi_{\bf m}(x,w)\td\mu_\lambda(x)}}\\
 ={}& 2^{r\lambda}\sum_{{\bf m}\geq0}{(-1)^{|{\bf m}|}\frac{d_{\bf m}}{(\frac{n}{r})_{\bf m}}\Delta(-iz)^{-\lambda}\Phi_{\bf m}(iz^{-1},w)}\\
 ={}& 2^{r\lambda}\Delta(-iz)^{-\lambda}e^{-i(z^{-1}|\overline{w})}.\qedhere
\end{align*}
\end{proof}

\begin{proposition}\label{prop:DiffEqJBessel}
For $\lambda\in\calW$ the function $\calJ_\lambda(z,w)$ solves the following differential equation:
\begin{align*}
 (\calB_\lambda)_z\calJ_\lambda(z,w) &= -\overline{w}\calJ_\lambda(z,w), & z\in V_\CC,w\in\overline{\calX_\lambda}.
\end{align*}
\end{proposition}

\begin{proof}
First note that it suffices to show the differential equation for $\lambda>(r-1)\frac{d}{2}$, then the general case follows by analytic continuation. Further, it suffices to show the differential equation for $z\in\Omega$ as $\calJ_\lambda(z,w)$ is holomorphic in $z\in V_\CC$. Since the Laplace transform $\calL_\lambda$ is injective on functions on $\Omega$ the differential equation is equivalent to the identity
\begin{equation}
 \calL_\lambda(\calB_\lambda\calJ_\lambda(-,w)) = -\overline{w}\calL_\lambda\calJ_\lambda(-,w).\label{eq:LaplaceTrafoOnJBesselDiffEq}
\end{equation}
Let $z\in T_\Omega$, then using the symmetry of the Bessel operator $\calB_\lambda$ we find
\begin{align*}
 & \left(\calL_\lambda(\calB_\lambda\calJ_\lambda(-,w))\right)(z)\\
 ={}& \int_{\calO_\lambda}{e^{i(z|x)}(\calB_\lambda)_x\calJ_\lambda(x,w)\td\mu_\lambda(x)}\\
 ={}& \int_{\calO_\lambda}{(\calB_\lambda)_xe^{i(z|x)}\calJ_\lambda(x,w)\td\mu_\lambda(x)}\\
 ={}& \int_{\calO_\lambda}{(P(iz)x+i\lambda z)e^{i(z|x)}\calJ_\lambda(x,w)\td\mu_\lambda(x)}\\
 ={}& i\left(P(z)\frac{\partial}{\partial z}+\lambda z\right)\left((\calL_\lambda\calJ_\lambda(-,w)\right)(z).
\end{align*}
Now by Lemma~\ref{lem:LaplaceTrafoJBessel} we have $\left(\calL_\lambda\calJ_\lambda(-,w)\right)(z)=2^{r\lambda}\Delta(-iz)^{-\lambda}e^{-i(z^{-1}|\overline{w})}$. Using $D(x^{-1})=-P(x)^{-1}$ and $\partial_y\Delta(x)=\Delta(x)(y|x^{-1})$ we find
\begin{align*}
 & i\left(P(z)\frac{\partial}{\partial z}+\lambda z\right)\left[\Delta(-iz)^{-\lambda}e^{-i(z^{-1}|\overline{w})}\right]\\
 ={}& i\left(-\lambda P(z)z^{-1}+iP(z)P(z)^{-1}\overline{w}+\lambda z\right)\Delta(-iz)^{-\lambda}e^{-i(z^{-1}|\overline{w})}\\
 ={}& -\overline{w}\Delta(-iz)^{-\lambda}e^{-i(z^{-1}|\overline{w})}
\end{align*}
and \eqref{eq:LaplaceTrafoOnJBesselDiffEq} follows.
\end{proof}

\subsection{$I$-Bessel function}\label{sec:IBessel}

For $\lambda\in\calW$, $z,w\in V_\CC$ and $z\in\overline{\calX_\lambda}$ or $w\in\overline{\calX_\lambda}$ we put
\begin{equation*}
 \calI_\lambda(z,w) := \calJ_\lambda(-z,w) = \calJ_\lambda(z,-w) = \sum_{{\bf m}\geq0}{\frac{d_{\bf m}}{(\frac{n}{r})_{\bf m}(\lambda)_{\bf m}}\Phi_{\bf m}(z,w)}.
\end{equation*}
By definition the $I$-Bessel function $\calI_\lambda(z,w)$ also satisfies the estimates in Proposition \ref{lem:JBesselConvergence} and has the same properties as in Proposition \ref{prop:JBesselProperties}

\begin{example}\label{ex:IBesselRank1}
By Example \ref{ex:JBesselRank1} it is immediate that on the rank $1$ orbit we have
\begin{align*}
 \calI_\lambda(z,w) &= \Gamma(\lambda)\widetilde{I}_{\lambda-1}(2\sqrt{(z|\overline{w})}), & z,w\in\calX_1,
\end{align*}
where $\widetilde{I}_\alpha(z)=(\frac{z}{2})^{-\alpha}I_\alpha(z)$ is the \textit{classical renormalized $I$-Bessel function} which is an even entire function on $\CC$.
\end{example}

\begin{lemma}\label{lem:IBesselIntFormula}
For $\lambda\in\calW$ and $y\in\Omega$, $z\in V_\CC$
\begin{align*}
 \int_{\calO_\lambda}{e^{-(x|y)}\calI_\lambda(x,z)\td\mu_\lambda(x)} &= 2^{r\lambda}\Delta(y)^{-\lambda}e^{(y^{-1}|\overline{z})}.
\end{align*}
\end{lemma}

\begin{proof}
This follows immediately from Lemma~\ref{lem:LaplaceTrafoJBessel}.
\end{proof}

\begin{proposition}\label{prop:DiffEqIBessel}
For $\lambda\in\calW$ the function $\calI_\lambda(z,w)$ solves the following differential equation:
\begin{align*}
 (\calB_\lambda)_z\calI_\lambda(z,w) &= \overline{w}\calI_\lambda(z,w), & z,w\in\calX_\lambda.
\end{align*}
\end{proposition}

\begin{proof}
Since $\calI_\lambda(z,w)=\calJ_\lambda(z,-w)$ this is equivalent to Proposition \ref{prop:DiffEqJBessel}.
\end{proof}

\subsection{$K$-Bessel function}\label{sec:KBessel}

For $\lambda\in\CC$ and $x\in\Omega$ we put
\begin{align}
 \calK_\lambda(x) &:= \int_\Omega{e^{-\tr(u^{-1})-(x|u)}\Delta(u)^{\lambda-\frac{2n}{r}}\td u}\label{eq:DefKBessel1}\\
 &= \int_\Omega{e^{-\tr(v)-(x|v^{-1})}\Delta(v)^{-\lambda}\td v}.\label{eq:DefKBessel2}
\end{align}
Note that our normalization of the parameter $\lambda$ differs from the one used in \cite{Cle88} and \cite[Chapter XVI.3]{FK94}. By \cite[Proposition XVI.3.1]{FK94} these integrals converge for all $\lambda\in\CC$ and $x\in\Omega$. Since the integrand is positive on $\Omega$ we have $\calK_\lambda(x)>0$ for $x\in\Omega$. To extend the $K$-Bessel function also to lower rank orbits we need the following result due to J.-L. Clerc \cite[Th\'eor\`{e}me 4.1]{Cle88}:

\begin{proposition}\label{prop:KBesselOnLowerRankOrbits}
Let $c\in V$ be an idempotent of rank $k$. Let $\Omega_1$ and $\Omega_0$ be the symmetric cones in the Euclidean Jordan algebras $V(c,1)$ and $V(c,0)$, respectively. Further, let $\calK_\lambda^1$ be the $K$-Bessel function of $\Omega_1$, $\Gamma_{\Omega_0}$ the Gamma function of $\Omega_0$ and $n_0$ and $r_0$ the dimension and rank of $V(c,0)$. Then for $x_1\in\Omega_1$
\begin{align}
 \calK_\lambda(x_1) &= (2\pi)^{k(r-k)\frac{d}{2}}\Gamma_{\Omega_0}\left(\frac{n_0}{r_0}+k\frac{d}{2}-\lambda\right)\calK_\lambda^1(x_1).\label{eq:KBesselOnLowerRankOrbits}
\end{align}
\end{proposition}

\begin{proof}
Denote by $\Delta^1(x)$ and $\Delta^0(x)$ the determinant functions of $V(c,1)$ and $V(c,0)$, respectively. Using Lemma \ref{lem:IntFormulaOmega} we obtain
\begin{align*}
 \calK_\lambda(x_1) ={}& \int_\Omega{e^{-\tr(u)-(x_1|u^{-1})}\Delta(u)^{-\lambda}\td u}\\
 ={}& 2^{-k(r-k)d}\int_{\Omega_1}{\int_{V(c,\frac{1}{2})}{\int_{\Omega_0}{e^{-\tr(\phi(u))}e^{-(x_1|\phi(u)^{-1})}}}}\\
 & \hspace{4cm}{{{\Delta^0(u_0)^{kd}\Delta(\phi(u))^{-\lambda}\td u_0}\td u_{\frac{1}{2}}}\td u_1}
\end{align*}
with $\phi(u)=\exp(c\Box u_{\frac{1}{2}})(u_1+u_0)$. Note that
\begin{align*}
 (c\Box u_{\frac{1}{2}}): V(c,1)&\to 0,\\
 (c\Box u_{\frac{1}{2}}): V(c,{\textstyle\frac{1}{2}})&\to V(c,1),\,x\mapsto c(u_{\frac{1}{2}}x),\\
 (c\Box u_{\frac{1}{2}}): V(c,0)&\to V(c,{\textstyle\frac{1}{2}}),\,x\mapsto u_{\frac{1}{2}}x,
\end{align*}
and hence
\begin{align*}
 \phi(u) &= \underbrace{\left(u_1+\frac{1}{2}c(u_{\frac{1}{2}}(u_{\frac{1}{2}}u_0))\right))}_{\in V(c,1)}+\underbrace{(u_{\frac{1}{2}}u_0)}_{\in V(c,\frac{1}{2})}+\underbrace{u_0}_{\in V(c,0)},\\
\intertext{whence}
 \tr(\phi(u)) &= \tr(u_1)+\frac{1}{4}(L(u_0)u_{\frac{1}{2}}|u_{\frac{1}{2}})+\tr(u_0).
\end{align*}
Further, we have
\begin{align*}
 \phi(u)^{-1} &= \exp(c\Box u_{\frac{1}{2}})^{-\#}(u_1+u_0) = \exp((e-c)\Box(-u_{\frac{1}{2}}))(u_1^{-1}+u_0^{-1})
\intertext{and by the same argument as above}
 (\phi(u)^{-1}|x_1) &= (u_1^{-1}|x_1).
\end{align*}
Finally
\begin{align*}
 \Delta(\phi(u)) &= \Delta(\exp(c\Box u_{\frac{1}{2}})(u_1+u_0))\\
 &= \Det(\exp(c\Box u_{\frac{1}{2}}))^{\frac{r}{n}}\Delta(u_1+u_0)\\
 &= \exp\left(\frac{r}{n}\Tr(c\Box u_{\frac{1}{2}})\right)\Delta^1(u_1)\Delta^0(u_0)\\
\intertext{and since $\Tr(c\Box u_{\frac{1}{2}})=\frac{1}{2}\tr(u_{\frac{1}{2}})=0$ we obtain}
 &= \Delta^1(u_1)\Delta^0(u_0).
\end{align*}
We now first calculate the integral over $\Omega_1$:
\begin{align*}
 \int_{\Omega_1}{e^{-\tr(u_1)-(x_1|u_1^{-1})}\Delta^1(u_1)^{-\lambda}\td u_1} &= \calK^1_\lambda(x_1).
\end{align*}
Next we calculate the integral over $V(c,\frac{1}{2})$:
\begin{align*}
 \int_{V(c,\frac{1}{2})}{e^{-\frac{1}{4}(L(u_0)u_{\frac{1}{2}}|u_{\frac{1}{2}})}\td u_{\frac{1}{2}}} &= \pi^{\frac{1}{2}\dim V(c,\frac{1}{2})}\Det\left(\frac{1}{4}L(u_0)_{V(c,\frac{1}{2})}\right)^{-\frac{1}{2}}
\end{align*}
by the well-known integral formula for Gaussians. Further,
\begin{align*}
 \Det(2L(u_0)|_{V(c,\frac{1}{2})}) &= \Delta^0(u_0)^{kd}
\end{align*}
by {\cite[Propositions IV.4.1 \& IV.4.2]{FK94}}. Finally we deal with the integral over $\Omega_0$:
\begin{align*}
 \int_{\Omega_0}{e^{-\tr(u_0)}\Delta^0(u_0)^{\frac{kd}{2}-\lambda}\td u_0} &= \Gamma_{\Omega_0}\left(\frac{n_0}{r_0}+\frac{kd}{2}-\lambda\right).
\end{align*}
Putting all together we obtain the claimed formula.
\end{proof}

This shows that for $\lambda$ near $k\frac{d}{2}$ the Bessel function $\calK_\lambda(x)$ is defined for $x\in\calO_k$ and hence we obtain Bessel functions $\calK_\lambda$ on $\calO_\lambda$ for $\lambda\in\calW$. Note that by \eqref{eq:KBesselOnLowerRankOrbits} the function $\calK_\lambda$ is positive on $\calO_\lambda$.

\begin{example}\label{ex:KBessel}
\begin{enumerate}
\item For $V=\RR$ we have by \cite[formula 3.471~(9)]{GR65}
\begin{align*}
 \calK_\lambda(x) &= 2\widetilde{K}_{\lambda-1}(2\sqrt{x}), & x\in\Omega=\RR_+,
\end{align*}
where $\widetilde{K}_\alpha(z)=\left(\frac{z}{2}\right)^{-\alpha}K_\alpha(z)$ is the \textit{classical renormalized $K$-Bessel function}.
\item In the general case the Bessel function $\calK_\lambda$ is by Proposition \ref{prop:KBesselOnLowerRankOrbits} on the rank $1$ orbit $\calO_1$ given by
\begin{align*}
 \calK_\lambda(x) &= \const\cdot\calK_\lambda^1(|x|c_1) = \const\cdot\widetilde{K}_{\lambda-1}(2\sqrt{|x|}), & x\in\calO_1.
\end{align*}
\end{enumerate}
\end{example}

\begin{lemma}\label{lem:IntPhimKBessel}
For $\lambda\in\calW$ and ${\bf m}\geq0$ we have
\begin{align*}
 \int_{\calO_\lambda}{p(x)\calK_\lambda(x)\td\mu_\lambda(x)} &= 2^{r\lambda}\Gamma_\Omega(\tfrac{n}{r})(\tfrac{n}{r})_{\bf m}(\lambda)_{\bf m}p(e), & p\in\calP_{\bf m}(V_\CC).
\end{align*}
In particular, for every $N>0$ we have
\begin{align*}
 \int_{\calO_\lambda}{(1+|x|)^N\calK_\lambda(x)\td\mu_\lambda(x)} &< \infty.
\end{align*}
\end{lemma}

\begin{proof}
It suffices to show the claim for $\lambda>(r-1)\frac{d}{2}$. The general case then follows by analytic continuation. For $\lambda>(r-1)\frac{d}{2}$ we have, using Lemma \ref{lem:LaplaceTransformPolynomials}
\begin{align*}
 & \int_{\calO_\lambda}{p(x)\calK_\lambda(x)\td\mu_\lambda(x)}\\
 ={}& \frac{2^{r\lambda}}{\Gamma_\Omega(\lambda)}\int_\Omega{p(x)\calK_\lambda(x)\Delta(x)^{\lambda-\frac{n}{r}}\td x}\\
 ={}& \frac{2^{r\lambda}}{\Gamma_\Omega(\lambda)}\int_\Omega{\int_\Omega{p(x)e^{-\tr(u^{-1})-(x|u)}\Delta(u)^{\lambda-\frac{2n}{r}}\Delta(x)^{\lambda-\frac{n}{r}}\td u}\td x}\\
 ={}& 2^{r\lambda}(\lambda)_{\bf m}\int_\Omega{e^{-\tr(u^{-1})}p(u^{-1})\Delta(u)^{-\frac{2n}{r}}\td u}\\
\intertext{and under the coordinate change $v=u^{-1}$, $\td v=\Delta(u)^{-\frac{2n}{r}}\td u$, this is}
 ={}& 2^{r\lambda}(\lambda)_{\bf m}\int_\Omega{e^{-\tr(v)}p(v)\td v}\\
 ={}&  2^{r\lambda}(\lambda)_{\bf m}\Gamma_\Omega(\tfrac{n}{r})(\tfrac{n}{r})_{\bf m}p(e),
\end{align*}
where we have again used Lemma \ref{lem:LaplaceTransformPolynomials} for the last equality. This shows the desired integral formula.\\
For the second claim we observe that by the previous calculations every polynomial can be integrated against the positive measure $\calK_\lambda(x)\td\mu_\lambda(x)$ since $\calP(V_\CC)=\bigoplus_{{\bf m}\geq0}{\calP_{\bf m}(V_\CC)}$ and hence the claim follows.
\end{proof}

\begin{proposition}\label{prop:KBesselDiffEq}
The function $\calK_\lambda(x)$ solves the following differential equation:
\begin{align*}
 \calB_\lambda\calK_\lambda(x) &= e\calK_\lambda(x).
\end{align*}
\end{proposition}

\begin{proof}
Differentiating under the integral we obtain
\begin{align*}
 & \calB_\lambda\calK_\lambda(x)\\
 ={}& \int_\Omega{(P(-u)x+\lambda(-u))e^{-\tr(u^{-1})-(x|u)}\Delta(u)^{\lambda-\frac{2n}{r}}\td u}\\
 ={}& \int_\Omega{-\left(P(u)\frac{\partial}{\partial u}+\lambda u\right)e^{-(x|u)}\cdot e^{-\tr(u^{-1})}\Delta(u)^{\lambda-\frac{2n}{r}}\td u}\\
 ={}& \sum_\alpha{\int_\Omega{-\left(P(u)e_\alpha\frac{\partial}{\partial u_\alpha}+\lambda u\right)e^{-(x|u)}\cdot e^{-\tr(u^{-1})}\Delta(u)^{\lambda-\frac{2n}{r}}\td u}}\\
 ={}& \sum_\alpha{\int_\Omega{e^{-(x|u)}\cdot\left(\frac{\partial}{\partial u_\alpha}P(u)e_\alpha-\lambda u\right)\left[e^{-\tr(u^{-1})}\Delta(u)^{\lambda-\frac{2n}{r}}\right]\td u}}.\\
\intertext{Using $\partial_yP(x)=2P(x,y)$, $D(x^{-1})=-P(x)^{-1}$ and $\partial_y\Delta(x)=\Delta(x)(y|x^{-1})$ we obtain}
 ={}& \sum_\alpha{\int_\Omega{e^{-(x|u)}\Big(2P(e_\alpha,u)e_\alpha+P(u)e_\alpha\tr(P(u)^{-1}e_\alpha)}}\\
 & +(\lambda-\tfrac{2n}{r})(e_\alpha|u^{-1})P(u)e_\alpha-\lambda u\Big)e^{-\tr(u^{-1})}\Delta(u)^{\lambda-\frac{2n}{r}}\td u\\
 ={}& \int_\Omega{e^{-(x|u)}\Big(2\Big(\sum_\alpha{e_\alpha^2}\Big)u+P(u)P(u)^{-1}e}\\
 & +(\lambda-\tfrac{2n}{r})P(u)u^{-1}-\lambda u\Big)e^{-\tr(u^{-1})}\Delta(u)^{\lambda-\frac{2n}{r}}\td u.
\intertext{By \eqref{eq:SumOfSquares} this is}
 ={}& \int_\Omega{e^{-(x|u)}\Big(\tfrac{2n}{r}u+e+(\lambda-\tfrac{2n}{r})u-\lambda u\Big)e^{-\tr(u^{-1})}\Delta(u)^{\lambda-\frac{2n}{r}}\td u}\\
 ={}& e\calK_\lambda(x).\qedhere
\end{align*}
\end{proof}

Now let $\lambda\in\calW$. For $z\in\calX_\lambda$, $z=ua$ with $u\in U$, $a=\sum_{i=1}^r{t_ic_i}$, $t_i\geq0$, we put
\begin{align*}
 \omega_\lambda(z) &:= \calK_\lambda\left(\left(\frac{a}{2}\right)^2\right).
\end{align*}
We note that $\omega_\lambda$ is positive on $\calX_\lambda$.

\begin{proposition}\label{prop:OmegaDiffEq}
The function $\omega_\lambda(z)$ solves the following differential equation:
\begin{align*}
 \calB_\lambda\omega_\lambda(z) &= \frac{\overline{z}}{4}\omega_\lambda(z), & z\in\calX_\lambda.
\end{align*}
\end{proposition}

\begin{proof}
Recall the operators $\calB_\lambda^V$ and $\calB_\lambda^W$ acting on functions of $r$ variables (see Proposition \ref{prop:BesselOnInvariants}). Let $F(a_1,\ldots,a_r)=\calK_\lambda(a)$, $a=\sum_{i=1}^r{a_ic_i}$, then $F$ solves the system $(\calB_\lambda^V)^iF=F$, $i=1,\ldots,r$ by Propositions \ref{prop:BesselOnInvariants} and \ref{prop:KBesselDiffEq}. Put $G(a_1,\ldots,a_r):=\omega_\lambda(a)=F((\frac{a_1}{2})^2,\ldots,(\frac{a_r}{2})^2)$. Then
\begin{align*}
 & (\calB_\lambda^W)^iG(a_1,\ldots,a_r)\\
 ={}& \frac{1}{4}\Bigg(a_i\frac{\partial^2}{\partial a_i^2}+\left(2\lambda-1-(r-1)d\right)\frac{\partial}{\partial a_i}\\
 & +\frac{d}{2}\sum_{j\neq i}{\left(\frac{1}{a_i-a_j}+\frac{1}{a_i+a_j}\right)\left(a_i\frac{\partial}{\partial a_i}-a_j\frac{\partial}{\partial a_j}\right)}\Bigg)G(a_1,\ldots,a_r)\\
 ={}& \frac{1}{4}\Bigg(a_i\left(\frac{a_i^2}{4}\frac{\partial^2}{\partial a_i^2}+\frac{1}{2}\frac{\partial}{\partial a_i}\right)+\frac{a_i}{2}\left(2\lambda-1-(r-1)d\right)\frac{\partial}{\partial a_i}\\
 & +\frac{d}{2}\sum_{j\neq i}{\left(\frac{1}{a_i-a_j}+\frac{1}{a_i+a_j}\right)\left(\frac{a_i^2}{2}\frac{\partial}{\partial a_i}-\frac{a_j^2}{2}\frac{\partial}{\partial a_j}\right)}\Bigg)F\left(\left(\frac{a_1}{2}\right)^2,\ldots,\left(\frac{a_r}{2}\right)^2\right)\\
 ={}& \frac{a_i}{4}\Bigg(\left(\frac{a_i}{2}\right)^2\frac{\partial^2}{\partial a_i^2}+\left(\lambda-(r-1)\frac{d}{2}\right)\frac{\partial}{\partial a_i}\\
 & +\frac{d}{2}\sum_{j\neq i}{\frac{1}{(\frac{a_i}{2})^2-(\frac{a_j}{2})^2}\left(\left(\frac{a_i}{2}\right)^2\frac{\partial}{\partial a_i}-\left(\frac{a_j}{2}\right)^2\frac{\partial}{\partial a_j}\right)}\Bigg)F\left(\left(\frac{a_1}{2}\right)^2,\ldots,\left(\frac{a_r}{2}\right)^2\right)\\
 ={}& \frac{a_i}{4}(\calB_\lambda^V)^iF\left(\left(\frac{a_1}{2}\right)^2,\ldots,\left(\frac{a_r}{2}\right)^2\right) = \frac{a_i}{4}F\left(\left(\frac{a_1}{2}\right)^2,\ldots,\left(\frac{a_r}{2}\right)^2\right)\\
 ={}& \frac{a_i}{4}G(a_1,\ldots,a_r).
\end{align*}
Hence we obtain $\calB_\lambda^W\omega_\lambda(z)=\frac{\overline{z}}{4}\omega_\lambda(z)$ for $z=a$. Since $\omega_\lambda(z)$ is further $U$-invariant, we obtain with \eqref{eq:BesselEquivariance} for $z=ua$ with $u\in U$ and $a=\sum_{i=1}^r{a_ic_i}$:
\begin{equation*}
 \calB_\lambda^W\omega_\lambda(z) = (u^{-1})^\#\calB_\lambda^W\omega_\lambda(a) = \overline{u}\frac{a}{4}\omega_\lambda(a) = \frac{\overline{z}}{4}\omega_\lambda(z)
\end{equation*}
since $\overline{u}^{\#}=u^{-1}$ for $u\in U$. Finally we use Proposition \ref{prop:PiCC} to find for every $a\in V$:
\begin{align*}
 i(a|\calB_\lambda)\omega_\lambda(z) &= \td\pi_\lambda^\CC(0,0,a)\omega_\lambda(z)\\
 &= \frac{1}{2}\left(\td\tau_\lambda(0,0,a)\omega_\lambda(z)-i\td\tau_\lambda(0,0,ia)\omega_\lambda(z)\right)\\
 &= \frac{1}{2}i\left((a|\calB_\lambda^W)_W\omega_\lambda(z)-i(ia|\calB_\lambda^W)_W\omega_\lambda(z)\right)\\
 &= \frac{1}{2}i\left((a|\textstyle{\frac{\overline{z}}{4}})_W-i(ia|\textstyle{\frac{\overline{z}}{4}})_W\right)\omega_\lambda(z)\\
 &= i(a|\textstyle{\frac{\overline{z}}{4}})\omega_\lambda(z).
\end{align*}
Since this holds for any $a\in V$ we find $\calB_\lambda\omega_\lambda(z)=\frac{\overline{z}}{4}\omega_\lambda(z)$ and the proof is complete.
\end{proof}

\begin{example}
On the rank $1$ orbit the function $\omega_\lambda$ takes by Example \ref{ex:KBessel}~(2) the form
\begin{align*}
 \omega_\lambda(z) &= \omega_\lambda(|z|c_1) = \calK_\lambda\left(\frac{|z|^2}{4}c_1\right) = \const\cdot\widetilde{K}_{\lambda-1}(|z|).
\end{align*}
\end{example}

\newpage
\section{A Fock model for unitary highest weight representations of scalar type}\label{sec:FockModel}

In this section we construct a \textit{Fock space} $\calF_\lambda=\calF(\calX_\lambda,\omega_\lambda\td\nu_\lambda)$ of holomorphic functions on the orbit $\calX_\lambda$ for every $\lambda\in\calW$, calculate its reproducing kernel and find a realization on $\calF_\lambda$ of the unitary highest weight representation corresponding to the Wallach point $\lambda$.

\subsection{Construction of the Fock space}

Let $\lambda\in\calW$. Recall the positive function $\omega_\lambda\in C^\infty(\calX_\lambda)$ from Section \ref{sec:KBessel}. We endow the space $\calP(\calX_\lambda)$ of polynomials on $\calX_\lambda$ with the $L^2$-inner product of $L^2(\calX_\lambda,\omega_\lambda\td\nu_\lambda)$:
\begin{align}
 \langle F,G\rangle_{\calF_\lambda} &:= \frac{1}{c_\lambda}\int_{\calX_\lambda}{F(z)\overline{G(z)}\omega_\lambda(z)\td\nu_\lambda(z)}, & F,G\in\calP(\calX_\lambda)\label{eq:DefFockInnerProduct}
\end{align}
with $c_\lambda=2^{3r\lambda}\Gamma_\Omega(\frac{n}{r})$. This turns $\calP(\calX_\lambda)$ into a pre-Hilbert space. Its completion $\calF_\lambda:=\calF(\calX_\lambda,\omega_\lambda\td\nu_\lambda)$ will be called the \textit{Fock space} on $\calX_\lambda$.

It remains to show that the integral in \eqref{eq:DefFockInnerProduct} converges. Using the polarization principle the following lemma suffices:

\begin{lemma}
For $F\in\calP(V_\CC)$ we have
\begin{align*}
 \int_{\calX_\lambda}{|F(z)|^2\omega_\lambda(z)\td\nu_\lambda(z)} &< \infty.
\end{align*}
\end{lemma}

\begin{proof}
Using the integral formula \eqref{eq:IntFormulaOnXlambda} we obtain
\begin{align*}
 \int_{\calX_\lambda}{|F(z)|^2\omega_\lambda(z)\td\nu_\lambda(z)} &= \int_U{\int_{\calO_\lambda}{|F(ux^{\frac{1}{2}})|^2\calK_\lambda(x)\td\mu_\lambda(x)}\td u}.
\end{align*}
Now put
\begin{align*}
 p(x) &:= \int_U{|F(ux)|^2\td u}.
\end{align*}
Clearly $p$ is a polynomial on $V$, so there are constants $C_1>0$ and $N\in\NN$ such that $|p(x)|\leq C_1(1+|x|)^N$. Now, every $x\in\overline{\Omega}$ has a decomposition $x=ka$ with $k\in K^L$ and $a=\sum_{j=1}^r{a_jc_j}$, $a_j\geq0$. In this decomposition the square root is of the form $x^{\frac{1}{2}}=ka^{\frac{1}{2}}$ with $a^{\frac{1}{2}}=\sum_{j=1}^r{a_j^{\frac{1}{2}}c_j}$. Since the norm $|-|$ on $V$ is $K^L$-invariant we obtain
\begin{align*}
 |p(x^{\frac{1}{2}})| &\leq C_1\left(1+(a_1+\cdots+a_r)^{\frac{1}{2}}\right)^N \leq C_1\left(1+\sqrt{r}(a_1^2+\cdots+a_r^2)^{\frac{1}{4}}\right)^N\\
 &\leq C_1\sqrt{r}\left(1+|x|^{\frac{1}{2}}\right)^N \leq C_1\sqrt{r}\left(\sqrt{2}(1+|x|)^{\frac{1}{2}}\right)^N\\
 &= C_2\left(1+|x|\right)^{\frac{N}{2}}
\end{align*}
with $C_2=C_1\sqrt{r}2^{\frac{N}{2}}$. Hence, we find
\begin{align*}
\int_{\calX_\lambda}{|F(z)|^2\omega_\lambda(z)\td\nu_\lambda(z)} &= \int_{\calO_\lambda}{p(x)\calK_\lambda(x)\td\mu_\lambda(x)}\\
&\leq C_2\int_{\calO_\lambda}{(1+|x|)^{\frac{N}{2}}\calK_\lambda(x)\td\mu_\lambda(x)}.
\end{align*}
The latter integral is finite by Lemma \ref{lem:IntPhimKBessel} and the proof is complete.
\end{proof}

We explicitly calculate the norms on the finite-dimensional subspaces $\calP_{\bf m}(\calX_\lambda)$.

\begin{proposition}\label{prop:FockNormCalculations}
Let ${\bf m},{\bf n}\geq0$ and $F\in\calP_{\bf m}(\calX_\lambda)$, $G\in\calP_{\bf n}(\calX_\lambda)$. Then
\begin{align*}
 \langle F,G\rangle_{\calF_\lambda} &= 4^{|{\bf m}|}(\tfrac{n}{r})_{\bf m}(\lambda)_{\bf m}\langle F,G\rangle_\Sigma.
\end{align*}
In particular the subspaces $\calP_{\bf m}(\calX_\lambda)\subseteq\calP(\calX_\lambda)$ are pairwise orthogonal with respect to the inner product $\langle-,-\rangle_{\calF_\lambda}$ and for $F=\Phi_{\bf m}$ we have
\begin{align*}
 \|\Phi_{\bf m}\|_{\calF_\lambda}^2 &= \frac{4^{|{\bf m}|}(\tfrac{n}{r})_{\bf m}(\lambda)_{\bf m}}{d_{\bf m}}.
\end{align*}
\end{proposition}

\begin{proof}
Using the integral formula \eqref{eq:IntFormulaOnXlambda}, Lemma \ref{lem:SchurOrthogonalityLemma} and Lemma \ref{lem:IntPhimKBessel} we obtain
\begin{align*}
 \langle F,G\rangle_{\calF_\lambda} &= \frac{1}{c_\lambda}\int_{\calX_\lambda}{F(z)\overline{G(z)}\omega_\lambda(z)\td\nu_\lambda(z)}\\
 &= \frac{1}{c_\lambda}\int_{\calO_\lambda}{\int_U{F(ux^{\frac{1}{2}})\overline{G(ux^{\frac{1}{2}})}\omega_\lambda(x^{\frac{1}{2}})\td\mu_\lambda(x)}\td u}\\
 &= \frac{1}{c_\lambda}\langle F,G\rangle_\Sigma\int_{\calO_\lambda}{\Phi_{\bf m}(x)\calK_\lambda({\textstyle\frac{x}{4}})\td\mu_\lambda(x)}\\
 &= \frac{4^{r\lambda+|{\bf m}|}}{c_\lambda}\langle F,G\rangle_\Sigma\int_{\calO_\lambda}{\Phi_{\bf m}(y)\calK_\lambda(y)\td\mu_\lambda(y)}\\
 &=4^{|{\bf m}|}(\tfrac{n}{r})_{\bf m}(\lambda)_{\bf m}\langle F,G\rangle_\Sigma.
\end{align*}
Since $\|\Phi_{\bf m}\|_\Sigma^2=\frac{1}{d_{\bf m}}$ by \eqref{eq:L2SigmaNormOfPhim} this finishes the proof.
\end{proof}

\begin{remark}\label{rem:ComparisonNormsFockBdSymDomain}
Comparing the norm on $\calF_\lambda$ with the norm on the space $\calH_\lambda^2(\calD)$ gives by \eqref{eq:NormOnBdSymDomainModel}
\begin{align*}
 \|F\|_{\calF_\lambda}^2 &= 4^{|{\bf m}|}(\lambda)_{\bf m}^2\|F\|_{\calH_\lambda^2(\calD)}^2 & \forall\,F\in\calP_{\bf m}(V_\CC).
\end{align*}
\end{remark}

If we denote by $\calO(\calX_\lambda)$ the space of holomorphic functions on the complex manifold $\calX_\lambda$, we obtain the following result:

\begin{proposition}\label{prop:FockContPointEvaluations}
$\calO(\calX_\lambda)\cap L^2(\calX_\lambda,\omega_\lambda\td\nu_\lambda)$ is a closed subspace of $L^2(\calX_\lambda,\omega_\lambda\td\nu_\lambda)$ and the point evaluation $\calO(\calX_\lambda)\cap L^2(\calX_\lambda,\omega_\lambda\td\nu_\lambda)\to\CC,\,F\mapsto F(z),$ is continuous for every $z\in\calX_\lambda$. In particular, $\calF_\lambda\subseteq\calO(\calX_\lambda)\cap L^2(\calX_\lambda,\omega_\lambda\td\nu_\lambda)$ and the point evaluation $\calF_\lambda\to\CC,\,F\mapsto F(z)$ is continuous for every $z\in\calX_\lambda$.
\end{proposition}

\begin{proof}
This is a local statement and hence, we may transfer it with a chart map to an open domain $U\subseteq\CC^k$. Here the measure $\omega_\lambda\td\nu_\lambda$ is absolutely continuous with respect to the Lebesgue measure $\td z$ and hence it suffices to show that $\calO(\CC^k)\cap L^2(\CC^k,\td z)\subseteq L^2(\CC^k,\td z)$ is a closed subspace with continuous point evaluations. This is done e.g. in \cite[Proposition 3.1 and Corollary 3.2]{Hel62}.
\end{proof}

The particular choice of the density $\omega_\lambda$ yields the following result:

\begin{proposition}\label{prop:BesselAdjoint}
The adjoint of $\calB_\lambda$ on $\calF_\lambda$ is $\frac{z}{4}$.
\end{proposition}

\begin{proof}
Let $F$ and $G$ be holomorphic functions on $\calX_\lambda$. Then by Proposition \ref{prop:PiCC} we know that
\begin{align*}
 \int_{\calX_\lambda}{\calB_\lambda F(z)\overline{G(z)}\omega_\lambda(z)\td\nu_\lambda(z)} &= \int_{\calX_\lambda}{F(z)\calB_\lambda\left[\overline{G(z)}\omega_\lambda(z)\right]\td\nu_\lambda(z)}.
\intertext{The function $\overline{G(z)}$ is antiholomorphic and hence}
 &= \int_{\calX_\lambda}{F(z)\overline{G(z)}\calB_\lambda\omega_\lambda(z)\td\nu_\lambda(z)}.
\intertext{By Proposition \ref{prop:OmegaDiffEq} we have $\calB_\lambda\omega_\lambda(z)=\frac{\overline{z}}{4}\omega_\lambda(z)$ and therefore}
 &= \int_{\calX_\lambda}{F(z)\overline{\frac{z}{4}G(z)}\omega_\lambda(z)\td\nu_\lambda(z)}.\qedhere
\end{align*}
\end{proof}

\subsection{The Bessel--Fischer inner product}\label{sec:BesselFischer}

We introduce another inner product on the space $\calP(\calX_\lambda)$ of polynomials, the \textit{Bessel--Fischer inner product}. For two polynomials $p$ and $q$ it is defined by
\begin{align*}
 [p,q]_\lambda &:= \left.p(\calB_\lambda)\overline{q}(4z)\right|_{z=0},
\end{align*}
where $\overline{q}(z)=\overline{q(\overline{z})}$ is obtained by conjugating the coefficients of the polynomial $q$. A priori it is not even clear that this sesquilinear form is positive definite.

\begin{theorem}\label{thm:FischerEqualsL2}
For $p,q\in\calP(\calX_\lambda)$ we have
\begin{align}
 [p,q]_\lambda &= \langle p,q\rangle_{\calF_\lambda}.\label{eq:FischerEqualsL2}
\end{align}
\end{theorem}

The proof is similar to the proof of \cite[Proposition 3.8]{BSO06}

\begin{proof}
First note that for all $p,q\in\calP(\calX_\lambda)$
\begin{align*}
 [(a|\textstyle\frac{z}{4})p,q]_\lambda &= [p,(\overline{a}|\calB_\lambda)q]_\lambda & \mbox{for }a\in V_\CC,\\
 \langle(a|\textstyle\frac{z}{4})p,q\rangle_{\calF_\lambda} &= \langle p,(\overline{a}|\calB_\lambda)q\rangle_{\calF_\lambda} & \mbox{for }a\in V_\CC.
\end{align*}
In fact, the second equation follows from Proposition \ref{prop:BesselAdjoint}. The first equation is immediate since the components $(a|\calB_\lambda)$, $a\in V_\CC$, of the Bessel operator form a commuting family of differential operators on $\calX_\lambda$. Therefore, $(a|\calB_\lambda)p(\calB_\lambda)\overline{q}(4z)=4p(\calB_\lambda)\overline{(\overline{a}|\calB_\lambda)q}(4z)$ and the claim follows.\\
To prove \eqref{eq:FischerEqualsL2} we proceed by induction on $\deg(q)$. First, if $p=q=\1$, the constant polynomial with value $1$, it is clear that $[p,q]_\lambda=1$ and by Proposition \ref{prop:FockNormCalculations} we also have $\langle p,q\rangle_{\calF_\lambda}=1$. Thus, \eqref{eq:FischerEqualsL2} holds for $\deg(p)=\deg(q)=0$. If now $\deg(p)$ is arbitrary and $\deg(q)=0$ then $(\overline{a}|\calB_\lambda)q=0$ and hence
\begin{align*}
 [(a|\textstyle\frac{z}{4})p,q]_\lambda &= [p,(\overline{a}|\calB_\lambda)q]_\lambda = 0 & \mbox{and}\\
 \langle(a|\textstyle\frac{z}{4})p,q\rangle_{\calF_\lambda} &= \langle p,(\overline{a}|\calB_\lambda)q\rangle_{\calF_\lambda} = 0.
\end{align*}
Therefore, \eqref{eq:FischerEqualsL2} holds if $\deg(q)=0$. We note that \eqref{eq:FischerEqualsL2} also holds if $\deg(p)=0$ and $\deg(q)$ is arbitrary. In fact,
\begin{align*}
 [p,q]_\lambda &= p(0)\overline{q(0)} = \overline{[q,p]_\lambda} & \mbox{and} && \langle p,q\rangle_{\calF_\lambda} &= \overline{\langle q,p\rangle_{\calF_\lambda}}
\end{align*}
and \eqref{eq:FischerEqualsL2} follows from the previous considerations. Now assume \eqref{eq:FischerEqualsL2} holds for $\deg(q)\leq k$. For $\deg(q)\leq k+1$ we then have $\deg((\overline{a}|\calB_\lambda)q)\leq k$ and hence, by the assumption
\begin{align*}
 [(a|\textstyle\frac{z}{4})p,q]_\lambda &= [p,(\overline{a}|\calB_\lambda)q]_\lambda = \langle p,(\overline{a}|\calB_\lambda)q\rangle_{\calF_\lambda} = \langle(a,\textstyle\frac{z}{4})p,q\rangle_{\calF_\lambda}.
\end{align*}
This shows \eqref{eq:FischerEqualsL2} for $\deg(q)\leq k+1$ and $p(0)=0$, i.e. without constant term. But for constant $p$, i.e. $\deg(p)=0$, we have already seen that \eqref{eq:FischerEqualsL2} holds and therefore the proof is complete.
\end{proof}

\subsection{Unitary action on the Fock space}\label{sec:ActionOnFockSpace}

In Section \ref{sec:ComplexificationSchrödinger} we verified that for every $\lambda\in\calW$ the complexification $\td\pi_\lambda^\CC$ of the action $\td\pi_\lambda$ defines a Lie algebra representation of $\frakg$ on $C^\infty(\calX_\lambda)$ by holomorphic polynomial differential operators in $z$. It is clear that this action preserves the subspace $\calP(\calX_\lambda)$ of holomorphic polynomials. Using this we now construct an action of $\frakg$ on $\calP(\calX_\lambda)$ by composing the action $\td\pi_\lambda^\CC$ with the Cayley type transform $C\in\Int(\frakg_\CC)$ introduced in Section \ref{sec:ConformalGroup}.

\begin{definition}
Let $\lambda\in\calW$. On $\calP(\calX_\lambda)$ we define a $\frakg$-action $\td\rho_\lambda$ by
\begin{align*}
 \td\rho_\lambda(X) &:= \td\pi_\lambda^\CC(C(X)), & X\in\frakg.
\end{align*}
\end{definition}

\begin{proposition}\label{prop:FockKAction}
Let $\lambda\in\calW$ and $k\in\{0,\ldots,r\}$ such that $\calX_\lambda=\calX_k$. The $\frakk$-type decomposition of $(\td\rho_\lambda,\calP(\calX_\lambda))$ is given by
\begin{align*}
 \calP(\calX_\lambda) &= \bigoplus_{{\bf m}\geq0,m_{k+1}=0}^\infty{\calP_{\bf m}(\calX_\lambda)}
\end{align*}
and in every $\frakk$-type $\calP_{\bf m}(\calX_\lambda)$ the space of $\frakk^\frakl$-fixed vectors is one-dimensional and spanned by the polynomial $\Phi_{\bf m}(z)$. In particular $(\td\rho_\lambda,\calP(\calX_\lambda))$ is an admissible $(\frakg,\frakk)$-module.
\end{proposition}

\begin{proof}
The Cayley type transform $C:\frakg_\CC\to\frakg_\CC$ maps $\frakk_\CC$ to $\frakl_\CC$ and $\frakk^\frakl$ to itself. The action of $\frakl_\CC$ in the complexified Schr\"odinger model $\td\pi_\lambda^\CC$ is induced by the action of $L_\CC$ on the orbit $\calX_\lambda=L_\CC\cdot e_k$ up to multiplication by a character. Hence the $\frakk$-type decomposition of $\calP(\calX_\lambda)$ is the same as the decomposition into $L_\CC$-representations under the natural action $\ell$ of $L_\CC$. Therefore the claimed decomposition is clear by Corollary \ref{cor:HuaKostantSchmid}. The action of $\frakk^\frakl$ is induced by the natural action of $K^L$ and hence the unique (up to scalar) $\frakk^\frakl$-invariant vector in the $\frakk$-type $\calP_{\bf m}(\calX_\lambda)$ is the unique (up to scalar) $K^L$-invariant vector in $\calP_{\bf m}(\calX_\lambda)$ under the natural action $\ell$. This finishes the proof.
\end{proof}

\begin{proposition}\label{prop:FockIrreducible}
For each $\lambda\in\calW$ the representation $(\td\rho_\lambda,\calP(\calX_\lambda))$ is an irreducible $(\frakg,\frakk)$-module.
\end{proposition}

\begin{proof}
Let ${\bf m},{\bf n}\geq0$ be arbitrary and denote by $\calU(\frakg)$ the universal enveloping algebra of $\frakg$. Then it suffices to show that there exists an element of $\td\rho_\lambda(\calU(\frakg))$ which maps $\Phi_{\bf m}$ to $\Phi_{\bf n}$. Using Theorem \ref{thm:FischerEqualsL2} we have $0\neq\|\Phi_{\bf m}\|_{\calF_\lambda}^2=\left.\Phi_{\bf m}(\calB_\lambda)\Phi_{\bf m}(4z)\right|_{z=0}=4^{|{\bf m}|}\left.\Phi_{\bf m}(\calB_\lambda)\Phi_{\bf m}(z)\right|_{z=0}$. Since $\calB_\lambda$ is a differential operator of Euler degree $-1$ the polynomial $\Phi_{\bf m}(\calB_\lambda)\Phi_{\bf m}(z)$ is constant and by the previous observation it is non-zero. Note that $\td\rho_\lambda(\calU(\frakg))=\td\pi_\lambda^\CC(\calU(\frakg))$ contains multiplication by arbitrary polynomials and differential operators which are polynomials in the Bessel operator $\calB_\lambda$. Hence, the operator $\Phi_{\bf n}(z)\Phi_{\bf m}(\calB_\lambda)\in\td\rho_\lambda(\calU(\frakg))$ maps $\Phi_{\bf m}(z)$ to a non-zero multiple of $\Phi_{\bf n}(z)$ and the claim follows.
\end{proof}

\begin{proposition}\label{prop:FockInfUnitary}
The $(\frakg,\frakk)$-module $(\td\rho_\lambda,\calP(\calX_\lambda))$ is infinitesimally unitary with respect to the $L^2$-inner product $\langle-,-\rangle_{\calF_\lambda}$.
\end{proposition}

\begin{proof}
As remarked in Section \ref{sec:ConformalGroup} the subspace $\{(a,L(b),a):a,b\in V\}\subseteq\frakg$ generates $\frakg$ as a Lie algebra and therefore it suffices to show that the operators
\begin{align*}
 \td\rho_\lambda(a,0,a) &=\frac{1}{2}i(a|4\calB_\lambda+z) && \mbox{and} & \td\rho_\lambda(0,2L(a),0) &= \frac{1}{2}(a|4\calB_\lambda-z)
\end{align*}
are skew-adjoint on $\calP(\calX_\lambda)$ with respect to the $L^2$-inner product $\langle-,-\rangle_{\calF_\lambda}$. But this is clear by Proposition \ref{prop:BesselAdjoint}.
\end{proof}

\begin{theorem}\label{thm:UnitaryRepOnFock}
The $(\frakg,\frakk)$-module $(\td\rho_\lambda,\calP(\calX_\lambda))$ integrates to an irreducible unitary representation $\rho_\lambda$ of the universal cover $\widetilde{G}$ of $G$ on $\calF_\lambda$. This representation factors to a finite cover of $G$ if and only if $\lambda\in\QQ$. In particular it factors to a finite cover of $G$ if $\lambda\in\calW_{\disc}$.
\end{theorem}

\begin{proof}
By the previous results it only remains to check in which cases the minimal $\frakk$-type $\calP_{\bf0}(\calX_\lambda)=\CC\1$ integrates to a finite cover. The $\frakk$-action on $\1$ is given by
\begin{align*}
 \td\rho_\lambda(a,D,-a)\1 &= \td\pi_\lambda^\CC(0,D+2iL(a),0)\1 = \frac{r\lambda}{2n}\Tr(2iL(a))\1\\
 &= i\lambda\tr(a)\1.
\end{align*}
Therefore, the center $Z(\frakk)=\RR(e,0,-e)$ acts by
\begin{align*}
 \td\rho_\lambda(e,0,-e)\1 &= ir\lambda\1.
\end{align*}
In $K$ we have $e^{\pi(e,0,-e)}=\1$ and hence, the claim follows.
\end{proof}

In the Fock model the action of the maximal compact subgroup is quite explicit. For this recall the group homomorphism $\eta:\widetilde{K}\to U\subseteq L_\CC$ with differential $\td\eta(u,D,-u)=D+2iL(u)$ defined in Section \ref{sec:Complexifications} and the character $\xi_\lambda:\widetilde{K}\to\TT$ of $\widetilde{K}$ with differential $\td\xi_\lambda(u,T,-u)=i\lambda\tr(u)$ defined in Section \ref{sec:SchrödingerModel}.

\begin{proposition}\label{prop:FockModelKAction}
For $k\in\widetilde{K}$ we have
\begin{align*}
 \rho_\lambda(k)F(z) &= \xi_\lambda(k)F(\eta(k)^\# z), & z\in\calX_\lambda.
\end{align*}
\end{proposition}

\begin{proof}
The action of $X=(u,D,-u)\in\frakk$ is given by
\begin{align*}
 \td\rho_\lambda(X) &= \td\pi_\lambda^\CC(0,\td\eta(X),0) = \partial_{\td\eta(X)^\#}+\frac{r\lambda}{2n}\Tr(\td\eta(X)^\#)\\
 &= \partial_{\td\eta(X)^\#}+i\lambda\tr(u)
\end{align*}
which implies the claim.
\end{proof}

\subsection{The reproducing kernel}

We now calculate the reproducing kernel $\KK_\lambda(z,w)$ of the Hilbert space $\calF_\lambda$. For this we first calculate the reproducing kernels $\KK_\lambda^{\bf m}(z,w)$ on the finite-dimensional subspaces $\calP_{\bf m}(\calX_\lambda)$ endowed with the inner product of $\calF_\lambda$.

\begin{lemma}\label{lem:RepKernelEquivariance}
Let $\lambda\in\calW$ and ${\bf m}\geq0$. If $\lambda=k\frac{d}{2}$, $k=0,\ldots,r-1$, we additionally assume that $m_{k+1}=\ldots=m_r=0$. Then the following invariance property holds:
\begin{align*}
 \KK_\lambda^{\bf m}(gz,w) &= \KK_\lambda^{\bf m}(z,g^*w), & z,w\in\calX_\lambda,g\in L_\CC.
\end{align*}
\end{lemma}

\begin{proof}
Using Proposition \ref{prop:FockModelKAction} we find that for $k\in\widetilde{K}$ and $F\in\calP_{\bf m}(\calX_\lambda)$ we have
\begin{align*}
 \rho_\lambda(k)F(z) &= \langle\rho_\lambda(k)F,\KK_\lambda^{\bf m}(-,z)\rangle_{\calF_\lambda}\\
 &= \langle F,\rho_\lambda(k^{-1})\KK_\lambda^{\bf m}(-,z)\rangle_{\calF_\lambda}\\
 &= \overline{\xi_\lambda(k^{-1})}\langle F,\KK_\lambda^{\bf m}((\eta(k)^{-1})^\# -,z)\rangle_{\calF_\lambda}\\
 &= \xi_\lambda(k)\langle F,\KK_\lambda^{\bf m}((\eta(k)^{-1})^\# -,z)\rangle_{\calF_\lambda}
\intertext{and on the other hand}
 \rho_\lambda(k)F(z) &= \xi_\lambda(k)F(\eta(k)^\# z)\\
 &= \xi_\lambda(k)\langle F,\KK_\lambda^{\bf m}(-,\eta(k)^\# z)\rangle_{\calF_\lambda}.
\end{align*}
Since $\eta:\widetilde{K}\to U$ is surjective and $u^{-1}=u^*=\overline{u}^\#$ as well as $\overline{u}\in U$ for $u\in U$ we obtain
\begin{align*}
 \KK_\lambda^{\bf m}(uz,w) &= \KK_\lambda^{\bf m}(z,u^*w), & z,w\in\calX_\lambda, u\in U.
\end{align*}
Now both sides are holomorphic in $u\in L_\CC$ and $U\subseteq L_\CC$ is totally real. Hence the claim follows.
\end{proof}

\begin{proposition}\label{prop:RepKernelm}
Let $\lambda\in\calW$ and ${\bf m}\geq0$. If $\lambda=k\frac{d}{2}$, $k=0,\ldots,r-1$, we additionally assume that $m_{k+1}=\ldots=m_r=0$. Then
\begin{align}
 \KK_\lambda^{\bf m}(z,w) &= \frac{d_{\bf m}}{(\frac{n}{r})_{\bf m}(\lambda)_{\bf m}}\Phi_{\bf m}\left(\frac{z}{2},\frac{w}{2}\right), & z,w\in\calX_\lambda.\label{eq:RepKernelm}
\end{align}
\end{proposition}

\begin{proof}
First let $\lambda>(r-1)\frac{d}{2}$. By Lemma \ref{lem:RepKernelEquivariance} the function $K_\lambda^{\bf m}(-,e)\in\calP_{\bf m}(\calX_\lambda)$ is $K^L$-invariant and hence there is a constant $c_\lambda^{\bf m}$ such that $K_\lambda^{\bf m}(z,e)=c_\lambda^{\bf m}\Phi_{\bf m}(z)$. Since
\begin{align*}
 1 &= \Phi_{\bf m}(e) = \langle\Phi_{\bf m},K_\lambda^{\bf m}(-,e)\rangle_{\calF_\lambda} = c_\lambda^{\bf m}\|\Phi_{\bf m}\|_{\calF_\lambda}^2,
\end{align*}
formula \eqref{eq:RepKernelm} now follows with Proposition \ref{prop:FockNormCalculations}. Now, for $\lambda>(r-1)\frac{d}{2}$ this implies the following identity:
\begin{align}
 p(z) &= [p,\KK^{\bf m}_\lambda(-,z)]_\lambda = \frac{d_{\bf m}}{(\frac{n}{r})_{\bf m}(\lambda)_{\bf m}}\cdot\left.p(\calB_\lambda)_w\Phi_{\bf m}(w,\overline{z})\right|_{w=0}\label{eq:RepKernelPropertyWithFischer}
\end{align}
for $z\in V_\CC$ and $p\in\calP_{\bf m}(V_\CC)$. The right hand side is clearly meromorphic in $\lambda$ with possible poles at the points where $(\lambda)_{\bf m}=0$. For $\lambda=k\frac{d}{2}$, $k\in\{0,\ldots,r-1\}$, and ${\bf m}\geq0$ with $m_{k+1}=\ldots=m_r=0$ we have $(\lambda)_{\bf m}\neq0$ and hence the identity \eqref{eq:RepKernelPropertyWithFischer} holds for such $\lambda$ and ${\bf m}$ by analytic continuation. Since the reproducing kernel of $\calP_{\bf m}(\calX_\lambda)$ is uniquely determined by \eqref{eq:RepKernelPropertyWithFischer} the claim now follows also for $\lambda=k\frac{d}{2}$.
\end{proof}

\begin{theorem}\label{thm:FockRepKernel}
The reproducing kernel $\KK_\lambda(z,w)$ of the Hilbert space $\calF_\lambda$ is given by
\begin{align*}
 \KK_\lambda(z,w) &= \calI_\lambda\left(\frac{z}{2},\frac{w}{2}\right), & z,w\in\calX_\lambda.
\end{align*}
\end{theorem}

\begin{proof}
Let $0\leq k\leq r$ be such that $\calX_\lambda=\calX_k$. By the previous result $\KK^{\bf m}_\lambda(z,w)$ is the reproducing kernel for the subspace $\calP_{\bf m}(\calX_\lambda)$. Further we know by Proposition \ref{prop:FockNormCalculations} that the spaces $\calP_{\bf m}(\calX_\lambda)$ are pairwise orthogonal. Therefore, by \cite[Proposition I.1.8]{Nee00}, the sum
\begin{align*}
 \sum_{{\bf m}\geq0,m_{k+1}=0}{\KK^m_\lambda(z,w)} = \sum_{{\bf m}\geq0,m_{k+1}=0}{\frac{d_{\bf m}}{(\frac{n}{r})_{\bf m}(\lambda)_{\bf m}}\Phi_{\bf m}\left(\frac{z}{2},\frac{w}{2}\right)} = \calI_\lambda\left(\frac{z}{2},\frac{w}{2}\right)
\end{align*}
converges pointwise to the reproducing kernel $\KK_\lambda(z,w)$ of the direct Hilbert sum of all subspaces $\calP_{\bf m}(\calX_\lambda)$ with ${\bf m}\geq0$, $m_{k+1}=0$. But this direct Hilbert sum is by definition $\calF_\lambda$ and the proof is complete.
\end{proof}

The following consequence is a standard result for reproducing kernel spaces and can e.g. be found in \cite[page 9]{Nee00}.

\begin{corollary}
For every $F\in\calF_\lambda$ and every $z\in\calX_\lambda$ we have
\begin{align*}
 |F(z)| \leq \calI_\lambda\left(\frac{z}{2},\frac{z}{2}\right)^{\frac{1}{2}}\|F\|_{\calF_\lambda}.
\end{align*}
\end{corollary}

\subsection{Rings of differential operators and associated varieties}\label{sec:RingsOfDiffOps}

We recall the definition of the associated variety of an admissible representation for the example $(\rho_\lambda,\calF_\lambda)$. Let $(\calU_k(\frakg))_{k\in\NN_0}$ denote the usual filtration of the universal enveloping algebra $\calU(\frakg)$ and form the corresponding graded algebra $\gr\calU(\frakg)$ which is by the Poincare--Birkhoff--Witt theorem naturally isomorphic to the symmetric algebra $S(\frakg_\CC)$. The underlying $(\frakg,\widetilde{K})$-module $X^\lambda=\calP(\calX_\lambda)$ of the representation $\rho_\lambda$ carries an action of $\calU(\frakg)$. For $m\in\NN_0$ further let $X^\lambda_m$ be the subspace of polynomials in $\calP(\calX_\lambda)$ of degree $\leq m$. Then $X^\lambda_0=\CC\1$ is $\widetilde{K}$-invariant and generates $X^\lambda$ as a $\calU(\frakg)$-module. We further have
\begin{align*}
 \td\rho_\lambda(\calU_k(\frakg))X^\lambda_m &= X^\lambda_{k+m}, & k,m\in\NN_0,
\end{align*}
i.e. the filtrations $(\calU_k(\frakg))_k$ and $(X^\lambda_m)_m$ are compatible. Thus the corresponding graded space
\begin{align*}
 \gr X^\lambda &= \bigoplus_{m=0}^\infty{X^\lambda_m/X^\lambda_{m-1}}
\end{align*}
is a module over $\gr\calU(\frakg)$. Consider the annihilator ideal
\begin{align*}
 J_\lambda &:= \Ann_{\gr\calU(\frakg)}(\gr X^\lambda)\subseteq\gr\calU(\frakg)\cong S(\frakg_\CC)\cong\CC[\frakg_\CC^*].
\end{align*}
Then the associated variety $\calV(\rho_\lambda)$ of $\rho_\lambda$ is by definition the affine subvariety of $\frakg_\CC^*$ consisting of the common zeros of $J_\lambda$. Since $\frakk_\CC$ lives in degree $1$ in $\gr\calU(\frakg)$, but leaves each $X_m^\lambda$ invariant, every element in $\calV(\rho_\lambda)$ vanishes on $\frakk_\CC$ and we can view $\calV(\rho_\lambda)$ as a subset of $\frakp_\CC^*$. Via the Killing form we identify $\frakp_\CC^*$ with $\frakp_\CC$ and view $\calV(\rho_\lambda)$ as a subset of $\frakp_\CC$. Then $\calV(\rho_\lambda)$ is a $K_\CC$-stable closed subvariety of $\frakp_\CC$ consisting of nilpotent elements and hence the union of finitely many nilpotent $K_\CC$-orbits (see \cite[Corollary 5.23]{Vog91}).

Recall from Section \ref{sec:NilpotentOrbits} the $K_\CC$-orbits $\calO_k^{K_\CC}\subseteq\frakp^+$ which are isomorphic to the $L_\CC$-orbits $\calX_k$ via the Cayley type transform $C\in\Int(\frakg_\CC)$. The following result is due to A. Joseph \cite[Theorem 7.14]{Jos92}:

\begin{proposition}\label{prop:AssVar}
Let $\lambda\in\calW$ and $0\leq k\leq r$ such that $\calX_\lambda=\calX_k$. Then $\calV(\rho_\lambda)=\overline{\calO_k^{K_\CC}}$.
\end{proposition}

\begin{corollary}
For $\lambda\in\calW$ let $k\in\{0,\ldots,r\}$ such that $\calX_\lambda=\calX_k$. Then the Gelfand--Kirillov dimension of $\rho_\lambda$ is $k+k(2r-k-1)\frac{d}{2}$.
\end{corollary}

\begin{proof}
The Gelfand--Kirillov dimension of an irreducible unitary representation equals the dimension of its associated variety in $\frakp_\CC^*$ (combine \cite[Corollary 4.7]{Vog78} and \cite[Theorem 8.4]{Vog91}). Therefore the result follows from \eqref{eq:OrbitDimension}.
\end{proof}

For an algebraic variety $\XX$ over $\CC$ denote by $\CC[\XX]$ the ring of regular functions. Further let $\DD(\XX)$ be the ring of algebraic differential operators on $\XX$. This subring of $\End_\CC(\CC[\XX])$ can be defined inductively as follows: Let $\DD_0(\XX):=\CC[\XX]$ be the ring of multiplication operators and for $m\in\NN$ put
\begin{align*}
 \DD_m(\XX) &:= \{D\in\End_\CC(\CC[\XX]):[D,f]\in\DD_{m-1}(\XX)\,\forall\,f\in\CC[\XX]\}.
\end{align*}
Then $\DD(\XX)=\bigcup_{m\in\NN_0}{\DD_m(\XX)}$.

Since the varieties $\overline{\calX_k}$ are affine it follows that $\CC[\overline{\calX_k}]=\calP(\calX_k)$. The representation $\td\rho_\lambda$ acts on $\calP(\calX_\lambda)$ by differential operators and hence it induces a map
\begin{equation*}
 \td\rho_\lambda:\calU(\frakg)\to\DD(\overline{\calX_\lambda}).
\end{equation*}
The following result is a qualitative version of \cite[Theorem 4.5]{Jos92}:

\begin{theorem}\label{thm:DiffOpSurjectivity}
For $\lambda=k\frac{d}{2}\in\calW_{\disc}$ the map $\td\rho_\lambda:\calU(\frakg)\to\DD(\overline{\calX_k})$ is surjective and induces an isomorphism $\calU(\frakg)/\calJ_k\cong\DD(\overline{\calX_k})$, where $\calJ_k=\Ann_{\calU(\frakg)}(X^\lambda)$.
\end{theorem}

\begin{proof}
By \cite[Theorem 4.5]{Jos92} the map $\td\rho_\lambda$ is surjective onto the space of $\CC$-endomorphisms of $\calP(\calX_\lambda)=\CC[\overline{\calX_k}]$ which are locally finite under the diagonal action of $\frakp^-$. Now $\frakp^-$ acts by multiplication with coordinate functions and hence the condition for $D\in\End_\CC(\CC[\overline{\calX_k}])$ to be locally finite under the action of $\frakp^-$ is equivalent to the existence of $N\in\NN$ such that the iterated commutator $[[\ldots[D,f_1(x)],\ldots,f_{N-1}(x)],f_N(x)]=0$ for all $f_1,\ldots,f_N\in\CC[\overline{\calX_k}]$. This again is equivalent to $D\in\DD(\overline{\calX_k})$ and the proof is complete.
\end{proof}

\begin{remark}
A quantitative version of Theorem \ref{thm:DiffOpSurjectivity} was obtained by Levasseur--Smith--Stafford \cite{LSS88} for the minimal orbit $\calX_1$, by Levasseur--Stafford \cite{LS89} for classical $\frakg$ and finally by Joseph \cite{Jos92} for the general case. However, their version is less explicit and does not provide a geometric construction of the unitary structure.
\end{remark}

\begin{corollary}\label{cor:GeneratorsAlgDiffOp}
For $k=1,\ldots,r-1$ the ring $\DD(\overline{\calX_k})$ of algebraic differential operators on the affine variety $\overline{\calX_k}$ is generated by the multiplications with regular functions in $\CC[\overline{\calX_k}]$ and the Bessel operators $(u|\calB_\lambda)$, $u\in V_\CC$, for $\lambda=k\frac{d}{2}$.
\end{corollary}

\begin{proof}
By Theorem \ref{thm:DiffOpSurjectivity} the ring $\DD(\overline{\calX_k})$ is generated by the constants and $\td\rho_\lambda(\frakg)$. Note that in the decomposition $\frakg_\CC=\frakp^-\oplus\frakk_\CC\oplus\frakp^+$ the subalgebra $\frakk_\CC$ is generated by $\frakp^+$ and $\frakp^-$ and hence $\DD(\overline{\calX_k})$ is generated by the constants, $\td\rho_\lambda(\frakp^+)=\{(u|\calB_\lambda):u\in V_\CC\}$ and $\td\rho_\lambda(\frakp^-)=\{(v|z):v\in V_\CC\}$.
\end{proof}

\begin{remark}
Neither Theorem \ref{thm:DiffOpSurjectivity} nor Corollary \ref{cor:GeneratorsAlgDiffOp} can hold for $\lambda\in\calW_\cont$ resp. $k=r$ since in this case $\overline{\calX_\lambda}=\overline{\calX_r}=V_\CC$ and $\DD(V_\CC)=\CC[x,\frac{\partial}{\partial x}]$ is a Weyl algebra (see \cite[Lemma IV.1.5]{LS89}).
\end{remark}

\newpage
\section{The Segal--Bargmann transform}\label{sec:BargmannTransform}

For every $\lambda\in\calW$ we explicitly construct an intertwining operator, the \textit{Segal--Bargmann transform}, between the Schr\"odinger model $(\pi_\lambda,L^2(\calO_\lambda,\td\mu_\lambda))$ and the Fock model $(\rho_\lambda,\calF_\lambda)$ in terms of its integral kernel.

\subsection{Construction of the Segal--Bargmann transform}

For each $\lambda\in\calW$ the \textit{Segal--Bargmann transform} $\BB_\lambda$ is defined for $\psi\in L^2(\calO_\lambda,\td\mu_\lambda)$ by
\begin{align}
 \BB_\lambda\psi(z) &:= e^{-\frac{1}{2}\tr(z)}\int_{\calO_\lambda}{\calI_\lambda(z,x)e^{-\tr(x)}\psi(x)\td\mu_\lambda(x)}, & z\in V_\CC.\label{eq:DefBargmann}
\end{align}
Recall the space $\calO(V_\CC)$ of holomorphic functions on $V_\CC$.

\begin{proposition}\label{prop:BargmannMappingProperty}
For $\psi\in L^2(\calO_\lambda,\td\mu_\lambda)$ the integral in \eqref{eq:DefBargmann} converges uniformly on bounded subsets and defines a function $\BB_\lambda\psi\in\calO(V_\CC)$. The Segal--Bargmann transform $\BB_\lambda$ is a continuous linear operator
\begin{align*}
 L^2(\calO_\lambda,\td\mu_\lambda)\to\calO(V_\CC).
\end{align*}
\end{proposition}

\begin{proof}
Since the kernel function $e^{-\frac{1}{2}\tr(z)}\calI_\lambda(z,x)e^{-\tr(x)}$ is analytic in $z$, it suffices to show that its $L^2$-norm in $x$ has a uniform bound for $|z|\leq R$, $R>0$. By Lemma \ref{lem:JBesselConvergence} there exists $C>0$ such that
\begin{align*}
 |\calI_\lambda(z,x)| &\leq C(1+|z|\cdot|x|)^{\frac{r(2n-1)}{4}}e^{2r\sqrt{|z|\cdot|x|}}.
\end{align*}
Then for $x\in\calO_\lambda$, $z\in V_\CC$ with $|z|\leq R$, we find
\begin{align*}
 |e^{-\frac{1}{2}\tr(z)}\calI_\lambda(z,x)e^{-\tr(x)}| &\leq C'(1+R\cdot|x|)^{\frac{r(2n-1)}{4}}e^{2r\sqrt{R\cdot|x|}}e^{-\tr(x)}\\
 &\leq C'(1+R\cdot|x|)^{\frac{r(2n-1)}{4}}e^{2r\sqrt{R}\sqrt{|x|}-|x|}
\end{align*}
with $C'=C\max_{|z|\leq R}{|e^{-\frac{1}{2}\tr(z)}|}$. This is certainly $L^2$ as a function of $x\in\calO_\lambda$ with norm independent of $z$ and the claim follows.
\end{proof}

Next, we show that $\BB_\lambda$ intertwines the action $\td\pi_\lambda$ on $L^2(\calO_\lambda,\td\mu_\lambda)$ with the action $\td\rho_\lambda$ on $\calF_\lambda$. Recall the Cayley type transform $C\in\Int(\frakg_\CC)$ introduced in Section \ref{sec:ConformalGroup}.

\begin{proposition}\label{prop:BargmannIntertwiningProperty}
The following intertwining identity holds on $L^2(\calO_\lambda,\td\mu_\lambda)^\infty$:
\begin{align}
 \BB_\lambda\circ\td\pi_\lambda(X) &= \td\pi_\lambda^\CC(C(X))\circ\BB_\lambda, & X\in\frakg.\label{eq:BargmannIntertwiningProperty}
\end{align}
\end{proposition}

\begin{proof}
As remarked in Section \ref{sec:ConformalGroup} the subspace $\{(a,L(b),a):a,b\in V\}$ generates $\frakg$ as a Lie algebra. Hence it suffices to prove \eqref{eq:BargmannIntertwiningProperty} for the elements $(a,\pm2iL(a),a)$, $a\in V$. We show \eqref{eq:BargmannIntertwiningProperty} for $X=(a,-2iL(a),a)$, $a\in V$, the proof for $(a,+2iL(a),a)$ works similarly. For $X=(a,-2iL(a),a)$ we have $C(X)=(a,0,0)$ and hence
\begin{align*}
 & (\td\pi_\lambda^\CC(C(X))\circ\BB_\lambda)\psi(z)\\
 ={}& i(a|z)e^{-\frac{1}{2}\tr(z)}\int_{\calO_\lambda}{\calI_\lambda(z,x)e^{-\tr(x)}\psi(x)\td\mu_\lambda(x)}.
\intertext{By Proposition \ref{prop:DiffEqIBessel} we have $z\calI_\lambda(z,x)=(\calB_\lambda)_x\calI_\lambda(z,x)$. Further the Bessel operator $\calB_\lambda$ is symmetric on $L^2(\calO_\lambda,\td\mu_\lambda)$ and we obtain}
 ={}& ie^{-\frac{1}{2}\tr(z)}\int_{\calO_\lambda}{(a|\calB_\lambda)_x\calI_\lambda(z,x)e^{-\tr(x)}\psi(x)\td\mu_\lambda(x)}\\
 ={}& ie^{-\frac{1}{2}\tr(z)}\int_{\calO_\lambda}{\calI_\lambda(z,x)(a|\calB_\lambda)\left[e^{-\tr(x)}\psi(x)\right]\td\mu_\lambda(x)}.
\end{align*}
By the product rule \eqref{eq:BesselProdRule} for the Bessel operator we obtain
\begin{align*}
 & \calB_\lambda\left[e^{-\tr(x)}\psi(x)\right]\\
 ={}& \calB_\lambda e^{-\tr(x)}\cdot\psi(x)+2P\left(\frac{\partial}{\partial x}e^{-\tr(x)},\frac{\partial\psi}{\partial x}\right)x+e^{-\tr(x)}\cdot\calB_\lambda\psi(x)\\
 ={}& (x-\lambda e)e^{-\tr(x)}\psi(x)-2e^{-\tr(x)}x\cdot\frac{\partial\psi}{\partial x}+e^{-\tr(x)}\calB_\lambda\psi(x)
\end{align*}
since
\begin{align*}
 \calB_\lambda e^{-\tr(x)} &= (P(-e)x-\lambda e)e^{-\tr(x)}=(x-\lambda e)e^{-\tr(x)} & \mbox{and}\\
 \frac{\partial}{\partial x}e^{-\tr(x)} &= -e\cdot e^{-\tr(x)}.
\end{align*}
Hence we have
\begin{multline*}
 (a|\calB_\lambda)\left[e^{-\tr(x)}\psi(x)\right]\\
 = e^{-\tr(x)}\left[(a|x)\psi(x)-\lambda\tr(a)\psi(x)-2\partial_{ax}\psi(x)+\calB_\lambda\psi(x)\right].
\end{multline*}
Inserting this into our calculation above we find
\begin{align*}
 & (\td\pi_\lambda^\CC(C(X))\circ\BB_\lambda)\psi(z)\\
 ={}& \BB_\lambda\circ\left(i(a|x)-2i\left[\partial_{L(a)x}+\frac{r\lambda}{2n}\Tr(L(a))\right]+i(a|\calB_\lambda)\right)\psi(z)\\
 ={}& (\BB_\lambda\circ\td\pi_\lambda(X))\psi(z).\qedhere
\end{align*}
\end{proof}

To conclude that $\BB_\lambda$ is an isomorphism $L^2(\calO_\lambda,\td\mu_\lambda)\to\calF_\lambda$, we show that it maps the underlying $(\frakg,\frakk)$-module $W^\lambda\subseteq L^2(\calO_\lambda,\td\mu_\lambda)$ to the $(\frakg,\frakk)$-module $\calP(\calX_\lambda)$. In order to do so we show that the function
\begin{align*}
 \Psi_0 &:= \BB_\lambda\psi_0\in\calO(V_\CC)
\end{align*}
is $L_\CC$-invariant. In fact, the function $\psi_0$ is $\widetilde{K}$-equivariant by the character $\xi_\lambda$ (see Section \ref{sec:SchrödingerModel}). By Proposition \ref{prop:BargmannIntertwiningProperty} the same has to be true for $\Psi_0$. But by Proposition \ref{prop:FockModelKAction} this implies that $\Psi_0$ is invariant under $\eta(\widetilde{K})^\#=U$. Now $U$ is a real form of $L_\CC$ and the action of $L_\CC$ on $\Psi_0$ is holomorphic, whence $\Psi_0$ has to be $L_\CC$-invariant. Therefore it has to be constant on every $L_\CC$-orbit. Since $\Psi_0$ is holomorphic on $V_\CC$ and $V_\CC$ decomposes into finitely many $L_\CC$-orbits, it follows that $\Psi_0$ is constant on $V_\CC$. It remains to show that $\Psi_0$ is non-zero.

\begin{proposition}
$\Psi_0(0) = 1$ and hence $\Psi_0=\BB_\lambda\psi_0=\1$.
\end{proposition}

\begin{proof}
Since $\calI_\lambda(0,x)=1$ we have
\begin{align*}
 \Psi_0(0) &= \int_{\calO_\lambda}{e^{-2\,\tr(x)}\td\mu_\lambda(x)} = \|\psi_0\|^2_{L^2(\calO_\lambda,\td\mu_\lambda)} = 1
\end{align*}
as shown in Section \ref{sec:StructureGroup}.
\end{proof}

\begin{theorem}\label{thm:BargmannIsomorphism}
$\BB_\lambda$ is a unitary isomorphism $L^2(\calO_\lambda,\td\mu_\lambda)\to\calF_\lambda$ intertwining the actions $\pi_\lambda$ and $\rho_\lambda$.
\end{theorem}

\begin{proof}
Since $W^\lambda=L^2(\calO_\lambda,\td\mu_\lambda)_\frakk$ is generated by $\psi_0$, $\calP(\calX_\lambda)=(\calF_\lambda)_\frakk$ is generated by $\Psi_0=\1$, $\BB_\lambda\psi_0=\Psi_0$ and $\BB_\lambda$ intertwines the actions $\td\pi_\lambda$ and $\td\rho_\lambda$, it has to map the irreducible $(\frakg,\frakk)$-module $W^\lambda$ into the irreducible $(\frakg,\frakk)$-module $\calP(\calX_\lambda)$ and is therefore, thanks to Schur's Lemma, an isomorphism between the underlying $(\frakg,\frakk)$-modules. It only remains to show that $\BB_\lambda$ is isometric between $W^\lambda$ and $\calP(\calX_\lambda)$, then the statement follows since $W^\lambda\subseteq L^2(\calO_\lambda,\td\mu_\lambda)$ and $\calP(\calX_\lambda)\subseteq\calF_\lambda$ are dense.\\
Both $W^\lambda$ and $\calP(\calX_\lambda)$ are irreducible infinitesimally unitary $(\frakg,\frakk)$-modules and hence $\BB_\lambda$ is a scalar multiple of a unitary operator. Since further
\begin{align*}
 \|\BB_\lambda\psi_0\|_{\calF_\lambda} &= \|\1\|_{\calF_\lambda} = 1 = \|\psi_0\|_{L^2(\calO_\lambda,\td\mu_\lambda)}
\end{align*}
the operator $\BB_\lambda$ itself has to be unitary.
\end{proof}

\begin{remark}
The Segal--Bargmann transform can also be obtained via a \textit{restriction principle} (see \cite{HZ09,OO96} for other instances of this principle). The formula $\calR_\lambda F(x)=e^{-\frac{1}{2}\tr(x)}F(x)$ defines an operator $\calP(\calX_\lambda)\to L^2(\calO_\lambda,\td\mu_\lambda)$ and hence we obtain a densely defined unbounded operator $\calR_\lambda:\calF_\lambda\to L^2(\calO_\lambda,\td\mu_\lambda)$. We consider its adjoint $\calR_\lambda^*:L^2(\calO_\lambda,\td\mu_\lambda)\to\calF_\lambda$ as a densely defined unbounded operator. One can show that the Segal--Bargmann transform appears as the unitary part in the polar decomposition of the operator $\calR_\lambda^*$:
\begin{align*}
 \calR_\lambda^* &= \BB_\lambda\circ\sqrt{\calR_\lambda\calR_\lambda^*}.
\end{align*}
For the case of the minimal discrete Wallach point this is done in \cite[Proposition 5.5]{HKMO12}.
\end{remark}

\begin{corollary}\label{cor:InverseBargmann}
The inverse Segal--Bargmann transform is given by
\begin{align*}
 \BB_\lambda^{-1}F(x) &= \frac{e^{-\tr(x)}}{c_\lambda}\int_{\calX_\lambda}{\calI_\lambda(x,z)e^{-\frac{1}{2}\tr(\overline{z})}F(z)\omega_\lambda(z)\td\nu_\lambda(z)}, & F\in\calF_\lambda.
\end{align*}
\end{corollary}

\begin{proof}
Since the Segal--Bargmann transform is a unitary operator we have
\begin{align*}
 & \langle\BB_\lambda^{-1}F,\psi\rangle_{L^2(\calO_\lambda,\td\mu_\lambda)} = \langle F,\BB_\lambda\psi\rangle_{\calF_\lambda} = \frac{1}{c_\lambda}\int_{\calX_\lambda}{F(z)\overline{\BB_\lambda\psi(z)}\omega_\lambda(z)\td\nu_\lambda(z)}\\
 ={}& \frac{1}{c_\lambda}\int_{\calX_\lambda}{\int_{\calO_\lambda}{F(z)e^{-\frac{1}{2}\overline{\tr(z)}}\overline{\calI_\lambda(z,x)}e^{-\tr(x)}\overline{\psi(x)}\td\mu_\lambda(x)}\omega_\lambda(z)\td\nu_\lambda(z)}\\
 ={}& \frac{1}{c_\lambda}\int_{\calO_\lambda}{\left(e^{-\tr(x)}\int_{\calX_\lambda}{e^{-\frac{1}{2}\tr(\overline{z})}\calI_\lambda(x,z)F(z)\omega_\lambda(z)\td\nu_\lambda(z)}\right)\overline{\psi(x)}\td\mu_\lambda(x)}
\end{align*}
which implies the claim.
\end{proof}

We now use the Segal--Bargmann transform to obtain an intrinsic description of the Fock space.

\begin{theorem}\label{thm:FockAsRestriction}
$$ \calF_\lambda = \left\{F|_{\calX_\lambda}:F\in\calO(V_\CC),\int_{\calX_\lambda}{|F(z)|^2\omega_\lambda(z)\td\nu_\lambda(z)}<\infty\right\}. $$
\end{theorem}

\begin{proof}
The Segal--Bargmann transform is an isomorphism $\BB_\lambda:L^2(\calO_\lambda,\td\mu_\lambda)\to\calF_\lambda$ and hence by Proposition \ref{prop:BargmannMappingProperty} every function in $\calF_\lambda$ extends to a holomorphic function on $V_\CC$. This shows the inclusion $\subseteq$.\\
For the other inclusion let $F\in\calO(V_\CC)$ such that
\begin{equation*}
 \int_{\calX_\lambda}{|F(z)|^2\omega_\lambda(z)\td\nu_\lambda(z)}<\infty.
\end{equation*}
We expand $F$ into a power series which we can arrange as
\begin{equation*}
 F=\sum_{{\bf m}\geq0}{p_{\bf m}} \quad \mbox{with} \quad p_{\bf m}\in\calP_{\bf m}(V_\CC)
\end{equation*}
by Theorem \ref{thm:HuaKostantSchmid}. Since $p_{\bf m}|_{\calX_\lambda}=0$ for $m_{k+1}\neq0$ by Corollary \ref{cor:HuaKostantSchmid} we may assume $F=\sum_{{\bf m}\geq0,\,m_{k+1}=0}{p_{\bf m}}$ for the study of $F|_{\calX_\lambda}$. This series converges uniformly on compact subsets. We show that this series also converges in $L^2(\calX_\lambda,\omega_\lambda\td\nu_\lambda)$. For $R>0$ let
\begin{equation*}
 \calX_\lambda^R := \{z\in\calX_\lambda:|z|\leq R\}, \qquad \calO_\lambda^R := \{x\in\calO_\lambda:|x|\leq R\}.
\end{equation*}
Note that $\calX_\lambda^R$ and $\calO_\lambda^R$ are compact and hence integration over these sets commutes with taking the limit $F=\sum_{{\bf m}\geq0,\,m_{k+1}=0}{p_{\bf m}}$. With \eqref{eq:IntFormulaOnXlambda} we find
\begin{align*}
 \infty &> \int_{\calX_\lambda}{|F(z)|^2\omega_\lambda(z)\td\nu_\lambda(z)}\\
 &= \lim_{R\to\infty}{\int_{\calX_\lambda^R}{|F(z)|^2\omega_\lambda(z)\td\nu_\lambda(z)}}\\
 &= \lim_{R\to\infty}{\sum_{\substack{{\bf m},{\bf n}\geq0,\\m_{k+1}=n_{k+1}=0}}{\int_{\calO_\lambda^R}{\int_U{p_{\bf m}(ux^{\frac{1}{2}})\overline{p_{\bf n}(ux^{\frac{1}{2}})}\td u}\omega_\lambda(x^{\frac{1}{2}})\td\mu_\lambda(x)}}}.\\
\intertext{By Theorem \ref{thm:CartanHelgason} and Lemma \ref{lem:SchurOrthogonalityLemma} the integrals over $U$ for ${\bf m}\neq{\bf n}$ vanish and we obtain}
 &= \lim_{R\to\infty}{\sum_{{\bf m}\geq0,\,m_{k+1}=0}{\int_{\calX_\lambda^R}{|p_{\bf m}(z)|^2\omega_\lambda(z)\td\nu_\lambda(z)}}}\\
 &= \sum_{{\bf m}\geq0,\,m_{k+1}=0}{\|p_{\bf m}\|_{L^2(\calX_\lambda,\omega_\lambda\td\nu_\lambda)}^2}
\end{align*}
which shows convergence in $L^2(\calX_\lambda,\omega_\lambda\td\nu_\lambda)$. Since $\calF_\lambda$ is the closure of the space of polynomials with respect to the norm of $L^2(\calX_\lambda,\omega_\lambda\td\nu_\lambda)$ it is clear that $F|_{\calX_\lambda}\in\calF_\lambda$ which shows the other inclusion.
\end{proof}

This intrinsic description leads to the following conjecture which was proved in \cite[Theorem 2.26]{HKMO12} for the minimal orbit, i.e. $\lambda\in\calW$ with $\calX_\lambda=\calX_1$:

\begin{conjecture}
\begin{align*}
 \calF_\lambda &= \left\{F\in\calO(\calX_\lambda):\int_{\calX_\lambda}{|F(z)|^2\omega_\lambda(z)\td\nu_\lambda(z)}<\infty\right\}.
\end{align*}
\end{conjecture}

Recall the Laguerre functions $\ell_{\bf m}^\lambda$ introduced in Section \ref{sec:Polynomials}.

\begin{proposition}
Let $\lambda\in\calW$ and $k\in\{0,\ldots,r\}$ such that $\calO_\lambda=\calO_k$. Then for every ${\bf m}\geq0$ with $m_{k+1}=0$ we have
\begin{align*}
 \BB_\lambda\ell_{\bf m}^\lambda &= \frac{(-1)^{|{\bf m}|}}{2^{|{\bf m}|}}\Phi_{\bf m}.
\end{align*}
\end{proposition}

\begin{proof}
$\ell_{\bf m}^\lambda$ is the unique (up to scalar multiples) $\frakk^\frakl$-invariant vector in the $\frakk$-type $W^\lambda_{\bf m}\subseteq L^2(\calO_\lambda,\td\mu_\lambda)$ whereas $\Phi_{\bf m}$ is the unique (up to scalar multiples) $\frakk^\frakl$-invariant vector in the $\frakk$-type $\calP_{\bf m}(\calX_\lambda)\subseteq\calF_\lambda$. Hence
\begin{equation*}
 \BB_\lambda\ell_{\bf m}^\lambda=\const\cdot\Phi_{\bf m}.
\end{equation*}
To find the constant we evaluate $\BB_\lambda^{-1}\Phi_{\bf m}(x)$ and $\ell_{\bf m}^\lambda(x)$ at $x=0$. First observe that
\begin{align*}
 \ell_{\bf m}^\lambda(0) &= L_{\bf m}^\lambda(0) = (\lambda)_{\bf m}.
\end{align*}
Next we calculate $\BB_\lambda^{-1}\Phi_{\bf m}(0)$ using Corollary \ref{cor:InverseBargmann}:
\begin{align*}
 \BB_\lambda^{-1}\Phi_{\bf m}(0) &= \frac{1}{c_\lambda}\int_{\calX_\lambda}{e^{-\frac{1}{2}\tr(\overline{z})}\Phi_{\bf m}(z)\omega_\lambda(z)\td\nu_\lambda(z)}.
\end{align*}
By \eqref{eq:ExponentialTraceExpansion} the function $e^{\tr(z)}$ has the following expansion:
\begin{align*}
 e^{\tr(z)} &= \sum_{{\bf m}\geq0}{\frac{d_{\bf m}}{(\frac{n}{r})_{\bf m}}\Phi_{\bf m}(z)}.
\end{align*}
Inserting this yields
\begin{align*}
 \BB_\lambda^{-1}\Phi_{\bf m}(0) &= \sum_{{\bf n}\geq0}{(-1)^{|{\bf n}|}\frac{d_{\bf n}}{2^{|{\bf n}|}(\frac{n}{r})_{\bf n}}\cdot\frac{1}{c_\lambda}\int_{\calX_\lambda}{\Phi_{\bf m}(z)\overline{\Phi_{\bf n}(z)}\omega_\lambda(z)\td\nu_\lambda(z)}}\\
 &= \sum_{{\bf n}\geq0}{(-1)^{|{\bf n}|}\frac{d_{\bf n}}{2^{|{\bf n}|}(\frac{n}{r})_{\bf n}}\langle\Phi_{\bf m},\Phi_{\bf n}\rangle_{\calF_\lambda}}\\
 &= (-1)^{|{\bf m}|}\frac{d_{\bf m}}{2^{|{\bf m}|}(\frac{n}{r})_{\bf m}}\cdot\frac{4^{|{\bf m}|}(\tfrac{n}{r})_{\bf m}(\lambda)_{\bf m}}{d_{\bf m}}\\
 &= (-1)^{|{\bf m}|}2^{|{\bf m}|}(\lambda)_{\bf m},
\end{align*}
where we have used Proposition \ref{prop:FockNormCalculations}.
\end{proof}

\subsection{Intertwiner to the bounded symmetric domain model}\label{sec:IntertwinerBdSymDomainModel}

The Fock space $\calF_\lambda$ and the bounded symmetric domain model $\calH^2_\lambda(\calD)$ have (as vector spaces) the same underlying $(\frakg,\widetilde{K})$-module, namely
\begin{align*}
 \bigoplus_{{\bf m}\geq0,\,m_{k+1}=0}{\calP_{\bf m}(V_\CC)},
\end{align*}
where $k\in\{0,\ldots,r\}$ is such that $\calX_\lambda=\calX_k$. However, inner product and group action differ. In Remark \ref{rem:ComparisonNormsFockBdSymDomain} we already compared the norms in the two models. To gain a better understanding of the relation between the two models we investigate the intertwiner between them.

For this we compose the inverse Segal--Bargmann transform $\BB_\lambda^{-1}$ with the Laplace transform $\calL_\lambda$ and the pullback of the Cayley transform $\gamma_\lambda$ to obtain an intertwining operator between the Fock space picture and the realization on functions on the bounded symmetric domain. (For the definition of $\calL_\lambda$ and $\gamma_\lambda$ see Sections \ref{sec:TubeDomainModel} and \ref{sec:BoundedSymmetricDomainModel}.) Let
\begin{align*}
 \calA_\lambda &:= \gamma_\lambda\circ\calL_\lambda\circ\BB_\lambda^{-1}:\calF_\lambda\to\calH_\lambda^2(\calD).
\end{align*}
The operator $\calA_\lambda$ intertwines the actions $\rho_\lambda$ and $\pi_\lambda^\calD$.

\begin{theorem}\label{thm:IntertwinerFockBdSymDomain}
For $F\in\calF_\lambda$ we have
\begin{align*}
 \calA_\lambda F(z) &= \frac{1}{c_\lambda}\int_{\calX_\lambda}{e^{-\frac{1}{2}(z|\overline{w})}F(w)\omega_\lambda(w)\td\nu_\lambda(w)}, & z\in\calD.
\end{align*}
\end{theorem}

\begin{proof}
Using Lemma \ref{lem:IBesselIntFormula} we obtain
\begin{align*}
 & \calA_\lambda F(z) = \gamma_\lambda\circ\calL_\lambda\circ\BB_\lambda^{-1}F(z) = \Delta(e-z)^{-\lambda}\calL_\lambda\circ\BB_\lambda^{-1}F(c(z))\\
 ={}& \Delta(e-z)^{-\lambda}\int_{\calO_\lambda}{e^{i(c(z)|u)}\BB_\lambda^{-1}F(u)\td\mu_\lambda(u)}\\
 ={}& \frac{\Delta(e-z)^{-\lambda}}{c_\lambda}\int_{\calO_\lambda}{\int_{\calX_\lambda}{e^{i(c(z)|u)}e^{-\tr(u)}\calI_\lambda(u,w)e^{-\frac{1}{2}\tr(\overline{w})}F(w)}}\\
 & \hspace{7.5cm}{{\omega_\lambda(w)\td\nu_\lambda(w)}\td\mu_\lambda(u)}\\
 ={}& \frac{\Delta(e-z)^{-\lambda}}{c_\lambda}\int_{\calX_\lambda}{\left(\int_{\calO_\lambda}{e^{-(u|e-ic(z))}\calI_\lambda(u,w)\td\mu_\lambda(u)}\right)e^{-\frac{1}{2}\tr(\overline{w})}F(w)}\\
 & \hspace{8.7cm}{\omega_\lambda(w)\td\nu_\lambda(w)}\\
 ={}& \frac{2^{r\lambda}\Delta(e-z)^{-\lambda}}{c_\lambda}\int_{\calX_\lambda}{\Delta(e-ic(z))^{-\lambda}e^{((e-ic(z))^{-1}|\overline{w})}e^{-\frac{1}{2}\tr(\overline{w})}F(w)}\\
 & \hspace{8.7cm}{\omega_\lambda(w)\td\nu_\lambda(w)}.
\end{align*}
Now
\begin{align*}
 e-ic(z) &= e+(e+z)(e-z)^{-1} = 2(e-z)^{-1}
\end{align*}
and therefore
\begin{align*}
 \Delta(e-ic(z))^{-\lambda} &= 2^{-r\lambda}\Delta(e-z)^\lambda.
\end{align*}
This also implies that
\begin{align*}
 ((e-ic(z))^{-1}|\overline{w}) &= \frac{1}{2}\tr(\overline{w})-\frac{1}{2}(z|\overline{w})
\end{align*}
so that altogether we obtain
\begin{align*}
 \calA_\lambda F(z) &= \frac{1}{c_\lambda}\int_{\calX_\lambda}{e^{-\frac{1}{2}(z|\overline{w})}F(w)\omega_\lambda(w)\td\nu_\lambda(w)}.\qedhere
\end{align*}
\end{proof}

\begin{corollary}
The operator $\calA_\lambda$ acts on $\calP_{\bf m}(V_\CC)$ by the scalar $\frac{(-1)^{|{\bf m}|}(\lambda)_{\bf m}}{2^{|{\bf m}|}}$.
\end{corollary}

\begin{proof}
Using the expansion \eqref{eq:ExponentialTraceExpansion} and Proposition \ref{prop:RepKernelm} we obtain
\begin{align*}
 e^{(z|\overline{w})} &= \sum_{{\bf n}\geq0}{(\lambda)_{\bf n}\KK_\lambda^{\bf n}(z,w)}.
\end{align*}
Hence for $F\in\calP_{\bf m}(V_\CC)$ we find
\begin{align*}
 \calA_\lambda F(z) &= \sum_{{\bf n}\geq0}{(\lambda)_{\bf n}\langle F,\KK_\lambda^{\bf n}(-,-{\textstyle\frac{1}{2}}z)\rangle_{\calF_\lambda}}\\
 &= (\lambda)_{\bf m}F(-{\textstyle{\frac{1}{2}}z}) = \frac{(-1)^{|{\bf m}|}(\lambda)_{\bf m}}{2^{|{\bf m}|}}F(z).\qedhere
\end{align*}
\end{proof}

\newpage
\section{The unitary inversion operator}\label{sec:UnitaryInversionOperator}

Using the Segal--Bargmann transform we now calculate the integral kernel of the \textit{unitary inversion operator} $\calU_\lambda$. The unitary inversion operator is essentially given by the action $\pi_\lambda(\widetilde{j})$ of the inversion element (see \cite[Section 3.3]{HKM12})
\begin{align*}
 \widetilde{j} &= \exp_{\widetilde{G}}\left(\frac{\pi}{2}(e,0,-e)\right)\in\widetilde{G}.
\end{align*}
More precisely, we put
\begin{align*}
 \calU_\lambda &:= e^{-i\pi\frac{r\lambda}{2}}\pi_\lambda(\widetilde{j}).
\end{align*}
The operator $\calU_\lambda$ is unitary on $L^2(\calO_\lambda,\td\mu_\lambda)$ of order $2$, i.e. $\calU_\lambda^2=\id$ (see \cite[Proposition 3.17~(1)]{HKM12}). We also study Whittaker vectors in the Schr\"odinger model. They can be derived from the explicit expression of the integral kernel of $\calU_\lambda$.

\subsection{The integral kernel of $\calU_\lambda$}

Recall the underlying $(\frakg,\frakk)$-module $W^\lambda=\calP(\calO_\lambda)e^{-\tr(x)}$ of the representation $(\pi_\lambda,L^2(\calO_\lambda,\td\mu_\lambda))$.

\begin{proposition}\label{prop:ConvergenceUnitInvOp}
For every $\psi\in W^\lambda$ the integral
\begin{align*}
 \calT_\lambda\psi(z) &:= 2^{-r\lambda}\int_{\calO_\lambda}{\calJ_\lambda(z,x)\psi(x)\td\mu_\lambda(x)}, & z\in V_\CC,
\end{align*}
converges uniformly on compact subsets of $V_\CC$. This defines a linear operator $\calT_\lambda:W^\lambda\to C(\calO_\lambda)$.
\end{proposition}

\begin{proof}
Let $\psi(x)=p(x)e^{-\tr(x)}\in\calP(\calO_\lambda)e^{-\tr(x)}=W^\lambda$. Then $|p(x)|\leq C_1(1+|x|)^N$ for some constants $C_1,N>0$. Further, by Lemma \ref{lem:JBesselConvergence} we have
\begin{align*}
 |\calJ_\lambda(z,x)| &\leq C_2(1+|z|\cdot|x|)^{\frac{r(2n-1)}{4}}e^{2r\sqrt{|z|\cdot|x|}} & \forall\,x\in\calO_\lambda,z\in V_\CC
\end{align*}
for a constant $C_2>0$. Hence, for $|z|\leq R$, $R\geq1$, we obtain
\begin{align*}
 & \int_{\calO_\lambda}{|\calJ_\lambda(z,x)\psi(x)|\td\mu_\lambda(x)}\\
 \leq{}& C_1C_2\int_{\calO_\lambda}{(1+R|x|)^{N+\frac{r(2n-1)}{4}}e^{2r\sqrt{R}\sqrt{|x|}-|x|}\td\mu_\lambda(x)} < \infty
\end{align*}
and the proof is complete.
\end{proof}

\begin{proposition}
The operator $\calT_\lambda$ is a unitary isomorphism $\calT_\lambda:W^\lambda\to W^\lambda$ with $\calT_\lambda\psi_0=\psi_0$ which intertwines the $\frakg$-action $\td\pi_\lambda$ with the $\frakg$-action $\td\pi_\lambda\circ\Ad(j)$.
\end{proposition}

\begin{proof}
By Proposition \ref{prop:DiffEqJBessel} the operator $\calT_\lambda$ intertwines the Bessel operator $\calB_\lambda$ with the coordinate multiplication $-x$ and vice versa. Therefore we obtain
\begin{align}
 \calU_\lambda\circ\td\pi_\lambda(X) &= \td\pi_\lambda(\Ad(j)X)\circ\calU_\lambda\label{eq:IntertwiningFormulaInversionOperator}
\end{align}
for $X\in\frakn\oplus\overline{\frakn}$. Since $\frakn$ and $\overline{\frakn}$ together generate $\frakg$ as a Lie algebra the intertwining formula \eqref{eq:IntertwiningFormulaInversionOperator} holds for all $X\in\frakg$. Further, for $x\in\calO_\lambda\subseteq\calX_\lambda$ we find
\begin{align*}
 \calT_\lambda\psi_0(x) &= 2^{-r\lambda}\int_{\calO_\lambda}{\calJ_\lambda(x,y)e^{-\tr(y)}\td\mu_\lambda(y)}\\
 &= \int_{\calO_\lambda}{\calI_\lambda(-2x,y)e^{-2\tr(y)}\td\mu(y)}\\
 &= e^{-\tr(x)}\BB_\lambda\psi_0(-2x) = e^{-\tr(x)} = \psi_0(x).
\end{align*}
Since $W^\lambda=\td\pi_\lambda(\calU(\frakg))\psi_0$, it follows that $\calT_\lambda W^\lambda\subseteq\td\pi_\lambda(\calU(\frakg))\calT_\lambda\psi_0=W^\lambda$. Now, since invariant Hermitian forms on the irreducible infinitesimally unitary $(\frakg,\frakk)$-module $W^\lambda$ are unique up to a scalar, we find that $\calT_\lambda$ is a unitary isomorphism.
\end{proof}

\begin{theorem}\label{thm:IntegralKernelInversionOperator}
$\calU_\lambda=\calT_\lambda$.
\end{theorem}

\begin{proof}
By the previous proposition $\calU_\lambda\circ\calT_\lambda^{-1}$ extends to a unitary isomorphism $L^2(\calO_\lambda,\td\mu_\lambda)\to L^2(\calO_\lambda,\td\mu_\lambda)$ which commutes with the $\frakg$-action $\td\pi_\lambda$. Therefore, by Schur's Lemma, $\calU_\lambda$ is a scalar multiple of $\calT_\lambda$. Since $\calU_\lambda\psi_0=\psi_0=\calT_\lambda\psi_0$ this gives the claim.
\end{proof}

\begin{remark}
The group $G$ is generated by the conformal inversion $j$ and the maximal parabolic subgroup $P:=L^\circ\ltimes N$, where
\begin{align*}
 L^\circ &= G\cap\Str(V).
\end{align*}
Write $\widetilde{P}$ for the preimage of $P$ under the covering map $\widetilde{G}\to G$. The restriction of $\pi_\lambda$ to $\widetilde{P}$ factors to $P$ and is in the Schr\"odinger model quite simple. In fact
\begin{align*}
 \pi_\lambda(n_a)\psi(x) &= e^{i(x|a)}\psi(x), & n_a\in N,\\
 \pi_\lambda(g)\psi(x) &= \chi(g^\#)^{\frac{\lambda}{2}}\psi(g^\# x), & g\in L^\circ,
\end{align*}
and by Mackey theory this representation is even irreducible on $L^2(\calO_\lambda,\td\mu_\lambda)$. Therefore, together with the action of $\widetilde{j}$ in the Schr\"odinger model (see Theorem \ref{thm:IntegralKernelInversionOperator}) this describes the complete group action $\pi_\lambda$.\\
For the case of the rank $1$ orbit $\calO_1$ Theorem \ref{thm:IntegralKernelInversionOperator} was shown in \cite[Theorem 4.3]{HKMO12}. Earlier contributions to special cases are \cite{KM07a} for the case $\frakg=\so(2,k)$ and $\lambda$ the minimal discrete Wallach point and \cite{Kos00} for the case $\frakg=\sl(2,\RR)$ and arbitrary $\lambda\in\calW$.
\end{remark}

\begin{remark}
Since the functions $x\mapsto\calJ_\lambda(x,y)$, $y\in\calO_\lambda$, are eigenfunctions of the Bessel operator, the inversion formula for $\calU_\lambda$ gives an expansion of any function $\psi\in L^2(\calO_\lambda,\td\mu_\lambda)$ into eigenfunctions of the Bessel operator:
\begin{align*}
 \psi(x) &= 2^{-r\lambda}\int_{\calO_\lambda}{\calJ_\lambda(x,y)\,\calU_\lambda\psi(y)\td\mu_\lambda(y)}, & \psi\in L^2(\calO_\lambda,\td\mu_\lambda).
\end{align*}
\end{remark}

Compared to the action of $\widetilde{j}$ in the Schr\"odinger model its action in the Fock model is extremely simple. Define $(-1)^*$ on $\calF_\lambda$ by $(-1)^*F(z)=F(-z)$.

\begin{proposition}\label{prop:InversionOnFockSpace}
$\BB_\lambda\circ\calU_\lambda = (-1)^*\circ\BB_\lambda.$
\end{proposition}

\begin{proof}
We have $\td\rho_\lambda(t(e,0,-e))=\td\pi_\lambda^\CC(2it(0,\1,0))=2it(\partial_z+\frac{r\lambda}{2})$. Therefore, we obtain
\begin{align*}
 \rho_\lambda(e^{t(e,0,-e)})F(z) &= e^{ir\lambda t}F(e^{2it}z).
\end{align*}
For $t=\frac{\pi}{2}$ we obtain the action of $\widetilde{j}$ which is given by $e^{i\pi\frac{r\lambda}{2}}(-1)^*$.
\end{proof}

This now allows us to compute the action of the unitary inversion operator on the Laguerre functions $\ell_{\bf m}^\lambda$ introduced in Section \ref{sec:Polynomials}.

\begin{proposition}
$$ \calU_\lambda\ell_{\bf m}^\lambda = (-1)^{|{\bf m}|}\ell_{\bf m}^\lambda. $$
\end{proposition}

\begin{proof}
Since $\BB_\lambda\ell_{\bf m}^\lambda=\const\cdot\Phi_{\bf m}$ and $\Phi_{\bf m}(-z)=(-1)^{|{\bf m}|}\Phi_{\bf m}(z)$ the claim follows with Proposition \ref{prop:InversionOnFockSpace}.
\end{proof}

\subsection{Whittaker vectors}\label{sec:WhittakerVectors}

The integral kernel of the unitary inversion operator provides us with explicit Whittaker vectors for the representations $\pi_\lambda$ on $L^2(\calO_\lambda,\td\mu_\lambda)$ and hence with explicit embeddings into Whittaker models. For this let $L^2(\calO_\lambda,\td\mu_\lambda)^\infty$ denote the space of smooth vectors of the representation $\pi_\lambda$ endowed with the usual Frechet topology. Since $\pi_\lambda$ leaves $L^2(\calO_\lambda,\td\mu_\lambda)^\infty$ invariant the representation extends to the space $L^2(\calO_\lambda,\td\mu_\lambda)^{-\infty}=(L^2(\calO_\lambda,\td\mu_\lambda)^\infty)'$ of distribution vectors by
\begin{align*}
 \langle\pi_\lambda(g)u,\varphi\rangle &:= \langle u,\overline{\pi_\lambda(g^{-1})}\varphi\rangle, & g\in\widetilde{G},
\end{align*}
for $u\in L^2(\calO_\lambda,\td\mu_\lambda)^{-\infty}$ and $\varphi\in L^2(\calO_\lambda,\td\mu_\lambda)^\infty$. Here we use the complex conjugate $\overline{U}$ of an operator $U$ on $L^2(\calO_\lambda,\td\mu_\lambda)^\infty$ which is defined by $\overline{U}\varphi:=\overline{U\overline{\varphi}}$ for $\varphi\in L^2(\calO_\lambda,\td\mu_\lambda)^\infty$. Further recall that $W^\lambda=\calP(\calO_\lambda)e^{-\tr(x)}\subseteq L^2(\calO_\lambda,\td\mu_\lambda)$ is the underlying $(\frakg,\frakk)$-module of $\pi_\lambda$. It does not carry a representation of the group $\widetilde{G}$, but the representation $\td\pi_\lambda$ of the Lie algebra $\frakg$ restricts to $W^\lambda$. By duality $\td\pi_\lambda$ extends to the algebraic dual $(W^\lambda)^*$:
\begin{align*}
 \langle\td\pi_\lambda(X)u,\varphi\rangle &:= -\langle u,\overline{\td\pi_\lambda(X)}\varphi\rangle, & X\in\frakg,
\end{align*}
for $u\in(W^\lambda)^*$ and $\varphi\in W^\lambda$. We have the following inclusions:
\begin{align*}
 W^\lambda \subseteq L^2(\calO_\lambda,\td\mu_\lambda)^\infty \subseteq L^2(\calO_\lambda,\td\mu_\lambda) \subseteq L^2(\calO_\lambda,\td\mu_\lambda)^{-\infty} \subseteq (W^\lambda)^*.
\end{align*}
Note that the unitary inversion operator $\calU_\lambda$ leaves each of these spaces invariant.

\begin{definition}
A distribution $u\in (W^\lambda)^*$ is an \textit{algebraic $\frakn$-Whittaker vector of weight $\eta\in\frakn_\CC^*$} (resp. an \textit{algebraic $\overline{\frakn}$-Whittaker vector of weight $\eta\in\overline{\frakn}_\CC^*$}) if
\begin{align*}
 \td\pi_\lambda(X)u &= \eta(X)u, & \mbox{for all $X\in\frakn$ (resp. $X\in\overline{\frakn}$).}
\end{align*}
If moreover $u\in L^2(\calO_\lambda,\td\mu_\lambda)^{-\infty}$ we call $u$ a \textit{smooth Whittaker vector}.
\end{definition}

For $z\in V_\CC$ let $\eta_z\in\frakn_\CC^*$ and $\overline{\eta}_z\in\overline{\frakn}_\CC^*$ be defined by
\begin{align*}
 \eta_z(u,0,0) &:= i(u|z), & u\in\frakn_\CC,\\
 \overline{\eta}_z(0,0,v) &:= -i(v|z), & v\in\overline{\frakn}_\CC.
\end{align*}

Now fix $\lambda\in\calW$. In \cite[Theorem 1]{Mat87} it is shown that there can only be non-trivial algebraic $\frakn$-Whittaker vectors (resp. $\overline{\frakn}$-Whittaker vectors) of weight $\eta\in\frakn_\CC^*$ (resp. $\eta\in\overline{\frakn}_\CC^*$) for $\td\pi_\lambda$ if $\eta$ is contained in the associated variety of $\td\pi_\lambda$ in $\frakg_\CC^*$. The intersection of the associated variety of $\td\pi_\lambda$ with $\frakn_\CC^*$ (resp. $\overline{\frakn}_\CC^*$) is equal to $\overline{\calX_\lambda}$ (after identifying $\frakn_\CC^*$ and $\overline{\frakn}_\CC^*$ with $V_\CC$). Hence the existence of a non-trivial algebraic $\frakn$-Whittaker vector (resp. $\overline{\frakn}$-Whittaker vector) of weight $\eta_z$ (resp. $\overline{\eta}_z$) implies $z\in\overline{\calX_\lambda}$. We prove the converse of this statement:

\begin{proposition}\label{prop:AlgebraicWhittakerVectors}
Let $z\in\overline{\calX_\lambda}$.
\begin{enumerate}
\item The Dirac delta distribution $\delta_{\lambda,z}$ at $z$ is contained in $(W^\lambda)^*$ and defines an algebraic $\frakn$-Whittaker vector of weight $\eta_z$.
\item The distribution
\begin{align*}
 \phi_{\lambda,z}(x) &:= \calJ_\lambda(z,x)\td\mu_\lambda(x), & x\in\calO_\lambda,
\end{align*}
is contained in $(W^\lambda)^*$ and defines an algebraic $\overline{\frakn}$-Whittaker vector of weight $\overline{\eta}_z$.
\item The algebraic Whittaker vectors $\delta_{\lambda,z}$ and $\phi_{\lambda,z}$ are related by
\begin{align*}
 \phi_{\lambda,z} &= 2^{r\lambda}\calU_\lambda\delta_{\lambda,z}.
\end{align*}
\end{enumerate}
\end{proposition}

\begin{proof}
\begin{enumerate}
\item Since $W^\lambda=\calP(\calO_\lambda)e^{-\tr(x)}$ and $\calP(\calO_\lambda)\cong\CC[\overline{\calX_\lambda}]$ the point evaluation of every $\varphi\in W^\lambda$ at $z\in\overline{\calX_\lambda}$ is well-defined. Now $\td\pi_\lambda(u,0,0)=i(u|x)$ and the Whittaker property follows.
\item By Proposition \ref{prop:ConvergenceUnitInvOp} the function $\phi_{\lambda,z}$ belongs to $(W^\lambda)^*$. Since $\td\pi_\lambda(0,0,v)=i(v|\calB_\lambda)$ the Whittaker property follows from Proposition \ref{prop:DiffEqJBessel}.
\item By Theorem \ref{thm:IntegralKernelInversionOperator} we have for $\varphi\in W^\lambda$
\begin{align*}
 \langle\phi_{\lambda,z},\varphi\rangle &= \int_{\calO_\lambda}{\calJ_\lambda(z,x)\varphi(x)\td\mu_\lambda(x)} = 2^{r\lambda}\calU_\lambda\varphi(z)\\
 &= 2^{r\lambda}\langle\delta_{\lambda,z},\calU_\lambda\varphi\rangle = 2^{r\lambda}\langle\calU_\lambda\delta_{\lambda,z},\varphi\rangle
\end{align*}
since $\overline{\calU_\lambda}=\calU_\lambda$.\qedhere
\end{enumerate}
\end{proof}

An interesting question is to determine the set of smooth Whittaker vectors for $\pi_\lambda$. We find that the algebraic $\frakn$-Whittaker vector $\delta_{\lambda,z}$ (resp. the algebraic $\overline{\frakn}$-Whittaker vector $\phi_{\lambda,z}$) is smooth if $z\in\overline{\calO_\lambda}$. For this we first prove estimates for $\delta_{\lambda,z}$ and $\phi_{\lambda,z}$ using the estimate for $\calJ_\lambda(z,w)$ by Nakahama~\cite{Nak12} stated in Proposition~\ref{prop:JBesselSharpEstimate}.

\begin{lemma}
For each $\lambda\in\calW$ and $x\in\overline{\calO_\lambda}$ there exist elements $X,Y\in\calU(\frakg)$ such that for all $\varphi\in W^\lambda$ we have
\begin{align}
 |\langle\delta_{\lambda,x},\varphi\rangle| &\leq \|\td\pi_\lambda(X)\varphi\|_{L^2(\calO_\lambda,\td\mu_\lambda)},\label{eq:WhittakerEstimateDelta}\\
 |\langle\phi_{\lambda,x},\varphi\rangle| &\leq \|\td\pi_\lambda(Y)\varphi\|_{L^2(\calO_\lambda,\td\mu_\lambda)}.\label{eq:WhittakerEstimatePhi}
\end{align}
\end{lemma}

\begin{proof}
We first prove \eqref{eq:WhittakerEstimatePhi}. From Proposition~\ref{prop:JBesselSharpEstimate} it follows that there exists constants $C_1>0$ and $k\in\NN_0$ such that
\begin{align*}
 |\calJ_\lambda(x,y)| &\leq C_1(1+|y|^2)^k & \forall\,y\in\overline{\calO_\lambda}.
\end{align*}
Now there exists an $N\in\NN_0$ with
\begin{align*}
 C_2 := \|(1+|\cdot|^2)^{k-N}\|_{L^2(\calO_\lambda,\td\mu_\lambda)} < \infty.
\end{align*}
Hence we obtain for $\varphi\in W^\lambda$, using H\"{o}lder's inequality:
\begin{align*}
 |\langle\phi_{\lambda,x},\varphi\rangle| &\leq C_1\int_{\calO_\lambda}{(1+|y|^2)^k|\varphi(y)|\td\mu_\lambda(y)}\\
 &= C_1\int_{\calO_\lambda}{(1+|y|^2)^{k-N}\cdot|(1+|y|^2)^N\varphi(y)|\td\mu_\lambda(y)}\\
 &\leq C_1C_2\|(1+|\cdot|^2)^N\varphi\|_{L^2(\calO_\lambda,\td\mu_\lambda)}.
\end{align*}
Since $\td\pi_\lambda(\calU(\frakg))$ contains all polynomials there exists $X\in\calU(\frakg)$ such that
\begin{equation*}
 \td\pi_\lambda(X)\varphi(y) = C_1C_2(1+|y|^2)^N\varphi(y)
\end{equation*}
which shows \eqref{eq:WhittakerEstimatePhi}. Now by Proposition~\ref{prop:AlgebraicWhittakerVectors}~(3) we have
\begin{equation*}
 \delta_{\lambda,x} = 2^{-r\lambda}\calU_\lambda\phi_{\lambda,x}
\end{equation*}
and we obtain
\begin{align*}
 |\langle\delta_{\lambda,x},\varphi\rangle| &= 2^{-r\lambda}|\langle\calU_\lambda\phi_{\lambda,x},\varphi\rangle| = 2^{-r\lambda}|\langle\phi_{\lambda,x},\calU_\lambda\varphi\rangle|\\
 &\leq 2^{-r\lambda}\|\td\pi_\lambda(X)\calU_\lambda\varphi\|_{L^2(\calO_\lambda,\td\mu_\lambda)}\\
 &= 2^{-r\lambda}\|\td\pi_\lambda(X)\pi_\lambda(\widetilde{j})\varphi\|_{L^2(\calO_\lambda,\td\mu_\lambda)}\\
 &= 2^{-r\lambda}\|\pi(\widetilde{j})\td\pi_\lambda(\Ad(\widetilde{j}^{-1})X)\varphi\|_{L^2(\calO_\lambda,\td\mu_\lambda)}\\
 &= 2^{-r\lambda}\|\td\pi_\lambda(\Ad(\widetilde{j}^{-1})X)\varphi\|_{L^2(\calO_\lambda,\td\mu_\lambda)}.
\end{align*}
Hence \eqref{eq:WhittakerEstimateDelta} follows with $Y=2^{-r\lambda}\Ad(\widetilde{j}^{-1})X$.
\end{proof}

\begin{theorem}\label{thm:DeltaPhiDistributionVectors}
Let $\lambda\in\calW$. Then for every $x\in\overline{\calO_\lambda}$ the algebraic Whittaker vectors $\delta_{\lambda,x}$ and $\phi_{\lambda,x}$ are smooth, i.e. $\delta_{\lambda,x},\phi_{\lambda,x}\in L^2(\calO_\lambda,\td\mu_\lambda)^{-\infty}$.
\end{theorem}

\begin{proof}
Recall that for $X\in\calU(\frakg)$ the map
\begin{equation*}
 L^2(\calO_\lambda,\td\mu_\lambda)\to[0,\infty),\,\varphi\mapsto\|\td\pi_\lambda(X)\varphi\|_{L^2(\calO_\lambda,\td\mu_\lambda)}
\end{equation*}
is a continuous seminorm on $L^2(\calO_\lambda,\td\mu_\lambda)^\infty$. Since $W^\lambda\subseteq L^2(\calO_\lambda,\td\mu_\lambda)^\infty$ is dense the claim now follows from \eqref{eq:WhittakerEstimateDelta} and \eqref{eq:WhittakerEstimatePhi}
\end{proof}

\begin{remark}
For $x\in\calO_\lambda$ it is easy to see that $\delta_{\lambda,x},\phi_{\lambda,x}\in L^2(\calO_\lambda,\td\mu_\lambda)^{-\infty}$. In fact, by the Sobolev embedding theorem it follows that there is a continuous linear embedding $L^2(\calO_\lambda,\td\mu_\lambda)^\infty\hookrightarrow C^\infty(\calO_\lambda)$ and hence it is immediate that $\delta_{\lambda,x}\in L^2(\calO_\lambda,\td\mu_\lambda)^{-\infty}$. Further $\phi_{\lambda,x}=2^{r\lambda}\calU_\lambda\delta_{\lambda,x}\in L^2(\calO_\lambda,\td\mu_\lambda)^{-\infty}$ as $\calU_\lambda$ is an automorphism of $L^2(\calO_\lambda,\td\mu_\lambda)^{-\infty}$. However, it is a priori not clear that every function $\varphi\in L^2(\calO_\lambda,\td\mu_\lambda)^\infty$ extends to $\overline{\calO_\lambda}$ and that $\delta_{\lambda,x}$ is defined on $L^2(\calO_\lambda,\td\mu_\lambda)^\infty$ for $x\in\partial\calO_\lambda$. In \cite[Theorems 5.15, 5.18 \&\ 5.19]{Kos00} the statements of Theorem~\ref{thm:DeltaPhiDistributionVectors} are shown for the special case $V=\RR$, i.e. $\frakg\cong\sl(2,\RR)$.
\end{remark}

\begin{remark}
Note that since $\calJ_\lambda(0,y)=1$ for all $y\in V$ we have $\phi_{\lambda,0}=\1\td\mu_\lambda$, the constant function on $\calO_\lambda$ with value $1$, and hence
\begin{align*}
 \calU_\lambda\delta_{\lambda,0} &= 2^{r\lambda}\1\td\mu_\lambda.
\end{align*}
This formula resembles the fact that the Euclidean Fourier transform maps the Dirac delta distribution at the origin to a constant function.
\end{remark}

\newpage
\section{Application to branching problems}\label{sec:ApplicationBranching}

For $\frakg=\so(2,n)$ we consider the representation $\td\pi^{\so(2,n)}_\lambda$ belonging to the minimal non-zero discrete Wallach point $\lambda=\frac{d}{2}=\frac{n-2}{2}$. We consider the restriction of $\td\pi^{\so(2,n)}_\lambda$ to the symmetric subalgebra $\frakh=\so(2,m)\oplus\so(n-m)$ of $\frakg$. Note that the pair $(\frakg,\frakh)$ is a symmetric pair of holomorphic type and hence the restriction of the lowest weight representation $\td\pi^{\so(2,n)}_\lambda$ to $\frakh$ is discretely decomposable with multiplicity one and the representations of $\frakh$ appearing in the decomposition will again be lowest weight representations (see \cite{Kob08}).

Consider the Jordan algebra $V=\RR^{1,n-1}$ which is the vector space $\RR\times\RR^{n-1}$ endowed with the multiplication
\begin{align*}
 (x_1,x')\cdot(y_1,y') &:= (x_1y_1+x'\cdot y',x_1y'+y_1x')
\end{align*}
for $x_1,y_1\in\RR$, $x',y'\in\RR^{n-1}$, where $x'\cdot y'$ denotes the standard inner product on $\RR^{n-1}$. We have $\der(V)=\so(n-1)$, acting on the last $n-1$ variables, and $\str(V)=\so(1,n-1)\oplus\RR$. The conformal Lie algebra of $V$ is given by $\frakg=\so(2,n)$. The subalgebra $\so(2,m)$, $1\leq m\leq n$, can be viewed as the conformal Lie algebra of the Jordan subalgebra $U=\{x\in V:x_{m+1}=\ldots=x_n=0\}\cong\RR^{1,m-1}$. Then $\so(2,m)$ consists of the elements $(u,L(a)+D,v)$ with $u,a,v\in U$ and $D\in\so(m-1)$ acting on the coordinates $x_2,\ldots,x_m$. The centralizer of $\so(2,m)$ in $\frakg$ is $\so(n-m)$ acting linearly on the coordinates $x_{m+1},\ldots,x_n$. Hence $\so(n-m)\subseteq\frakk^\frakl=\so(n-1)$. Finally let $\frakh=\so(2,m)\oplus\so(n-m)\subseteq\frakg$ be the corresponding subalgebra of $\frakg$.

To study the restriction $\td\pi^{\so(2,n)}_\lambda|_{\frakh}$ we use the Fock model $(\td\rho^{\so(2,n)}_\lambda,\calF_\lambda)$ of the representation $\td\pi^{\so(2,n)}_\lambda$. In this case
\begin{align*}
 \calX^{\so(2,n)}_1 &= \{z\in\CC^n:z_1^2-z_2^2-\cdots-z_n^2=0\}\setminus\{0\},\\
 \calX^{\so(2,n)}_2 &= \{z\in\CC^n:z_1^2-z_2^2-\cdots-z_n^2\neq0\},
\end{align*}
realized in the ambient space $V_\CC=\CC^n$, and hence
\begin{align*}
 \calP(\calX_1^{\so(2,n)}) &= \CC[Z_1,\ldots,Z_n]/(Z_1^2-Z_2^2-\cdots-Z_n^2),\\
 \calP(\calX_2^{\so(2,n)}) &= \CC[Z_1,\ldots,Z_n].
\end{align*}
For arbitrary $\lambda\in\calW$ the representation $\td\rho^{\so(2,n)}_\lambda$ of $\frakg_\CC$ acts by
\begin{align*}
 \td\rho_\lambda(a,-2iL(a),a) &= 2i\sum_{j=1}^n{a_jz_j}, & \td\rho_\lambda(a,0,-a) &= 2i\left(\partial_{az}+\frac{\lambda}{2}\tr(a)\right),\\
 \td\rho_\lambda(a,2iL(a),a) &= -2i\sum_{j=1}^n{a_jB_\lambda^{n,j}}, & \td\rho_\lambda(0,D,0) &= -\partial_{Dz},
\end{align*}
where
\begin{align*}
 B_\lambda^{n,j}&= \varepsilon_jz_j\Box^n-2(E^n+\lambda)\frac{\partial}{\partial z_j}, & \Box^n &= \sum_{j=1}^n{\varepsilon_j\frac{\partial^2}{\partial z_j^2}},\\
 E^n &= \sum_{j=1}^n{z_j\frac{\partial}{\partial z_j}}, & \varepsilon_j &= \begin{cases}+1 & \mbox{for $j=1$,}\\-1 & \mbox{for $2\leq j\leq n$.}\end{cases}
\end{align*}

We first consider the action of the compact part $\so(n-m)$ which acts naturally on the coordinates $z_{m+1},\ldots,z_n$. Note that
\begin{align*}
 \CC[Z_1,\ldots,Z_n] = \CC[Z_1,\ldots,Z_m]\otimes\CC[Z_{m+1},\ldots,Z_n].
\end{align*}
Every polynomial in $Z_{m+1},\ldots,Z_n$ can be written as a sum of spherical harmonics multiplied with powers of $(Z_{m+1}^2+\cdots+Z_n^2)$. In $\calP(\calX_1^{\so(2,n)})$ we have $Z_{m+1}^2+\cdots+Z_n^2=Z_1^2-Z_2^2-\cdots-Z_m^2$ and hence
\begin{align*}
 \calP(\calX_1^{\so(2,n)}) &\cong \CC[Z_1,\ldots,Z_n]/(Z_1^2-Z_2^2-\cdots-Z_n^2)\\
 &\cong \bigoplus_{k=0}^\infty{\CC[Z_1,\ldots,Z_m]\otimes\calH^k(\CC^{n-m})},
\end{align*}
which gives the decomposition into irreducible $\so(n-m)$-representations.

Next we have to examine the action of $\so(2,m)$ on each of the summands. For this we use the notation $V_\CC=\CC^n=\CC^m\oplus\CC^{n-m}=U_\CC\oplus U^\perp_\CC$.

\begin{lemma}
For $k\in\NN_0$ and $X\in\so(2,m)$ we have
\begin{align*}
 \td\rho_\lambda^{\so(2,n)}(X)|_{\CC[Z_1,\ldots,Z_m]\otimes\calH^k(\CC^{n-m})} &= \td\rho_{\lambda+k}^{\so(2,m)}(X)\otimes\id.
\end{align*}
\end{lemma}

\begin{proof}
On $\CC[Z_1,\ldots,Z_m]\otimes\calH^k(\CC^{n-m})$ we check four exhaustive cases separately:
\begin{enumerate}
\item $X=(0,D,0)$, $D\in\so(m-1)$. The action is given by
\begin{align*}
 \td\rho_\lambda^{\so(2,n)}(X) &= -\partial_{Dz} = \td\rho_{\lambda+k}^{\so(2,m)}(X)\otimes\id
\end{align*}
since $D$ only acts on the coordinates $z_2,\ldots,z_m$.
\item $X=(a,-2iL(a),a)$, $a\in U$. Since $a\in U\subseteq\CC^m$ it is immediate that
\begin{align*}
 \td\rho_\lambda^{\so(2,n)}(X) &= 2i\sum_{j=1}^m{a_jz_j} = \td\rho_{\lambda+k}^{\so(2,m)}(X)\otimes\id.
\end{align*}
\item $X=(a,0,-a)$, $a\in U$. Assume first that $a_1=0$. Then $L(a)$ annihilates $\CC^{n-m}$ and hence
\begin{align*}
 \td\rho_\lambda^{\so(2,n)}(X) &= 2i\partial_{az} = \td\rho_{\lambda+k}^{\so(2,m)}(X)\otimes\id.
\end{align*}
For $a=e$ we consider $X=(e,0,-e)$ which acts by
\begin{align*}
 \td\rho_\lambda^{\so(2,n)}(X) &= 2i(E^n+\lambda) = 2i(E^m+\lambda+k) = \td\rho_{\lambda+k}^{\so(2,m)}(X)\otimes\id
\end{align*}
since $E^{n-m}=\sum_{j=m+1}^n{z_j\frac{\partial}{\partial z_j}}$ acts on $\calH^k(\CC^{n-m})$ by the scalar $k$.
\item $X=(a,2iL(a),a)$, $a\in U$. It suffices to check the claimed formula for $a=e_j$, $j=1,\ldots,m$. We then have
\begin{align*}
 \td\rho_\lambda^{\so(2,n)}(X) &= -2i\left(\varepsilon_jz_j\Box^n-2(E^n+\lambda)\frac{\partial}{\partial z_j}\right)\\
 &= -2i\left(\varepsilon_jz_j\Box^m-2(E^m+\lambda+k)\frac{\partial}{\partial z_j}\right)\\
 &= \td\rho_{\lambda+k}^{\so(2,m)}(X)\otimes\id
\end{align*}
since
\begin{multline*}
 \Box^n[f(z_1,\ldots,z_m)g(z_{m+1},\ldots,z_n)]\\
 = \Box^mf(z_1,\ldots,z_m)\cdot g(z_{m+1},\ldots,z_n)\\
 + f(z_1,\ldots,z_m)\cdot\Delta^{n-m}g(z_{m+1},\ldots,z_n),
\end{multline*}
where $\Delta^{n-m}$ denotes the Laplacian on $\CC^{n-m}$ and $E^{n-m}$ the Euler operator on $\CC^{n-m}$ which acts on $\calH^k(\CC^{n-m})$ by the scalar $k$.\qedhere
\end{enumerate}
\end{proof}

This shows the following theorem:

\begin{theorem}\label{thm:BranchingSO(n,2)}
For $\lambda=\frac{n-2}{2}$ the minimal non-zero discrete Wallach point for $\so(2,n)$ we have
\begin{align*}
 \td\pi_\lambda^{\so(2,n)} &= \bigoplus_{k=0}^\infty{\td\pi_{\lambda+k}^{\so(2,m)}\boxtimes\calH^k(\RR^{n-m})}.
\end{align*}
\end{theorem}

For $m<n$ we have $\lambda+k=\frac{n+2k-2}{2}>\frac{m-2}{2}$ for $k\in\NN_0$ and hence the representations $\td\pi_{\lambda+k}^{\so(2,m)}$ belong to Wallach points in the continuous part of the Wallach set. For $2k>2m-n$ the representation $\td\pi_{\lambda+k}^{\so(2,m)}$ are holomorphic discrete series. In particular, if $2m\geq n$ then there occur representations in the branching law which are not holomorphic discrete series.

\newpage

\bibliographystyle{amsplain}
\bibliography{bibdb}

\def\cprime{$'$}
\providecommand{\bysame}{\leavevmode\hbox to3em{\hrulefill}\thinspace}
\providecommand{\MR}{\relax\ifhmode\unskip\space\fi MR }
\providecommand{\MRhref}[2]{%
  \href{http://www.ams.org/mathscinet-getitem?mr=#1}{#2}
}
\providecommand{\href}[2]{#2}
\begin{thebibliography}{10}

\bibitem{Ach12}
D.~Achab, \emph{Minimal representations of simple real {L}ie groups of
  {H}ermitian type - {T}he {F}ock model},  (2012), preprint, available at
  \href{http://arxiv.org/abs/1206.1737}{arXiv:1206.1737}.

\bibitem{AF11}
D.~Achab and J.~Faraut, \emph{Analysis of the {B}rylinski-{K}ostant model for
  spherical minimal representations}, Canad. J. Math. \textbf{64} (2012),
  no.~4, 721--754.

\bibitem{AAR99}
G.~E. Andrews, R.~Askey, and R.~Roy, \emph{Special functions}, Encyclopedia of
  Mathematics and its Applications, vol.~71, Cambridge University Press,
  Cambridge, 1999.

\bibitem{ADO06}
M.~Aristidou, M.~Davidson, and G.~{\'O}lafsson, \emph{Differential recursion
  relations for {L}aguerre functions on symmetric cones}, Bull. Sci. Math.
  \textbf{130} (2006), no.~3, 246--263.

\bibitem{BSO06}
S.~Ben~Sa{\"{\i}}d and B.~{\O}rsted, \emph{Segal--{B}argmann transforms
  associated with finite {C}oxeter groups}, Math. Ann. \textbf{334} (2006),
  no.~2, 281--323.

\bibitem{Ber75}
F.~A. Berezin, \emph{Quantization in complex symmetric spaces}, Izv. Akad. Nauk
  SSSR Ser. Mat. \textbf{39} (1975), no.~2, 363--402, 472.

\bibitem{BK94}
R.~Brylinski and B.~Kostant, \emph{Minimal representations, geometric
  quantization, and unitarity}, Proc. Nat. Acad. Sci. U.S.A. \textbf{91}
  (1994), no.~13, 6026--6029.

\bibitem{Cle88}
J.-L. Clerc, \emph{Fonctions {$K$} de {B}essel pour les alg\`ebres de
  {J}ordan}, Harmonic analysis ({L}uxembourg, 1987), Lecture Notes in Math.,
  vol. 1359, Springer, Berlin, 1988, pp.~122--134.

\bibitem{DOZ03}
M.~Davidson, G.~{\'O}lafsson, and G.~Zhang, \emph{Laplace and
  {S}egal--{B}argmann transforms on {H}ermitian symmetric spaces and orthogonal
  polynomials}, J. Funct. Anal. \textbf{204} (2003), no.~1, 157--195.

\bibitem{Dib90}
H.~Dib, \emph{Fonctions de {B}essel sur une alg\`ebre de {J}ordan}, J. Math.
  Pures Appl. (9) \textbf{69} (1990), no.~4, 403--448.

\bibitem{DG93}
H.~Ding and K.~I. Gross, \emph{Operator-valued {B}essel functions on {J}ordan
  algebras}, J. Reine Angew. Math. \textbf{435} (1993), 157--196.

\bibitem{DGKR00}
H.~Ding, K.~I. Gross, R.~A. Kunze, and D.~St.~P. Richards, \emph{Bessel
  functions on boundary orbits and singular holomorphic representations}, The
  mathematical legacy of {H}arish-{C}handra ({B}altimore, {MD}, 1998), Proc.
  Sympos. Pure Math., vol.~68, Amer. Math. Soc., Providence, RI, 2000,
  pp.~223--254.

\bibitem{EHW83}
T.~Enright, R.~Howe, and N.~Wallach, \emph{A classification of unitary highest
  weight modules}, Representation theory of reductive groups ({P}ark {C}ity,
  {U}tah, 1982), Progr. Math., vol.~40, Birkh\"auser Boston, Boston, MA, 1983,
  pp.~97--143.

\bibitem{EJ90}
T.~Enright and A.~Joseph, \emph{An intrinsic analysis of unitarizable highest
  weight modules}, Math. Ann. \textbf{288} (1990), no.~4, 571--594.

\bibitem{FK94}
J.~Faraut and A.~Kor{\'a}nyi, \emph{Analysis on symmetric cones}, The Clarendon
  Press, Oxford University Press, New York, 1994.

\bibitem{FT87}
J.~Faraut and G.~Travaglini, \emph{Bessel functions associated with
  representations of formally real {J}ordan algebras}, J. Funct. Anal.
  \textbf{71} (1987), no.~1, 123--141.

\bibitem{GK98}
S.~Gindikin and S.~Kaneyuki, \emph{On the automorphism group of the generalized
  conformal structure of a symmetric {$R$}-space}, Differential Geom. Appl.
  \textbf{8} (1998), no.~1, 21--33.

\bibitem{GR65}
I.~S. Gradshteyn and I.~M. Ryzhik, \emph{Table of integrals, series, and
  products}, Academic Press, New York, 1965.

\bibitem{HJ82}
M.~Harris and H.~P. Jakobsen, \emph{Singular holomorphic representations and
  singular modular forms}, Math. Ann. \textbf{259} (1982), no.~2, 227--244.

\bibitem{Hel62}
S.~Helgason, \emph{Differential geometry and symmetric spaces}, Pure and
  Applied Mathematics, Vol. XII, Academic Press, New York, 1962.

\bibitem{Her55}
C.~S. Herz, \emph{Bessel functions of matrix argument}, Ann. of Math. (2)
  \textbf{61} (1955), 474--523.

\bibitem{HKM12}
J.~Hilgert, T.~Kobayashi, and J.~M{\" o}llers, \emph{Minimal representations
  via {B}essel operators},  (2012), to appear in J. Math. Soc. Japan, available
  at \href{http://arxiv.org/abs/1106.3621}{arXiv:1106.3621}.

\bibitem{HKMO12}
J.~Hilgert, T.~Kobayashi, J.~M{\"o}llers, and B.~{\O}rsted, \emph{Fock model
  and {S}egal--{B}argmann transform for minimal representations of {H}ermitian
  {L}ie groups}, J. Funct. Anal. \textbf{263} (2012), no.~11, 3492--3563.

\bibitem{HN01}
J.~Hilgert and K.-H. Neeb, \emph{Vector valued {R}iesz distributions on
  {E}uclidian {J}ordan algebras}, J. Geom. Anal. \textbf{11} (2001), no.~1,
  43--75.

\bibitem{HNO94}
J.~Hilgert, K.-H. Neeb, and B.~{\O}rsted, \emph{The geometry of nilpotent
  coadjoint orbits of convex type in {H}ermitian {L}ie algebras}, J. Lie Theory
  \textbf{4} (1994), no.~2, 185--235.

\bibitem{HNO96a}
\bysame, \emph{Conal {H}eisenberg algebras and associated {H}ilbert spaces}, J.
  Reine Angew. Math. \textbf{474} (1996), 67--112.

\bibitem{HNO96b}
\bysame, \emph{Unitary highest weight representations via the orbit method.
  {I}. {T}he scalar case}, Acta Appl. Math. \textbf{44} (1996), no.~1-2,
  151--184, Representations of Lie groups, Lie algebras and their quantum
  analogues.

\bibitem{HZ09}
J.~Hilgert and G.~Zhang, \emph{Segal--{B}argmann and {W}eyl transforms on
  compact {L}ie groups}, Monatsh. Math. \textbf{158} (2009), no.~3, 285--305.

\bibitem{Jak83}
H.~P. Jakobsen, \emph{Hermitian symmetric spaces and their unitary highest
  weight modules}, J. Funct. Anal. \textbf{52} (1983), no.~3, 385--412.

\bibitem{Jos92}
A.~Joseph, \emph{Annihilators and associated varieties of unitary highest
  weight modules}, Ann. Sci. \'Ecole Norm. Sup. (4) \textbf{25} (1992), no.~1,
  1--45.

\bibitem{Kan98}
S.~Kaneyuki, \emph{The {S}ylvester's law of inertia in simple graded {L}ie
  algebras}, J. Math. Soc. Japan \textbf{50} (1998), no.~3, 593--614.

\bibitem{Kob08}
T.~Kobayashi, \emph{Multiplicity-free theorems of the restrictions of unitary
  highest weight modules with respect to reductive symmetric pairs},
  Representation theory and automorphic forms, Progr. Math., vol. 255,
  Birkh\"auser Boston, Boston, MA, 2008, pp.~45--109.

\bibitem{KM07a}
T.~Kobayashi and G.~Mano, \emph{The inversion formula and holomorphic extension
  of the minimal representation of the conformal group}, Harmonic analysis,
  group representations, automorphic forms and invariant theory, Lect. Notes
  Ser. Inst. Math. Sci. Natl. Univ. Singap., vol.~12, World Sci. Publ.,
  Hackensack, NJ, 2007, pp.~151--208.

\bibitem{KMa11}
\bysame, \emph{The {S}chr\"odinger model for the minimal representation of the
  indefinite orthogonal group {${\rm O}(p,q)$}}, Memoirs of the AMS
  \textbf{213} (2011), no.~1000.

\bibitem{KO03b}
T.~Kobayashi and B.~{\O}rsted, \emph{Analysis on the minimal representation of
  {$\textup{O}(p,q)$}. {II}. {B}ranching laws}, Adv. Math. \textbf{180} (2003),
  no.~2, 513--550.

\bibitem{KO12}
T.~Kobayashi and Y.~Oshima, \emph{Classification of symmetric pairs with
  discretely decomposable restrictions of $(\frakg,k)$-modules},  (2012),
  preprint, available at
  \href{http://arxiv.org/abs/1202.5743}{arXiv:1202.5743}.

\bibitem{Kos00}
B.~Kostant, \emph{On {L}aguerre polynomials, {B}essel functions, {H}ankel
  transform and a series in the unitary dual of the simply-connected covering
  group of {${\rm Sl}(2,{\bf R})$}}, Represent. Theory \textbf{4} (2000),
  181--224 (electronic).

\bibitem{Las87}
M.~Lassalle, \emph{Alg\`ebre de {J}ordan et ensemble de {W}allach}, Invent.
  Math. \textbf{89} (1987), no.~2, 375--393.

\bibitem{LSS88}
T.~Levasseur, S.~P. Smith, and J.~T. Stafford, \emph{The minimal nilpotent
  orbit, the {J}oseph ideal, and differential operators}, J. Algebra
  \textbf{116} (1988), no.~2, 480--501.

\bibitem{LS89}
T.~Levasseur and J.~T. Stafford, \emph{Rings of differential operators on
  classical rings of invariants}, Mem. Amer. Math. Soc. \textbf{81} (1989),
  no.~412.

\bibitem{Mat87}
H.~Matumoto, \emph{Whittaker vectors and associated varieties}, Invent. Math.
  \textbf{89} (1987), no.~1, 219--224.

\bibitem{Moe10}
J.~M{\"{o}}llers, \emph{Minimal representations of conformal groups and
  generalized {L}aguerre functions}, Ph.D. thesis, University of Paderborn,
  2010, available at \href{http://arxiv.org/abs/1009.4549}{arXiv:1009.4549}.

\bibitem{MO12}
J.~M{\" o}llers and Y.~Oshima, \emph{Branching laws for minimal holomorphic
  representations}, in preparation.

\bibitem{MS12c}
J.~M{\" o}llers and B.~Schwarz, \emph{Bessel operators on {J}ordan pairs and
  {$L^2$}-models for small representations}, in preparation.

\bibitem{Nak12}
R.~Nakahama, \emph{Integral formula and upper estimate of {I}- and {J}-{B}essel
  functions on {J}ordan algebras},  (2012), preprint, available at
  \href{http://arxiv.org/abs/1211.4702}{arXiv:1211.4702}.

\bibitem{Nee00}
K.-H. Neeb, \emph{Holomorphy and convexity in {L}ie theory}, de Gruyter
  Expositions in Mathematics, vol.~28, Walter de Gruyter \& Co., Berlin, 2000.

\bibitem{OO96}
G.~{\'O}lafsson and B.~{\O}rsted, \emph{Generalizations of the {B}argmann
  transform}, Lie theory and its applications in physics ({C}lausthal, 1995),
  World Sci. Publ., River Edge, NJ, 1996, pp.~3--14.

\bibitem{Sek87}
J.~Sekiguchi, \emph{Remarks on real nilpotent orbits of a symmetric pair}, J.
  Math. Soc. Japan \textbf{39} (1987), no.~1, 127--138.

\bibitem{RV76}
M.~Vergne and H.~Rossi, \emph{Analytic continuation of the holomorphic discrete
  series of a semi-simple {L}ie group}, Acta Math. \textbf{136} (1976),
  no.~1-2, 1--59.

\bibitem{Vog78}
D.~A. Vogan, Jr., \emph{Gel\cprime fand-{K}irillov dimension for
  {H}arish-{C}handra modules}, Invent. Math. \textbf{48} (1978), no.~1, 75--98.

\bibitem{Vog91}
\bysame, \emph{Associated varieties and unipotent representations}, Harmonic
  analysis on reductive groups ({B}runswick, {ME}, 1989), Progr. Math., vol.
  101, Birkh\"auser Boston, Boston, MA, 1991, pp.~315--388.

\bibitem{Wal79}
N.~R. Wallach, \emph{The analytic continuation of the discrete series. {I},
  {II}}, Trans. Amer. Math. Soc. \textbf{251} (1979), 1--17, 19--37.

\end{thebibliography}

\vspace{30pt}

\textsc{Jan M\"ollers\\Institut for Matematiske Fag, Aarhus Universitet, Ny Munkegade 118, 8000 Aarhus C, Danmark.}\\
\textit{E-mail address:} \texttt{moellers@imf.au.dk}

\end{document}